\documentclass[oneside,11pt,reqno]{amsart}
\usepackage{stmaryrd}
\usepackage{bbm}
\usepackage{tikz}
\usepackage{mathrsfs}
\usepackage{enumerate}
\usepackage{latexsym,amsxtra}
\usepackage{color}
\usepackage{amscd}
\usepackage{amsmath,amsfonts,amsthm,amssymb}
\usepackage{latexsym,amsmath}
\usepackage{graphicx,psfrag}
\usepackage{mathabx}
\usepackage[utf8]{inputenc}
\usepackage{hyperref}

\usepackage{xy}
\xyoption{all}
\usepackage{pinlabel}

\usepackage[marginparwidth=.8in,marginparsep=0.15in]{geometry}

\newtheorem{thm}{Theorem}[section]

\newtheorem{assum}{Assumption}
\newtheorem{lem}[thm]{Lemma}
\newtheorem{cor}[thm]{Corollary}
\newtheorem{pro}[thm]{Proposition}

\newtheorem{introthm}{Theorem}
\newtheorem{introcor}[introthm]{Corollary}

\theoremstyle{definition}
\newtheorem{defi}[thm]{Definition}
\newtheorem{ex}[thm]{Example}
\newtheorem*{introex}{Example}

\newtheorem{rmk}[thm]{Remark}

\newcommand{\Fib}{\operatorname{Fib}}
\newcommand{\ind}{\operatorname{ind}}
\newcommand{\Lef}{\operatorname{Lef}}
\newcommand{\arf}{\operatorname{Arf}}
\newcommand{\Pin}{\operatorname{Pin}}
\newcommand{\red}{\operatorname{red}}
\newcommand{\hmred}{HM^{\red}}
\renewcommand{\top}{\operatorname{top}}
\newcommand{\sign}{\operatorname{sign}}

\newcommand{\grad}{\operatorname{grad}}
\newcommand{\coker}{\operatorname{coker}}
\newcommand{\im}{\operatorname{im}}
\newcommand{\gr}{\operatorname{gr}}
\newcommand{\B}{\mathcal B}
\newcommand{\C}{\mathcal C}
\newcommand{\D}{\mathcal D}
\newcommand{\M}{\mathcal M}
\renewcommand{\L}{\mathcal L}
\newcommand{\E}{\mathcal E}
\newcommand{\K}{\mathcal K}
\newcommand{\T}{\mathcal T}
\renewcommand{\S}{\mathcal S}
\newcommand{\Q}{\mathbb Q}
\newcommand{\R}{\mathbb R}
\newcommand{\Z}{\mathbb Z}
\newcommand{\s}{\mathfrak s}
\newcommand{\p}{\partial}
\newcommand{\ep}{\varepsilon}
\newcommand{\zp}{Z_{\infty}}
\renewcommand{\wp}{W_{\infty}}
\newcommand{\lsw}{\lambda_{\rm{SW}}}
\newcommand{\supp}{\operatorname{supp}}
\newcommand{\spin}{\,\operatorname{spin}}
\newcommand{\spinc}{\,\operatorname{spin}^c}

\newcommand{\Tr}{\operatorname{Tr}}
\newcommand{\Id}{\operatorname{Id}}
\renewcommand{\phi}{\varphi}

\newcommand{\loc}{\operatorname{loc}}
\newcommand{\Tor}{\operatorname{Tor}}

\title{A splitting theorem for the Seiberg-Witten invariant of a homology $S^1 \times S^3$}
\thanks{The second author was partially supported by NSF grant DMS-1506328, and the third author was partially supported by a Collaboration Grant from the Simons Foundation}
\author[Jianfeng Lin]{Jianfeng Lin}
\address{Department of Mathematics\newline\indent Massachusetts Institute of
Technology \newline\indent Cambridge MA 02139}
\email{\rm{linjian5477@gmail.com}}
\author[Daniel Ruberman]{Daniel Ruberman}
\address{Department of Mathematics, MS 050\newline\indent Brandeis
University \newline\indent Waltham, MA 02454}
\email{\rm{ruberman@brandeis.edu}}
\author[Nikolai Saveliev]{Nikolai Saveliev}
\address{Department of Mathematics\newline\indent
University of Miami \newline\indent PO Box 249085
\newline\indent Coral Gables, FL 33124}
\email{\rm{saveliev@math.miami.edu}}
\subjclass[2010]{57R57, 57R58, 57M27, 53C21, 58J28, 58J35}

\begin{document}
\begin{abstract}
We study the Seiberg--Witten invariant $\lsw(X)$ of smooth spin $4$-manifolds $X$ with rational homology of $S^1\times S^3$ defined by Mrowka, Ruberman, and Saveliev as a signed count of irreducible monopoles amended by an index-theoretic correction term. We prove a splitting formula for this invariant in terms of the Fr{\o}yshov invariant $h(X)$ and a certain Lefschetz number in the reduced monopole Floer homology of Kronheimer and Mrowka.  We apply this formula to obstruct existence of metrics of positive scalar curvature on certain 4-manifolds, and to exhibit new classes of homology $3$-spheres of infinite order in the homology cobordism group. 
\end{abstract}

%%%%%%%%%%%%%%%%%%%%%%%%%%%%%%%%%%%%%%%%%%%%

\maketitle

\section{Introduction}
Let $X$ be a smooth oriented spin 4-manifold with the rational homology of $S^1\,\times\,S^3$. Such manifolds play an important role in the study of homology cobordisms of homology 3-spheres and in addressing certain classification problems in 4-dimensional topology; see discussion in \cite{RS1}. Their study, however, represents a challenge because the usual count of the Seiberg--Witten monopoles on $X$  generally depends on the auxiliary choices of metric and perturbation and hence does not result in a smooth invariant of $X$. This problem has been remedied by Mrowka, Ruberman, and Saveliev \cite{MRS1}, who defined an invariant $\lsw (X)$ for integral homology $S^1 \times S^3$ by adding an index-theoretic correction term to the count of the Seiberg--Witten monopoles on $X$; we will extend their definition to the case of rational homology $S^1 \times S^3$ in this paper. Because of its gauge theoretic nature, the invariant $\lsw(X)$ is difficult to compute directly from its definition. We address this problem in the current paper by expressing $\lsw(X)$ in terms of Floer theoretic invariants via a gluing theory, the way it is done for the classical Seiberg--Witten invariants.

An invariant relevant to this gluing theory was defined by Fr{\o}yshov \cite{F}: under the additional  hypothesis that a generator of $H_3 (X;\Z)$ is carried by an embedded rational homology sphere $Y$, he showed \cite[Theorem 8]{F} that the invariant $h(Y,\s)$ arising from the monopole Floer homology of $Y$ with the induced spin structure $\s$ is an invariant of the spin manifold $X$ alone; we will denote this invariant by $h(X)$. The invariants $\lsw(X)$ and $h(X)$ are certainly different: for instance, the mod $2$ reduction of $\lsw (X)$ equals the Rohlin invariant of $X$, while this is not the case for $h(X)$. 

The following theorem establishes a precise relation between $\lsw(X)$ and $h(X)$. It is followed by some strong applications (Theorems \ref{T:psc}, \ref{T:Theta}, and \ref{T:Theta2} below) to the study of metrics of positive scalar curvature and of the homology cobordism group of homology 3-spheres. 

\begin{introthm}\label{T:main}
Let $X$ be a smooth oriented spin rational homology $S^1 \times S^3$ which is homology oriented by a choice of generator in $H^1 (X;\mathbb Z)$. Assume that the Poincar\'e dual of this generator is realized by a rational homology sphere $Y \subset X$. Let $\s$ be the induced spin structure on $Y$, and let $W$ be the spin cobordism from $Y$ to itself obtained by cutting $X$ open along $Y$. Then
\begin{equation}\label{E:main}
\lsw(X)\, +\, h(X)\; =\;- \Lef\, (W_*: \hmred(Y,\s)\to \hmred(Y,\s)).
\end{equation}
\end{introthm}

The reduced monopole Floer homology $\hmred(Y,\s)$ that appears in the statement of this theorem is the Floer theory defined by Kronheimer and Mrowka \cite{KM}, with rational coefficients. We could use instead the monopole Floer homology defined by Fr{\o}yshov \cite{F}, which would change the sign of the Lefschetz number because of different grading conventions in the two theories. 
%It follows in particular that the Lefschetz number in \eqref{E:main} is an invariant of $X$ alone; this was previously observed by Fr{\o}yshov \cite[Theorem 8]{F}. 
The rational homology sphere $Y$ in Theorem \ref{T:main} is oriented by the rule that the orientation on a curve Hom-dual to $\alpha$ followed by the orientation on $Y$ gives the orientation on $X$. The precise sign convention for $\lsw (X)$ is described in Section \ref{S:lsw}. %One should also mention that a specific choice of spin structure on $X$ is immaterial because any two such choices are equivalent as $\spinc$ structures and therefore lead to the same invariants.

In the special case of $X = S^1 \times Y$, where $Y$ is an integral homology sphere, Theorem \ref{T:main} reduces to a theorem of Fr{\o}yshov \cite[Theorem 5]{F} which relates the Casson invariant $\lambda(Y)$, the Fr{\o}yshov invariant $h(Y)$, and the Euler characteristic of $\hmred(Y)$. An analogous theorem \cite[Theorem 7]{F} holds for manifolds $X$ with $b^2_+(X) > 1$ once the term $\lsw (X) + h(X)$ in formula (\ref{E:main}) is replaced by the usual Seiberg--Witten invariant. %It should be pointed out that our work is similar in some respects to \cite{F} but it has several key new features, which we discuss next.

%%%%%%%%%%%%%%%%%%%%%%%%%%%%%%%%%%%%%%%%%%%%%%%

\subsection{An outline of the proof}
The proof of Theorem \ref{T:main} relies on the calculation of the two terms in the definition of $\lsw(X)$, the count of the Seiberg--Witten monopoles on $X$ and the index-theoretic correction term, using metrics on $X$ with long necks $(0,R) \times Y$. This neck stretching technique for counting monopoles is well known in gauge theory, although mainly in the separating case. The non-separating case at hand was studied in \cite[Section 11.1]{F:gluing} under the technical assumption of the absence of reducible monopoles on the non-compact manifold $\wp$ obtained by attaching infinite product  ends to $W$. In our case, this assumption fails because $b^{+}_{2}(W) = 0$. Instead, we prove that such a reducible monopole does not cause any trouble because it is isolated in the sense that a sequence of irreducible monopoles on $X$ cannot converge to it as $R \to \infty$. This is proved using an a priori estimate on the smallest eigenvalue of the Dirac Laplacian on a manifold with long neck. We provide  a detailed argument in Sections \ref{S:eigenvalue}, \ref{S:compact}, and \ref{S:gluing}, which all use the setup of the book \cite{KM}. As a result, we are able to match the count of monopoles on $X$ with a certain Lefschetz number in the monopole Floer chain complex. The truly novel part of the proof of Theorem \ref{T:main}, however, is the calculation of the correction term in $\lsw(X)$, which boils down to calculating the index of the spin Dirac operator on a manifold with periodic end modeled on the infinite cyclic cover of $X$ as $R \to \infty$. This is done in two substantially different ways, one direct in Section \ref{S:correction} and the other using the end-periodic index theorem of \cite{MRS3} in Section \ref{S:eta}.

A  very useful technical result underpinning the proof of Theorem \ref{T:main} is the existence of Riemannian metrics on the manifold $\wp$ with infinite product ends which make the $L^2$ Sobolev completion of the spin Dirac operator invertible. This existence result, which is proved in Section \ref{S:generic} in all dimensions divisible by four, is an extension of the generic metric theorem of \cite{ADH} to certain non-compact manifolds. One advantage of working with such generic metrics is that they greatly simplify the treatment of perturbations needed to ensure the regularity of the Seiberg--Witten moduli spaces, and allow to avoid perturbations on manifolds with periodic ends altogether. See the discussion at the end of Section \ref{S:lsw}.

%%%%%%%%%%%%%%%%%%%%%%%%%%%%%%%%%%%%%%%%%%%%%%%

\subsection{Calculations and applications}
The splitting formula of Theorem~\ref{T:main} makes the invariant $\lsw(X)$ computable in a number of cases. This is due, on one hand, to the availability of advanced computational tools in monopole Floer homology (such as Floer exact triangles and the $\Pin(2)$ symmetry), and on the other, to the identification between monopole Floer homology and Heegaard Floer homology, by the work of Kutluhan, Lee, and Taubes \cite{kutluhan-lee-taubes:HFSW-I,kutluhan-lee-taubes:HFSW-II,kutluhan-lee-taubes:HFSW-III,kutluhan-lee-taubes:HFSW-IV,kutluhan-lee-taubes:HFSW-V}, or alternatively, the work of Colin, Ghiggini, and Honda \cite{colin-ghiggini-honda:HFECH-I,colin-ghiggini-honda:HFECH-II,colin-ghiggini-honda:HFECH-III} and Taubes \cite{taubes:HMECH}. This newly found computability of $\lsw(X)$ leads to a  number of applications, of which we present two in Sections~\ref{S:Theorem B} and \ref{S:Theorem C}. In both applications, the Lefschetz number in formula \eqref{E:main} vanishes, albeit for different reasons. 

The first application gives an obstruction to a $4$-manifold having a Riemannian metric of positive scalar curvature. Historically~\cite{witten:monopole}, the Seiberg--Witten invariants have been used to produce many obstructions of this nature that go well beyond the classical index-theoretic obstruction of Lichnerowicz~\cite{lichnerowicz:spinors}. We add to this body of knowledge the following theorem, which was originally proved by the first-named author~\cite[Theorem 1.2]{lin:sw-psc} using different techniques. It was conjectured in \cite[Remark 1]{lin:sw-psc} that there should exist a proof along the lines of this paper.

\begin{introthm}\label{T:psc}
Let $X$ be a smooth oriented spin rational homology $S^1\times S^3$ which is homology oriented by a choice of generator in $H^1 (X; \Z)$. Assume that the Poincar\'e dual of this generator is realized by a rational homology sphere $Y \subset X$ with the induced spin structure $\s$. Then $X$ admits no Riemannian metric of positive scalar curvature unless $\lsw(X) + h(Y,\s) = 0$.
\end{introthm}

It was shown in \cite{MRS1} that $\lsw(X)$ reduces modulo $2$ to the Rohlin invariant $\rho(X)$. As in~\cite{lin:sw-psc}, this fact leads to the corollary that, if $X$ admits a metric of positive scalar curvature, any rational homology sphere $Y$ carrying the generator of $H_3 (X;\Z)$ must satisfy the relation $h(Y,\s) = \rho (X)\pmod 2$ with respect to the induces spin structure $\s$. For example, if a generator of $H_3 (X;\Z)$ is carried by the Brieskorn homology sphere $\Sigma(2,3,7)$, then $X$ cannot admit a positive scalar curvature metric.

The second, and more elaborate, application is to the study of homology cobordisms. Recall that oriented 3-manifolds $Y_0$ and $Y_1$ are called homology cobordant (respectively, $\Z/2$ homology cobordant) if there exists a smooth, compact, oriented cobordism $W$ from $Y_0$ to $Y_1$ such that $H_*(W,Y_i; \Z) = 0$ (respectively, $H_*(W,Y_i; \Z/2)=0$) for $i=0,1$. The homology cobordism group $\Theta^3_{\Z}$ is generated by oriented integral homology 3-spheres, modulo the equivalence relation of being homology cobordant. Similarly, the $\Z/2$ homology cobordism group $\Theta^3_{\Z/2}$ is generated by oriented $\Z/2$ homology 3-spheres, modulo the equivalence relation of being $\Z/2$ homology cobordant.

Let us first consider the group $\Theta^3_{\Z}$. Recall that the Rohlin invariant provides a surjective homomorphism $\rho: \Theta^3_{\Z}\to \Z/2$. Manolescu \cite{manolescu:triangulation} used $\Pin(2)$-equivariant Seiberg--Witten theory to show that this homomorphism does not split, that is, no integral homology sphere $Y$ with $\rho(Y)= 1$ has order two in $ \Theta^{3}_{\mathbb{Z}}$. It seems reasonable to conjecture that $\rho(Y)= 1$ in fact implies that $Y$ has infinite order in $\Theta^{3}_{\mathbb{Z}}$; indeed, this was shown to be true for all Seifert fibered homology spheres by the third-named author~\cite{saveliev:plumbed}. Generalizing this result, we show that the conjecture holds under an additional assumption that $Y$ is $h$-positive or $h$-negative: a $\Z/2$ homology sphere $Y$ is said to be $h$-positive (respectively, $h$-negative) if the reduced monopole homology $\hmred(Y,\s)$ corresponding to the unique spin structure $\s$ on $Y$ is supported in degrees $\geq -2h(Y,\s)$ (respectively, $\leq -2h(Y,\s)-1$); see Definition \ref{defi:h-positive/h-negative} for more details. All Seifert fibered homology spheres satisfy this assumption, and many more examples can be found in Section \ref{S:Theorem C}.

\begin{introthm}\label{T:Theta}
Any $h$-positive (or $h$-negative) integral homology sphere $Y$ with $\rho(Y) = 1$ has infinite order in $\Theta^{3}_{\mathbb{Z}}$.
\end{introthm}

It is worth mentioning that, instead of using Theorem \ref{T:main}, one could follow similar arguments to deduce Theorem \ref{T:Theta} from Stoffregen's connected Seiberg--Witten Floer homology \cite{stoffregen:pin(2)-seifert}.

Let us now turn our attention to the group $\Theta^3_{\Z/2}$ where we can prove a stronger result. Let $\Theta^3_L$ be the subgroup of $\Theta^3_{\Z/2}$ generated by all the $L$-spaces (over $\Q$) that are $\Z/2$ homology spheres. It was shown by Stoffregen \cite{stoffregen:pin(2) connected sum} (using the invariants $\alpha$, $\beta$, and $\gamma$) and Hendricks, Manolescu, and Zemke \cite[Proposition 1.4]{manolescu-hendricks-zemke:involutive connected sum} (using the invariants $\overline{d}$ and $\underline{d}$) that the Brieskorn homology sphere $\Sigma(2,3,7)$ has infinite order in the quotient group $\Theta^3_{\Z/2}/\Theta^3_L$. In particular, this implies that $\Theta^3_{\Z/2}/\Theta^3_L$ is an infinite group. The following theorem generalizes this result by exhibiting a large family of such manifolds.

\begin{introthm}\label{T:Theta2}
Any $h$-positive (or $h$-negative) $\Z/2$ homology sphere $Y$ with $\rho(Y)\neq h(Y)\pmod 2$ has infinite order in the group $\Theta^3_{\Z/2}/\Theta^3_L$.
\end{introthm}

A further study of the $h$-positive ($h$-negative) condition using the rational surgery formula of Ozsv\'ath and Szab\'o~\cite{oz:q-surgery} leads to the following corollary.

\begin{introcor}\label{C: surgery infinite order}
Let $Y$ be an integral homology sphere obtained by $1/n$ surgery on a knot $K$ in $S^3$ with $n$ odd, and suppose that $\arf(K) =1$.  Then $Y$ has infinite order in $ \Theta^{3}_{\mathbb{Z}}$ in the following cases\,:
\begin{enumerate}
\item $K$ is the figure-eight knot;
\item $K$ is a quasi-alternating knot with non-zero signature $\sigma(K)$.
\end{enumerate}
\end{introcor}

We remark that the Fr{\o}yshov invariant $h(Y)$ in the above corollary vanishes both in case (1) and in case (2) whenever $n\cdot\sigma(K) > 0$ (see Lemma \ref{property of thin knots}) hence the result of the corollary cannot be proved using $h(Y)$ alone. It may be possible, however, to give an alternative proof using other Fr\o yshov--type invariants involving the $\Pin(2)$-symmetry, such as the invariants $\alpha$, $\beta$, and $\gamma$ of Manolescu \cite{manolescu:triangulation, lin:morse-bott} and the invariants $\overline{d}$ and $\underline{d}$ of Hendricks and Manolescu \cite{manolescu-hendricks:involutive}, once their behavior under the connected sum operation is better understood. 

\begin{introex}
Let $Y$ be the integral homology sphere obtained by $1/n$ surgery on a two-bridge knot $K(24m \pm 5,3)$ or $K(24m\pm 11,3)$ with $m > 0$ and odd $n > 0$. Then $h(Y) = 0$ and $Y$ has infinite order in the groups $\Theta^3_{\Z}$ and $\Theta^3_{\Z/2}/\Theta^3_L$. 
\end{introex}

Another major reason to study the group $\Theta^3_{\Z/2}$ is to gain information about the smooth knot concordance group $\C_{\rm s}$ via passing to the double branched cover of the knot. Recall that a knot $K$ is called Khovanov-homology thin (over $\mathbb{Z}/2$) if its reduced Khovanov homology $\widetilde{Kh}(K;\mathbb{Z}/2)$ is supported in a single $\delta$-grading (see \cite{khovanov I, khovanov II}). Such knots are  very common: all quasi-alternating knots \cite{manolescu-ozsvath:quasi} and 238 of the 250 prime knots with up to 10 crossings are  Khovanov-homology thin. Let $\C_{\rm{thin}}$ be the subgroup of $\C_{\rm{s}}$ generated by the Khovanov-homology thin knots. Theorem \ref{T:Theta2} has the following curious corollary.

\begin{introcor}\label{C: thin group infinite index}
Any knot $K$ whose double branched cover satisfies the conditions of Theorem \ref{T:Theta2} generates an infinite cyclic subgroup in $\C_{\rm s}/\C_{\rm{thin}}$.
\end{introcor}

\begin{introex}
Let $K$ be the $(3,12n-5)$ or $(3,12n-1)$ torus knot, $n>0$. Then the connected sum $\#_{m}K$  ($m>0$) is never smoothly concordant to any Khovanov-homology thin knot. 
\end{introex}

In our subsequent paper \cite{LRS2}, we will use similar techniques to compute the invariant $\lsw(X)$ for the mapping tori $X$ of all smooth orientation preserving involutions $\tau: Y \to Y$ on integral homology spheres $Y$ with the quotient $S^3$. That calculation will confirm the conjecture \cite[Conjecture B]{MRS1} for all such mapping tori $X$ by identifying $\lsw(X)$ with its Yang--Mills counterpart, the invariant $\lambda_{\rm{FO}}(X)$ of Furuta and Ohta \cite{Furuta-Ohta}.

%%%%%%%%%%%%%%%%%%%%%%%%%%%%%%%%%%%%%%%%%%%%%%

\subsection{Organization of the paper}
We begin by reviewing the definitions of the monopole Floer homology and of the invariants $\lsw(X)$ and $h(Y,\s)$ that go into the statement of Theorem \ref{T:main} in Section \ref{S:prelim}. We also use this section to settle various technical matters and introduce some notations. The proofs of Theorems A, B, C and D, as well as Corollaries \ref{C: surgery infinite order} and \ref{C: thin group infinite index},  are given in the three sections that follow. They rely on certain technical results whose proofs are postponed until later in the paper for the sake of exposition. The first of these results, discussed in Section \ref{S:correction}, is a calculation of the index-theoretic correction term $w(X,g)$ on manifolds with long necks. An alternative calculation using the end-periodic index theorem of \cite{MRS3} is given in Section \ref{S:eta}. Both calculations boil down to computing the $L^2$ index of a spin Dirac operator on certain manifolds with periodic ends, which we do in all dimensions $n \equiv 0\pmod 4$. Similar techniques lead to the first eigenvalue estimates in Section \ref{S:eigenvalue}. These are later used for the compactness and gluing results in Sections \ref{S:compact} and \ref{S:gluing}. Several of these results rely on the generic metric theorem whose proof uses a rather different set of techniques and for this reason is postponed until Section \ref{S:generic}. 

Although many results in this paper can be generalized to $\spinc$ structures, we will be mainly concerned with spin structures. We will omit spin structures on 4-manifolds from our notations. We will usually include spin structures on 3-manifolds, to be consistent with \cite{KM}. One exception are $\Z/2$ homology spheres: each of these has a unique spin structure which will be omitted.  

\medskip\noindent
\textbf{Acknowledgments:}\; We are thankful to Tye Lidman, Ciprian Manolescu, and Matthew Stoffregen for generously sharing their expertise.

%%%%%%%%%%%%%%%%%%%%%%%%%%%%%%%%%%%%%%%%%%%

\section{Preliminaries}\label{S:prelim}
We begin by recalling the definitions of all the invariants involved and settling some technical matters.

%%%%%%%%%%%%%%%%%%%%%%%%%%%%%%%%%%%%%%%%%%%

\subsection{The invariant $\lsw (X)$}\label{S:lsw}
Let $X$ be an oriented smooth 4-manifold with rational homology of $S^1 \times S^3$, equipped with a fixed spin structure. We will review the definition of $\lsw(X)$ following \cite{MRS1}. Note that the invariant $\lsw(X)$ was defined in \cite{MRS1} only for an integral homology $S^1\times S^3$ but a careful check of details shows that the construction of \cite{MRS1} extends to a rational $S^1\times S^3$ essentially word for word. There are only two exceptions, which will be discussed in Remark \ref{R:p-cyclic cover} and Remark \ref{R: dependence on spin str}. With a slight abuse of language, we will cite \cite{MRS1} directly.

%The two different spin structures on $X$ are equivalent as $\spin^c$ structures, which makes the Seiberg--Witten moduli space independent of this choice. 

Fix a homology orientation on $X$ by choosing a generator $1 \in H^1 (X; \Z) = \Z$. Given a metric $g$ on $X$ and a co-closed 1-form $\beta \in \Omega^1(X;i\mathbb R)$ orthogonal to $\mathcal{H}^1 (X; i\mathbb R)$ (the space of harmonic 1-forms on $X$), consider the triples $(A,s,\phi)$ consisting of a $U(1)$ connection $A$ on the determinant bundle of the spin bundle, a real number $s \ge 0$, and a positive spinor $\phi$ such that $\|\phi\|_{L^2(X)} = 1$.  The gauge group ${\rm Map}\,(X,S^1)$ acts freely on such triples by the rule $u(A,s,\phi) = (A - u^{-1} du,s,u\phi)$. The blown-up Seiberg--Witten moduli space $\M(X,g,\beta)$ consists of the gauge equivalence classes of triples $(A,s,\phi)$ that solve the perturbed Seiberg--Witten equations
\[
\begin{cases}
\; F^+_A - s^2\, \tau(\phi) = d^+\beta \\
\; \D^+_A \,(X,g) (\phi) = 0.
\end{cases}
\]
The solutions are referred to as monopoles. Monopoles with $s = 0$ are called reducible, and all other monopoles are called irreducible. The latter are identified with the irreducible monopoles in the usual moduli Seiberg--Witten moduli space via the map $(A,s,\phi) \to (A,s\phi)$.

According to \cite[Proposition 2.2]{MRS1}, the moduli space $\M(X,g,\beta)$ is regular for a generic choice of $(g,\beta)$. In particular, $\M(X,g,\beta)$ is a compact oriented manifold of dimension zero that contains only irreducible solutions. The count of points in $\M(X,g,\beta)$ with signs determined by the orientation and homology orientation is denoted by $\#\M (X,g,\beta)$. The invariant $\lsw(X)$ is defined in \cite{MRS1} by the formula
\[
\lsw (X) = \#\,\M(X,g,\beta) - w\,(X,g,\beta),
\]
where $w(X,g,\beta)$ is the correction term which cancels the dependence of $\#\M (X,g,\beta)$ on the parameters $(g,\beta)$. The definition of $w(X,g,\beta)$ is as follows.

Let $Y \subset X$ be a connected manifold which is Poincar\'e dual to the chosen generator in $H^1 (X;\mathbb Z)$. Note that $Y$ is canonically oriented, and inherits a spin structure $\s$ from $X$. %Also note that the two spin structures on $X$ have the same restriction to $Y$. 
Denote by $W$ the cobordism from $Y$ to itself obtained by cutting $X$ open along $Y$. For any smooth compact spin manifold $Z$ with spin boundary $Y$, consider the manifold
\begin{equation}\label{E:zp}
\zp(X)\; =\; Z \,\cup\,W_0\,\cup\,W_1\,\cup\,W_2\,\cup\ldots
\end{equation}
with periodic end modeled on the infinite cyclic cover of $X$; each of the manifolds $W_i$ in this formula is just a copy of $W$. Choose a metric and a perturbation on $\zp(X)$ which match the metric and the perturbation over the end lifted from those on $X$. Then the operator $\D^+(\zp(X),g,\beta) = \D^+(\zp(X),g) + \beta$ is Fredholm with respect to the usual Sobolev $L^2$ completion, and the correction term
\begin{equation}\label{E:w}
w(X,g,\beta)\; =\; \ind \D^+(\zp(X),g,\beta)\, +\, \frac 1 8\,\sign Z
\end{equation}
is independent of the choice of $Y$ and $Z$. 

\begin{rmk}\label{R:p-cyclic cover}
The proof of well-definedness of $w(X,g,\beta)$ in \cite[Proposition 3.2]{MRS1} makes use of the following fact: if $X$ is an integral homology $S^1 \times S^3$, its $p$--fold cyclic cover $X_p$ is a rational homology $S^1\times S^3$ for any prime number $p$. This need not be true if $X$ is a rational homology $S^1 \times S^3$. However, one can use the argument of \cite{MRS1} to show that $X_p$ is a rational homology $S^1\times S^3$ as long as the prime number $p$ is large enough so that $H_* (X; \Z/{p}) = H_* (S^1\times S^3; \Z/p)$. This is sufficient to complete the proof of well-definedness of $w(X,g,\beta)$.
\end{rmk}

Theorem A of \cite{MRS1} asserts that $\lsw(X)$ is indepen\-dent of the choice of metric $g$ and generic perturbation $\beta$, and that the reduction of $\lsw (X)$ modulo $2$ is the Rohlin invariant of $X$.

\begin{rmk}\label{R: dependence on spin str} 
Unlike in the case of an integral homology $S^1\times S^3$ treated in \cite{MRS1}, different spin structures on $X$ may lead to different invariants $\lsw(X)$. To keep our notations clean, we will not include the spin structure in the notation. Note that when $X$ is a $\Z/2$ homology $S^1\times S^3$, different spin structures are all equivalent as $\spinc$ structures and hence give the same invariant $\lsw(X)$.
\end{rmk}  

There are several implicit orientation conventions that go into the definition of $\lsw(X)$. We will not discuss them here but notice that altering these conventions consistently only changes $\lsw(X)$ by an overall sign. The sign of $\lsw(X)$ was fixed in [32, Section 11.2] by the condition
\[
\lsw (S^1 \times Y) = - \lambda (Y),
\]
where $\lambda (Y)$ is the Casson invariant of an integral homology sphere $Y$ normalized so that $\lambda(\Sigma(2,3,5)) = -1$ for the Brieskorn homology sphere $\Sigma(2,3,5)$ oriented as a link of complex surface singularity.

We conclude this section by addressing an important technical point about choices of metrics and perturbations. According to \cite{RS}, the operator $\D^+(\zp(X),g,\beta)$ can be made Fredholm by choosing a generic metric $g$ and letting $\beta = 0$. This choice of metric also guarantees \cite[Proposition 7.2]{MRS1} that the moduli space $\M(X,g,0)$ has no reducibles but  not that it is regular. One way to ensure regularity is to choose a generic perturbation $\beta$ small enough so that $\ind \D^+(\zp(X),g) =  \ind \D^+(\zp(X),g,\beta)$. Then
\begin{equation}\label{E:wxg}
\lsw (X) = \#\,\M(X,g,\beta) - w\,(X,g),
\end{equation}
where $w(X,g)$ stands for $w(X,g,0)$. The perturbation $\beta$ in this formula can be replaced by a more general perturbation $\gamma$ as in \cite[Section 9.2]{MRS1} or in Section \ref{S: perturbation} below. As long as $\gamma$ ensures regularity and is small enough, one can connect $\gamma$ and $\beta$ by a generic path $\{\,\beta_{t}\mid 0\leq t\leq 1\}$ of small perturbations. Since the perturbations $\beta_{t}$ are small, the union 
\[
\mathop{\bigcup}\limits_{0\leq t\leq 1} \M(X,g,\beta_{t})
\]
contains no reducibles (just like $\M(X,g,0)$) and provides an oriented cobordism between $\M (X,g,\beta)$ and $\M (X,g,\gamma)$, as explained in \cite[Proposition 9.4]{MRS1}. This gives us the formula
\begin{equation}\label{E:gen-pert}
\lsw (X) = \#\,\M(X,g,\gamma) - w\,(X,g).
\end{equation}

% The other point concerns metrics with product regions. These are metrics on $X$ which take the form $g = dt^2 + h$ in a tubular neighborhood $[-\ep,\ep] \times Y$, where $t$ is the normal coordinate and $h$ is a metric on $Y$. Note that $\D^{\pm}(X,g) = \p/\p t \pm \D(Y,h)$ in the product region. According to \cite[Theorem 1.1]{ADH}, the metric $h$ can be chosen so that $\ker \D(Y,h) = 0$. It then follows from Theorem \ref{T:one} below that one can choose a metric $g$ on $X$ with a product region which makes the operator $\D^+ (\zp(X),g)$ Fredholm. Moreover, such metrics are generic among all metrics on $X$ with a fixed product metric on $[-\ep,\ep] \times Y$; see Remark \ref{R:two}. In particular, our formula \eqref{E:gen-pert} holds for a generic metric $g$ on $X$ with a product region and any sufficiently small perturbation $\gamma$.

%%%%%%%%%%%%%%%%%%%%%%%%%%%%%%%%%%%%%%%%%%%%%%

\subsection{Monopole Floer homology}\label{S:HM}
In this subsection, we will briefly recall the definition of the monopole Floer homology $HM(Y,\s)$. We will focus on the special case when $Y$ is a rational homology sphere and $\s$ is a spin structure, which will suffice for the purpose of this paper. The general definition can be found in  Kronheimer--Mrowka \cite[Chapter 1]{KM}. We will work with rational coefficients and omit the coefficient ring from our notations.

Let $Y$ be an oriented rational homology sphere with a Riemannian metric $h$ and a spin structure $\s$. An important example to have in mind is the rational homology sphere $Y \subset X$ of Theorem \ref{T:main} with the induced spin structure $\s$. Trivialize the spinor bundle $\S$ and choose the product connection $B_0$ to be our reference connection. We will make the following assumption, which according to \cite[Theorem 1.1]{ADH} holds for a generic metric $h$.

%\begin{rmk}
%In later sections, when $Y$ is an integral homology sphere or the $\spin^c$ structure is obvious from the context, we will neglect $\mathfrak{s}$ from our notations.
%\end{rmk}

\begin{assum}\label{generic metric on Y}
The spin Dirac operator $\D (Y,h)$ has zero kernel. %In particular, there exists $\epsilon_0 > 0$ such that $\D_{B_0} (Y,h)$ has no eigenvalues in the interval $[-\epsilon_0,\epsilon_0]$.
\end{assum}

Let $\mathcal{C} (Y)$ be the Sobolev $L^2_{k-1/2}$ completion of the affine space of configurations $(B,\psi)$, where $B$ is a connection in $\S$, $\psi$ is a spinor, and $k \ge 3$ is an integer which will be fixed throughout the paper. We refer to $\mathcal{C}(Y)$ as the configuration space. We also introduce the blown-up configuration space $\mathcal{C}^{\sigma}(Y)$, which consists of the triples $(B,s,\psi)$, where $B$ is a connection in $\S$, $s \ge 0$ is a real number, and $\psi$ is a spinor with $\|\psi\|_{L^2 (Y)}=1$. 

Let $p: \mathcal{C}^{\sigma}(Y)\rightarrow \mathcal{C}(Y)$ be the natural projection sending $(B,s,\psi)$ to $(B,s \psi)$. Using the terminology of \cite{KM}, configurations in $\mathcal{C}(Y)$ are said to be downstairs, and those in $\mathcal{C}^{\sigma}(Y)$ are said to be upstairs. A downstairs configuration $(B,\psi)$ is called irreducible if $\psi\neq 0$, while an upstairs configuration $(B,s,\psi)$ is called irreducible if $s\neq  0$. All other configurations are called reducible. The projection $p$ provides a diffeomorphism between the spaces of irreducible configurations upstairs and downstairs; we denote both of these spaces by $\mathcal{C}^*(Y)$.

The group of $L^{2}_{k+1/2}$ gauge transformations $u: Y \to S^1$ acts on $\mathcal{C}(Y)$ by the formula $u(B,\psi) = (B - u^{-1} du,u\psi)$ and on $\mathcal{C}^{\sigma}(Y)$ by the formula $u(B,s,\psi) = (B - u^{-1} du,s,u\psi)$. Both actions restrict to an action on $\mathcal{C}^*(Y)$. The corresponding quotient spaces will be denoted by $\B(Y)$, $\B^{\sigma}(Y)$, and $\B^*(Y)$. Note that $\B^*(Y)$ is a Hilbert manifold, while $\B^{\sigma}(Y)$ is a Hilbert manifold with boundary. The boundary of $\B^{\sigma}(Y)$, given by the equation $s = 0$, contains the gauge equivalence classes of all reducible configurations $[(B,0,\psi)]$.

The monopole Floer homology was defined in \cite{KM} as a variant of Morse homology of the Chern--Simons--Dirac functional $\L: \mathcal{C}(Y) \to \mathbb{R}$. The definition of $\L$ can be found in \cite[Definition 4.1.1]{KM}. To ensure that transversality holds, $\L$ is perturbed using a perturbation $\mathfrak{q}$ which is the formal gradient of a gauge invariant functional $f:\mathcal{C}(Y) \to \mathbb{R}$. Note that $\mathfrak{q}$ has two well-defined components, the connection component $\mathfrak{q}^{0}$ and the spinor component $\mathfrak{q}^{1}$. The perturbed Chern--Simons--Dirac functional is denoted by $\L_{\mathfrak{q}} = \L + f$. Its gradient $\grad \L_{\mathfrak{q}} = \grad \L +\mathfrak{q}$ gives rise to a vector field $v^{\sigma}_{\mathfrak{q}}$ on $\mathcal{B}^{\sigma}(Y)$.

Let $\epsilon_0 > 0$ be any small number such that $\D(Y,h)$ has no eigenvalues in the interval $[-\epsilon_0, \epsilon_0]$ (the existence of $\epsilon_0$ follows from Assumption \ref{generic metric on Y}). Then one can prove as in \cite[Proposition 2.8]{lin:sw-psc} that there exist perturbations $\mathfrak{q}$ satisfying the following assumption. 

\begin{assum}\label{small perturbation}
The perturbation $\mathfrak{q}$ satisfies the following three conditions:
\begin{enumerate}[(a)]
\item $\mathfrak{q}$ is nice, that is, $\mathfrak{q}(B,0) = 0$ for all $B$; in other words, $\mathfrak{q}$ equals zero when restricted to reducible configurations,
\item $\mathfrak{q}$ is admissible, that is, the critical points of $v^{\sigma}_{\mathfrak{q}}$ are non-degenerate and the moduli spaces of trajectories connecting them are regular; see Definition 22.1.1 of \cite{KM}, and
\item the derivative of the spinor component $\mathfrak{q}^1$ of $\mathfrak{q}$ satisfies the inequality 
\begin{equation}\label{E: derivative of perturbation small}
\left\| D_{(B_0,0)}\,\mathfrak{q}^{1}(0,\psi) \right\|_{L^2(Y)}\; \leq\; \frac 14\,\epsilon_0\cdot \|\psi\|_{L^2(Y)}
\end{equation}
for any $\psi\in L^2_{k-1/2}(Y;\S)$.
\end{enumerate}
\end{assum}
Under Assumption \ref{small perturbation}, the set $\mathfrak{C}$ of critical points of $v^{\sigma}_{\mathfrak{q}}$ is discrete and can be decomposed into the disjoint union of three subsets:
\begin{itemize}
\item $\mathfrak{C}^{o}$: the set of  irreducible critical points;
\item $\mathfrak{C}^{s}$: the set of reducible, boundary stable critical points (i.e., reducible critical points near which    $v^{\sigma}_{\mathfrak{q}}$ points inside the boundary);
\item $\mathfrak{C}^{u}$: the set of reducible, boundary unstable critical points (i.e., reducible critical points near which    $v^{\sigma}_{\mathfrak{q}}$ points outside the boundary).
\end{itemize}

We will next take a closer look at the reducible critical points. According to \cite[Corollary 4.2.2]{KM}, the vector field $\grad \L + \mathfrak{q}$ has a unique reducible critical point $[(B_0,0)]$ downstairs, which we call $[\theta]$. The situation with the reducible critical points upstairs is quite different. To describe it, consider the perturbed Dirac operator
\[
\D_{B_0,\mathfrak{q}} (Y,h) = \D_{B_0} (Y,h) + D_{(B_{0},0)}\,\mathfrak{q}^{1}(0,-):\, L^2_1\,(Y; \S)\longrightarrow L^2(Y; \S).
\]
This is a self-adjoint elliptic operator. Since $\mathfrak{q}$ is admissible, its eigenvalues are all non-zero and have multiplicity one. We enumerate them so that
\[
\cdots <\, \lambda_{-2}\,<\,\lambda_{-1}\,<\,0\,<\,\lambda_0\,<\,\lambda_1\,< \cdots.
\]
For each $\lambda_i$ pick an eigenvector $\psi_i$ of unit $L^2$-norm and let $[\theta_i] = [(B_0, 0,\psi_i)] \in \B^{\sigma}(Y)$. Then 
\[
\mathfrak{C}^{s}=\{[\theta_{i}]\mid i\geq 0\}\quad\text{and}\quad \mathfrak{C}^{u}=\{[\theta_{i}]\mid i<0\}.
\]

Let $C^{o}$ (respectively $C^{s}$ and $C^{u}$) be a vector space over $\mathbb{Q}$ with the basis $\{e_{[\alpha]}\}$ indexed by the critical points $[\alpha]$ in $\mathfrak{C}^{o}$ (respectively $\mathfrak{C}^{s}$ and  $\mathfrak{C}^{u}$). Define a linear map $\partial^{o}_{o}: C^{o}\rightarrow C^{o}$ by the formula
\[
\partial^{o}_{o} \, e_{[\alpha]} = \mathop{\sum}\limits_{[\beta]\in \mathfrak{C}^{o}} \#\breve{\mathcal{M}}([\alpha],[\beta])\cdot e_{[\beta]} %\ \ \ \ ([\alpha]\in \mathfrak{C}^{o})
\]
where $\breve{\mathcal{M}}([\alpha],[\beta])$ is the moduli space of unparameterized flow lines going from $[\alpha]$ to $[\beta]$ and  $\#\breve{\mathcal{M}}([\alpha],[\beta])$ is the signed count of points in this moduli space. (This number is set to be zero if the dimension of the moduli space is positive.) One defines maps $\partial^{o}_{s}:C^{o}\rightarrow C^{s}$ and $\partial^{u}_{o}:C^{u}\rightarrow C^{o}$ similarly. Consider the vector spaces
\[
\overline{C}=C^{s}\oplus C^{u},\quad \widecheck{C}=C^{o}\oplus C^{s}\quad\text{and}\quad\widehat{C}=C^{o}\oplus C^{u}.
\]
The monopole Floer homologies $\overline{HM}(Y,\s)$, $\widecheck{HM}(Y,\mathfrak{s})$ and $\widehat{HM}(Y,\mathfrak{s})$ are defined, respectively, as the homology of the chain complexes
$(\overline{C},0)$, $(\widecheck{C},\widecheck{\partial})$ and $(\widehat{C},\widehat{\partial})$ with the differentials

\[
\widecheck{\partial}=\left(\begin{array} {cc}
 \partial^{o}_{o}  & 0  \\
 \partial^{o}_{s}  & 0
\end{array}\right)\quad \text{and}\quad\widehat{\partial}=\left(\begin{array} {cc}
 \partial^{o}_{o}  & \partial^{u}_{o}  \\
 0  & 0
\end{array}\right).
\]

\medskip\noindent
(Note that our formulas are simpler than those in \cite{KM} because we are working with a rational homology sphere $Y$ and a nice perturbation $\mathfrak{q}$.) The chain map $i: \overline{C} \rightarrow \widecheck{C}$ with the matrix 
\[
\begin{pmatrix}
0&\partial^{u}_{o} \\ 1&0
\end{pmatrix}
\]
induces a natural map 
\begin{equation}\label{E:map i}
i_*: \overline{HM}(Y,\mathfrak{s})\longrightarrow \widecheck{HM}(Y,\mathfrak{s}).
\end{equation}
We define the reduced monopole Floer homology as $\hmred(Y,\mathfrak{s}) = \coker i_*$. This is a finite dimensional vector space. Note that this definition matches the definition of the reduced monopole Floer homology in \cite[Definition 3.6.3]{KM} because of the long exact sequence \cite[(3.4)]{KM}. A rational homology sphere $Y$ is called an $L$-space (over $\mathbb{Q}$) if  the reduced monopole Floer homology  $\hmred(Y,\mathfrak{s})$ vanishes for all $\spinc$ structures $\s$.

All these different versions of monopole Floer homology are modules over the polynomial ring $\mathbb{Q}[U]$, where $U$ is a formal variable of degree $-2$. In fact, we have canonical isomorphisms $\overline{HM}(Y,\mathfrak{s}) \cong \mathbb{Q}[U,U^{-1}]$ and $\im i_*\cong\mathbb{Q}[U]/U\cdot \mathbb{Q}[U]$ as well as a non-canonical splitting $\widecheck{HM}(Y,\mathfrak{s})\cong \im i_*\,\oplus\, \hmred(Y,\mathfrak{s})$. In addition, the monopole Floer homology has the so-called ``TQFT property''. More precisely, any $\spin^c$ cobordism $(W,\mathfrak{s}_{W})$ from $(Y_{0},\mathfrak{s}_{0})$ to $(Y_{1},\mathfrak{s}_{1})$ induces a morphism 
\begin{equation}\label{E:W*}
(W,\mathfrak{s}_{W})_{*}:HM (Y_{0},\mathfrak{s}_{0})\rightarrow HM (Y_{1},\mathfrak{s}_{1})
\end{equation}
of $\mathbb{Q}[U]$-modules, where $HM$ stands for any one of the monopole Floer homologies $\widecheck{HM}$, $\widehat{HM}$, $\overline{HM}$ or $\hmred$. Each of these morphisms is induced by a respective chain map, whose definition requires further perturbations as described in Section \ref{S: perturbation}. Note that in the current paper, we only consider spin structures $\s_{W}$ and omit them from our notations.
%we always have $H^{2}(W,\partial W;\mathbb{Z})=0$. As a result, there is a unique spin$^{c}$ structure $\mathfrak{s}_{W}$ extending $\mathfrak{s}_{0}\cup \mathfrak{s}_{1}$ and .

Next, we need to discuss the canonical gradings in monopole Floer homology. With each critical point $[\alpha]\in \mathfrak{C}$ one associates two gradings,
\[
\gr^{\Q}([\alpha])\,\in\,\Q \quad \text{and} \quad \gr^{(2)}([\alpha])\,\in\,\Z/2;
\]
see \cite[page 587]{KM} for the former and \cite[page 427]{KM} for the latter. These naturally induce (absolute) $\Q$ and $\Z/2$ gradings on $\widecheck{HM}(Y,\s)$, $\widehat{HM}(Y,\s)$ and $\hmred(Y,\s)$. The generators $[\theta_i] \in \overline{HM}(Y,\s)$ are graded by $\gr^{\mathbb{Q}}([\theta_{i}])$ and $\gr^{(2)}([\theta_{i}])$ if $i\geq 0$, and by $\gr^{\mathbb{Q}}([\theta_{i}])-1$ and $\gr^{(2)}([\theta_{i}])-1$ if $i< 0$. In all cases, the $U$-action decreases the $\Q$-grading by two and preserves the $\Z/2$ grading. We will use the $\Z/2$  grading to define various Lefschetz numbers, and use the $\Q$-grading to define the Fr{\o}yshov invariant.

\begin{defi}
The Fr{\o}yshov invariant $h(Y,\s)$ is defined as negative one-half of the lowest $\mathbb{Q}$-grading of elements in $\im (i)$, where $i$ is the map \eqref{E:map i}. If $X$ is a spin homology $S^1 \times S^3$ as in Theorem \ref{T:main}, we define its Fr{\o}yshov invariant by the formula $h(X) = h(Y,\s)$, where $Y \subset X$ is an oriented rational homology sphere Poincar\'e dual to $1 \in H^1 (X;\Z)$, with the induced spin structure $\s$. It follows from \cite[Theorem 8]{F} that $h(X)$ is well-defined.
\end{defi}

The Fr{\o}yshov invariant $h(Y,\s)$ is an invariant of $\spinc$ rational homology cobordism. It is also known \cite[Theorem 3]{F} that it changes sign with the change of orientation and that it is additive with respect to connected sums. A Heegaard Floer version of $h(X)$ was defined in~\cite{levine-ruberman:codim1} without the assumption that $Y$ be a rational homology sphere.

We will conclude this section by computing the gradings of the generators $[\theta_i] \in \overline{HM}(Y,\s)$. To this end, consider a smooth compact spin manifold $Z$ with spin boundary $Y$ and define
\begin{equation}\label{E:correction term Y}
n(Y,h,\s)\; =\; \ind \D^+ (\zp)\, +\, \frac 1 8\,\sign Z,
\end{equation}
where $\D^+ (\zp)$ is the spin Dirac operator on the manifold $\zp$ with cylindrical end obtained by setting $X = S^1 \times Y$ in formula \eqref{E:zp}. In essence, $n(Y,h,\s)$ is a special case of the correction term \eqref{E:w} and, like the latter, it is independent of the arbitrary choices in its definition. 

\begin{lem}\label{L: reducible grading}
 For any $i\geq 0$, we have $\gr^{(2)}([\theta_{i}])=0$ and $\gr^{\mathbb{Q}}([\theta_{i}])=-2n(Y,h,\mathfrak{s})+2i$.
\end{lem}
\begin{proof}
This follows directly from the definition of $\gr^{(2)}$ and $\gr^{\mathbb{Q}}$ (see the discussion at the end of \cite[Page 421]{KM}). The only non-trivial point is the use of Assumption \ref{small perturbation}\,(c) to ensure that $\D_{B_{0},\mathfrak{q}}(Y,h)$ can be deformed into $\D_{B_{0}}(Y,h)$ without acquiring non-zero spectral flow, which allows one to compute the grading with zero perturbation.
\end{proof}

%%%%%%%%%%%%%%%%%%%%%%%%%%%%%%%%%%%%%%%%%%%%%%

\section{Proof of Theorem \ref{T:main}}\label{S:main}
Our proof will use the neck stretching operation which is well-known in gauge theory; we will use its non-separating version. 

%%%%%%%%%%%%%%%%%%%%%%%%%%%%%%%%%%%%%%%%%%%%%%

\subsection{Manifolds with long necks}
Let $X$ be a spin rational homology $S^1 \times S^3$ and $Y \subset X$ a rational homology sphere Poincar\'e dual to the choice of homology orientation $1 \in H^1(X;\Z)$. The spin structure on $X$ induces a spin structure $\s$ on $Y$. %which will often be suppressed in our notations in this section% 
Let $h$ be a metric on $Y$ satisfying Assumption \ref{generic metric on Y}, and extend it to a metric $g$ on $X$ which takes the form $g = dt^2 + h$ in a product region $[-\ep,\ep] \times Y$, $\ep > 0$. Given a real number $R > 0$, consider the spin manifold `with long neck'
\begin{equation}\label{E: XR}
X_R = W\,\cup\,([0,R] \times Y)
\end{equation}
obtained by cutting $X$ open along $\{ 0\} \times Y$ and gluing in the cylinder $[0,R] \times Y$ along the two copies of $Y$. We also consider the non-compact manifold
\begin{equation}\label{E: Winfinity}
W_{\infty}=((-\infty,0]\times Y)\,\cup\,W\,\cup\,([0,+\infty)\times Y)
\end{equation}
with two product ends. The metric $g$ induces metrics on $X_{R}$ and $W_{\infty}$, which will be denoted respectively by $g_{R}$ and $g_{\infty}$. 

\begin{assum}\label{A:Dirac over W}
The metric $g$ on $X$ has the form $g = dt^2 + h$ in a product region $[-\ep,\ep] \times Y$, $\ep > 0$, and makes the spin Dirac operator
\begin{equation}\label{E: Dirac over W}
\D^+ (W_{\infty},g_{\infty}):\; L^2_1\,(W_{\infty};\,\S^+)\longrightarrow L^2(W_{\infty};\,\S^-)
\end{equation}
invertible.
\end{assum}

\noindent
The existence of metrics satisfying Assumption \ref{A:Dirac over W} will be proved in Theorem \ref{T:two}; see also Remark \ref{R:two}. Note that, once Assumption \ref{A:Dirac over W} is satisfied, the metric $h$ that shows up in its statement will automatically satisfy Assumption \ref{generic metric on Y}.

The proof of Theorem \ref{T:main} will rely on the following two theorems about manifolds with long necks, whose proofs occupy Section \ref{S:correction} to Section \ref{S:eta}.

\begin{thm}\label{T:Eta long neck}
Let $g$ be a metric on $X$ satisfying Assumption \ref{A:Dirac over W}. Then, for all sufficiently large $R$, the correction term $w(X_R,g_R)$ in formula \eqref{E:wxg} is well-defined, and we have the equality (see (\ref{E:correction term Y}))
\[
w(X_R, g_R)\; =\; n(Y,h,\s).
\]
\end{thm}

\begin{thm}\label{T: SW=Lefschetz}
Let $g$ be a metric on $X$ satisfying Assumption \ref{A:Dirac over W}. Then, for all sufficiently large $R$ and all sufficiently small perturbations $\hat{\mathfrak p}_R$ (defined in Section \ref{S: perturbation}) which make the moduli space $\M(X_R,g_R,\hat{\mathfrak p}_R)$ regular, we have the equality
\[
\#\,\M(X_R,g_R,\hat{\mathfrak p}_R) = - \Lef\, (W_*: C^o \to C^o).
\]
\end{thm}

%%%%%%%%%%%%%%%%%%%%%%%%%%%%%%%%%%%%%%%%%%%%%%%

\subsection{The proof}
Let $g$ be a metric on $X$ satisfying Assumption \ref{A:Dirac over W}. Fix a large number $N\in -2n(Y,h,\s) + 2\Z$ and consider the truncated monopole chain complex $(\widecheck{C}_{\leq N},\widecheck{\partial}_{\leq N})$ for $Y$. It is generated by the irreducible critical points and the boundary stable reducible critical points of grading $\le N$. Then we have a short exact sequence
\[
\begin{CD}
0 @>>> C^s_{\leq N} @>>> \widecheck{C}_{\leq N} @>>> C^o @>>> 0,
\end{CD}
\]
where the chain complex $C^s_{\leq N}$ is generated by the boundary stable reducible critical points of grading $\le N$ and has trivial differential. By Lemma \ref{L: reducible grading}, each $m\in -2n(Y,h,\s) + 2\mathbb{Z}$ with $-2n(Y,h,\s) \leq m\leq N$ contributes a generator to $C^s_{\leq N}$. Therefore, we have
\smallskip
\[
\dim\, (C^s_{\leq N})\; =\; \frac{N + 2n(Y,h,\s)} 2 + 1.
\]
The cobordism $W$ induces chain maps $W_*$ on the three chain complexes in the above exact sequence making the following diagram commute
\medskip
\[
\begin{CD}
0 @>>> C^{s}_{\leq N} @>>> \widecheck{C}_{\leq N} @>>> C^{o} @>>> 0 \\
@. @VV W_* V @VV W_* V @VV W_* V @. \\
0 @>>> C^{s}_{\leq N} @>>> \widecheck{C}_{\leq N} @>>> C^{o} @>>> 0
\end{CD}
\]

\smallskip\noindent
With the obvious abuse of notations, the Lefschetz numbers of the three maps $W_*$ in this diagram are therefore related by the equation
\[
\Lef (\widecheck{C}_{\leq N})\; =\; \Lef (C^{s}_{\leq N})\, +\, \Lef (C^{o}).
\]

\begin{lem}\label{L:id}
The restriction of $W_*$ to the chain complex $C^{s}_{\leq N}$ is the identity map.
\end{lem}

\begin{proof}
This is essentially proved in \cite[Proposition 39.1.2]{KM}. The result is stated there for homology but it holds as well for the chain complex because the boundary map is trivial for grading reasons.
\end{proof}

By Lemma \ref{L: reducible grading}, the $\Z/2$ gradings of the generators of $C^{s}_{\leq N}$ are all zero. %\smargin{Explain in what sense: they are not integers. Jianfeng: corrected} 
Then Lemma \ref{L:id} implies that $\Lef(C^{s}_{\leq N}) = \dim(C^{s}_{\leq N})$ and therefore
\smallskip
\begin{equation}\label{total lefs}
\Lef (\widecheck{C}_{\leq N})\; =\; \frac{N+2n(Y,h,\s)}{2}+1\, +\, \Lef (C^{o}).
\end{equation}

On the other hand, for all sufficiently large $N$, the group $\hmred (Y,\s)$ can be identified with the cokernel of the map
$$
i_{\leq N}: \widebar{HM}_{\leq N}(Y,\s) \longrightarrow \widecheck{HM}_{\leq N}(Y,\s).
$$
Therefore, we have the commutative diagram
\medskip
\[
\begin{CD}
0 @>>> \im(i_{\leq N}) @>>> \widecheck{HM}_{\leq N}(Y,\s) @>>>\hmred (Y,\s) @>>> 0 \\
@. @VV W_* V @VV W_* V @VV W_* V @. \\
0 @>>> \im(i_{\leq N}) @>>> \widecheck{HM}_{\leq N}(Y,\s) @>>>\hmred (Y,\s) @>>> 0
\end{CD}
\]

\medskip\noindent
with the exact rows, which gives us the identity
\[
\Lef(\widecheck{HM}_{\leq N}(Y,\s))=\Lef(\im(i_{\leq N}))+\Lef (\hmred(Y,\s)).
\]
Since $\im(i_{\leq N})$ is a finite length $U$-tail whose top grading is $N$ and lowest grading is $-2h(Y,\s)$, we have
$$
\dim (\im(i_{\leq N}))\; =\; \frac{N+2h(Y,\s)}{2}+1.
$$
The restriction of $W_*$ on $\im(i_{\leq N})$ is the identity map by Lemma \ref{L:id}, therefore, $\Lef(\im i_*) = \dim (\im i_*)$ and
\begin{equation}\label{lefs on U-tail}
\Lef(\widecheck{HM}_{\leq N}(Y,\s)) = \frac{N+2h(Y,\s)}{2} + 1+ \Lef (\hmred(Y,\s)).
\end{equation}
Combining (\ref{total lefs}) and (\ref{lefs on U-tail}) with the fact that the Lefschetz number of a chain map equals the Lefschetz number of the induced map on homology, we obtain the identity
\smallskip
$$
\frac{N+2n(Y,h,\s)}{2}+1+\Lef(C^{o}) = \frac{N+2h(Y,\s)}{2}+1 + \Lef(\hmred(Y,\s))
$$
and, after simplification,
 \[
- \Lef(C^o) - n(Y,h,\s) + h(Y,\s) = -\Lef(\hmred(Y,\s)).
 \]
The proof is now complete because it follows from Theorem \ref{T:Eta long neck} and Theorem \ref{T: SW=Lefschetz} that, for all sufficiently large $R$ and the small perturbation $\hat{\mathfrak{p}}_R$,
$$
\lsw(X) = \#\,\M(X_R,g_R,\hat{\mathfrak{p}}_R) - w\,(X_R,g_R) = -\Lef(C^o) - n(Y,h,\s).
$$

 %%%%%%%%%%%%%%%%%%%%%%%%%%%%%%%%%%%%%%%%%%%%

\section{Proof of Theorem \ref{T:psc}}\label{S:Theorem B}
In this section we prove Theorem \ref{T:psc}, which is an application of Theorem \ref{T:main} to the question of existence of metrics of positive scalar curvature. 

Let $X$ be as in the statement of Theorem \ref{T:psc} and suppose that it admits a metric of positive scalar curvature. According to a theorem of Schoen and Yau~\cite{schoen-yau:psc}, the Poincar\'e dual to the generator of $H^1 (X; \mathbb{Z})$ can be realized by an embedded manifold $M \subset X$ which admits a metric of positive scalar curvature. Since the first Chern class of the spin structure $\s_M$ induced on $M$ vanishes, it follows from \cite[Proposition 36.1.3]{KM} that $\hmred (M,\s_M) = 0$. Note that the manifold $M$ need not be a rational homology sphere, however, we will prove that its existence implies that $\Lef\,(W_*: \hmred(Y,s) \to \hmred(Y,\s)) = 0$.

Our proof will adapt the argument of Fr{\o}yshov~\cite[Section 13]{F} that shows the well-definedness of the Lefschetz number. Since $M$ generates $H_3 (X;\Z)$, the standard covering space theory implies that the manifold $M$ lifts to the infinite cyclic cover $\tilde X$, and that this lift can be arranged to be disjoint from a copy of $Y$. It follows that, for some $k\geq 0 $, the manifold 
\[
W^{(k)}\; =\; W \cup_Y W \cup_Y \ldots \cup_Y W\quad\text{($k$ times)}
\]
contains a copy of $M$ separating its two boundary components. Therefore, the map
\[
(W_*)^{k} = W^{(k)}_*: \hmred(Y,\s)\longrightarrow \hmred(Y,\s)
\]
factors through $\hmred(M,\s_M) = 0$ making $W_*$ nilpotent. It then follows that the trace of $W_*$ vanishes in each $\Z/2$ grading, and that the Lefschetz number of $W_*$ must therefore be zero.

%%%%%%%%%%%%%%%%%%%%%%%%%%%%%%%%%%%%%%%%%%%%%%%

\section{Proof of Theorem \ref{T:Theta} and \ref{T:Theta2}}\label{S:Theorem C}
We now prove Theorems \ref{T:Theta} and \ref{T:Theta2} from the introduction which assert that, in a number of circumstances, a homology sphere must have infinite order in the homology cobordism groups $\Theta^3_{\Z}$ or $\Theta^3_{\Z/2}/\Theta^3_L$. The proofs can be found at the end of Section \ref{S:obstruction}. The part of Section \ref{S:Theorem C} after that is dedicated to examples and the proofs of Corollaries \ref{C: surgery infinite order} and \ref{C: thin group infinite index}.

%%%%%%%%%%%%%%%%%%%%%%%%%%%%%%%%%%%%%%%%%%%%%%%

\subsection{A homology cobordism obstruction from $\hmred$}\label{S:obstruction}
Let $(Y,\s)$ be a rational homology sphere with a $\spinc$ structure  then $\hmred (Y,\s)$ is graded by the rational numbers, and we define the support of $\hmred (Y,\s)$ by the formula
\[
S(Y,\s)\, =\, \{\,a\in \Q\,\mid  \hmred_{a}(Y,\s)\neq 0\,\}.
\]
For the rest of this section, whenever we are considering the unique spin structure $\s$ on a $\mathbb{Z}/2$ homology sphere $Y$, we will usually drop $\s$ from our notations. In particular, the notations $\rho(Y)$, $S(Y)$ and $h(Y)$ will be used, respectively, to denote the Rohlin invariant, the support of reduced monopole Floer homology, and the Fr{\o}yshov invariant for the unique spin structure.

\begin{pro}\label{cobordism obstruction}
Let $Y_1$ and $Y_2$ be $\mathbb{Z}/2$ homology spheres such that $S(Y_1)\,\cap\,S(Y_2) = \emptyset$ and $ h(Y_1)\neq\rho(Y_1) \pmod 2.$ %, $j = 1, 2$.
% where $\rho$ stands for the Rohlin invariant. 
Then $Y_1$ is not $\mathbb{Z}/2$ homology cobordant to $Y_2$.
\end{pro}

\begin{proof}
Suppose to the contrary that we have an $\mathbb{Z}/2$ homology cobordism $W_1$ from $Y_1$ to $Y_2$. It carries a unique spin structure, which restricts to $\s_{i}$ on $Y_{i}$. Reversing the orientation, we obtain a spin cobordism $-W_1$ from $Y_2$ to $Y_1$. Now, consider the composite cobordism
\[
W = W_1\,\cup_{Y_2} (-W_1)
\]
from $Y_1$ to itself. The morphism
\[
(W_1)_*: \hmred(Y_1,\s_{1})\longrightarrow \hmred(Y_2,\s_{2})
\]
induced by the cobordism $W_1$ as in \eqref{E:W*} preserves the absolute grading. Since the intersection $S(Y_1)\,\cap\,S(Y_2)$ is empty, we conclude that the map $(W_1)_*$ must be zero. By functoriality, the map 
\[
W_*: \hmred(Y_1,\s_{1})\longrightarrow \hmred(Y_1,\s_{1})
\]
is also zero; in particular, its Lefschetz number vanishes. Let $X$ be the homology $S^1\times S^3$ obtained by identifying the two boundary components of $W$ via the identity map. Then, using Theorem \ref{T:main}, we obtain
\[
\lsw(X) = -h(Y_1) - \Lef\,(W_*) = - h(Y_1)\, \neq\, \rho(Y_1)\pmod 2.
\]
This contradicts \cite[Theorem A]{MRS1} which asserts that $\lsw (X)$ equals the Rohlin invariant of $X$ modulo 2.
\end{proof}
\begin{rmk}
For a $\mathbb{Z}/2$ homology sphere $Y$ with spin structure $\s$, the condition $ h(Y)\neq\rho(Y) \pmod 2$ is equivalent to condition that $\operatorname{dim}_{\mathbb{Q}}\hmred(Y,\s)$ is odd.
\end{rmk}
The following definition is a slight generalization of the one given in the introduction.
\begin{defi}\label{defi:h-positive/h-negative} Let $Y$ be a rational homology sphere with $\spinc$-structure $\s$. 
\begin{itemize}
\item We will say that $(Y,\s)$ is $h$-positive if $\hmred(Y,\s)$ is supported in degrees $\geq -2h(Y,\s)$, that is, $S(Y,\s) \subset [-2h(Y,\s),+\infty)$. 
\item  We will say that $(Y,\s)$ is $h$-negative if $\hmred(Y,\s)$ is supported in degrees $\leq -2h(Y,\s)-1$, that is, $S(Y,s)\subset (-\infty,-2h(Y,\s) - 1]$. 
\end{itemize}
\end{defi}

Because Heegaard Floer homology is (at present) easier to compute, we would prefer to use it in place of the monopole Floer homology in our calculations whenever possible. In fact, these two theories are known to be isomorphic. Furthermore, by combining the main results of \cite{Gripp,Gripp-Huang,Gardiner}, the absolute $\Q$-gradings in the two theories coincide. Therefore, the relation $d(Y,\s) = -2h(Y,\s)$ between the Fr\o yshov invariant and the Heegaard Floer correction term holds for any rational homology spheres. It then follows that the property of $(Y,\s)$ being $h$-positive (respectively, $h$-negative) can be characterized by saying that  the reduced Heegaard Floer homology $HF^{\text{red}} (Y,\s)$ is supported in degrees $\geq d(Y,\s)$ (respectively, in degrees $\leq d(Y,\s)-1$). 

%\smargin{\textcolor{red}{It is actually possible to deduce the result $d(Y,\s)=-2h(Y,\s)$ from Theorem 3.4 of  Rustamov's paper ``Surgery formula for the renormalized Euler characteristic of Heegaard Floer homology.'' However, we might not want to use this because there are several subtle technical points.}} %(it is known to hold for integral homology spheres, by comparison with the Casson invariant; see \cite[Theorem 1.3]{oz:boundary} and~\cite[Theorem 5]{F}). On the other hand, the isomorphism between the two theories discussed in~\cite[Main Theorem]{kutluhan-lee-taubes:HFSW-I} does preserve the relative $\Z$-grading. 

\begin{ex}\label{seifert and figure eight}
By a positive orientation on a Seifert fibered homology sphere we will mean the canonical orientation of it as a link of singularity. Using the graded roots model for computing Heegaard Floer homology \cite[Section 11.13]{Nemethi}, one can show that all Seifert fibered homology spheres with positive orientation are $h$-negative, while the ones with negative orientation are $h$-positive. According to \cite[Proposition 8.3]{oz:boundary}, the homology sphere obtained by $1/n$ surgery on the figure-eight knot is $h$-negative if $n\geq 0$ and $h$-positive if $n\leq 0$. (Note that $L$-spaces are both $h$-positive and $h$-negative.)
\end{ex}

\begin{lem}\label{orientation reversal}
A $\mathbb{Z}/2$ homology sphere $Y$ is $h$-positive if and only if the $\mathbb{Z}/2$ homology sphere $-Y$ obtained from $Y$ by orientation reversal is $h$-negative.
\end{lem}

\begin{proof}
We use $\s$ and $-\s$ to denote the spin structure on $Y$ and $-Y$ respectively.  Recall from \cite[(3.4)]{KM} that there is a long exact sequence\footnote{While the grading convention for $\widecheck{HM}$ and $\overline{HM}$ are consistent with those for $HF^{+}$ and $HF^{\infty}$, respectively, the grading convention for $\widehat{HM}$ differs from that for $HF^{-}$ by $1$. For example, the generator of $\widehat{HM}(S^3)$ as a $\mathbb{Q}[U]$-module has grading $-1$, while the generator of $HF^-(S^3)$ has grading $-2$. We follow here the conventions of \cite{kmos:lens}. }
\[
\begin{CD}
\cdots\to\widehat{HM}_{a+1}(Y,\s) @> p_{Y,a} >> \overline{HM}_{a}(Y,\s) @> i_{Y,a} >> \widecheck{HM}_{a}(Y,\s) @> j_{Y,a} >> \widehat{HM}_{a}(Y,\s)\to\cdots,
\end{CD}
\]
therefore, the set $S(Y)$ can be equivalently defined as $S(Y) = \{\,a\in \Q\mid j_{Y,a}\neq 0\,\}$. Under the natural duality isomorphisms
\begin{equation*}
\begin{split}
&\widecheck{HM}_{a}(Y,\s)\; \cong\; (\widehat{HM}_{-1-a}(-Y,-\s))^{*},\\
&\widehat{HM}_{a}(Y,\s)\; \cong\; (\widecheck{HM}_{-1-a}(-Y,-\s))^{*},
\end{split}
\end{equation*}
the dual map
\[
(j_{Y,a})^{*}:\widehat{HM}_{-1-a}(-Y,-\s)\longrightarrow \widecheck{HM}_{-1-a}(Y,-\s)
\]
is exactly the map $j_{-Y,-1-a}$, therefore, $S(-Y) = \{-1-a\mid a\in S(Y)\}$. The result now follows because $h(Y) = -h(-Y)$.
\end{proof}

\begin{lem}\label{close under connected sum}
For $\mathbb{Z}/2$ homology spheres $Y_{1},Y_{2}$, the connected sum $Y_{1}\,\#\,Y_{2}$ is $h$-positive (respectively, $h$-negative) if $Y_1$ and $ Y_2$ are both $h$-positive (respectively, $h$-negative).
\end{lem}

\begin{proof}
Because of Lemma \ref{orientation reversal}, we only need to treat the $h$-negative case. Let us introduce the notations 
\begin{equation}\label{U-tail}
\T_{a}(b) = (\mathbb{Q}[U]/U^{b})[-a]\quad\text{and}\quad
\mathcal{T}^{-}_{a} = (\mathbb{Q}[U])[-a],
\end{equation}
where $1\in \mathbb{Q}[U]$ has degree 0 and for a graded module $M$, we follow the convention $(M[c])_k = M_{k+c}$ for the grading shift. Let $\s_{j}$ ($j=1,2$) be the spin structure on $Y_{j}$. For both $Y_1$ and $Y_2$, we have a (non-canonical) splitting of the $\Q[U]$-modules,
\begin{align*}
\widehat{HM}(Y_1,\s_{1}) &\, =\, \T^{-}_{-2h(Y_1)-1}\oplus \T_{a_{1}}(b_{1})\oplus\ldots\oplus \T_{a_{k}}(b_{k})
\quad\text{and} \\
\widehat{HM}(Y_2,\s_{2}) &\, =\, \T^{-}_{-2h(Y_2)-1}\oplus \T_{c_{1}}(d_{1})\oplus\ldots\oplus \T_{c_{\ell}}\,(b_{\ell}).
\end{align*}
By combining the connected sum formula for Heegaard Floer homology \cite[Proposition 6.2]{Ozsvath-Szabo2} with the identification between monopole Floer homology and Heegaard Floer homology (or alternatively, by directly using the connected sum formula in \cite{BMO, baldwin-bloom:sum}) we obtain \footnote{Only the relatively graded version of this formula appears in \cite{Ozsvath-Szabo2}. One obtains the absolutely graded version with the help of the Fr\o yshov invariant, which is additive under connected sum.}
\begin{equation*}\begin{split}
\widehat{HM}(Y_1\# Y_2,\s_{1}\#\s_{2})\, =&\\ \, (\widehat{HM}(Y_1,\s_{1})& \otimes_{\,\mathbb{Q}[U]}  \widehat{HM}(Y_2,\s_{2}))[-1]\, \oplus\, (\Tor_{\,\mathbb{Q}[U]}\,(\widehat{HM}(Y_1,\s_{1}),\widehat{HM}(Y_2,\s_{2}))[-2].\end{split}
\end{equation*}
We will now trace the contributions of each of the summands of $\widehat{HM}(Y_1)$ and $\widehat{HM}(Y_2)$ to $\widehat{HM}(Y_1 \# Y_2,\s_{1}\#\s_{2})$:
\begin{itemize}
\item The tensor product 
\[
(\T^{-}_{-2h(Y_1)-1} \otimes_{\,\Q[U]} \T^{-}_{-2h(Y_2)-1})[-1]\, = \, \T^{-}_{-2h(Y_1)-2h(Y_2)-1}
\]
contributes the infinite $U$-tail to $\widehat{HM}(Y_1\# Y_2,\s_{1}\#\s_{2})$;
\item Each of the tensor products
\begin{align*}
& (\T_{a_i}(b_i)\, \otimes_{\,\Q[U]} \T^{-}_{-2h(Y_2)-1})[-1]\, =\, \T_{-2h(Y_2)+a_i}(b_i),\\
& (\T^{-}_{-2h(Y_1)-1} \otimes_{\,\Q[U]} \T_{c_j}(d_j))[-1]\, =\, \T_{-2h(Y_1)+c_j}(d_j), \\
& (\T_{a_i}(b_i) \otimes_{\,\Q[U]} \T_{c_j}(d_j))[-1]\; =\; \T_{a_i+c_j+1}(\min(b_i,d_j))
\end{align*}
contributes a summand to the kernel of the map $p_{\,Y_1\#Y_2,*}$;
\item To compute $\Tor\,(\T_{a_{i}}(b_{i})\otimes_{\,\Q[U]}\T_{c_{j}}(d_{j}))[-2]$, consider the following \emph{grading preserving} free resolution
\[
\begin{CD}
0 @>>> \T^{-}_{a_i - 2b_i} @> U^{b_i} >> \T^{-}_{a_i} @>>> \T_{a_i}(b_i) @>>> 0.
\end{CD}
\]
By taking tensor product with $\T_{c_j}(d_j)[-2]$ and computing homology of the resulting chain complex, we obtain
\[
\Tor\,(\T_{a_i}(b_i)\otimes_{\,\Q[U]} \T_{c_j}(d_j))[-2]\, =\, \T_{a_i + c_j + 2 - 2\max(b_i,d_j)}(\min(b_i,d_j)),
\]
which contributes another summand to the kernel of $p_{Y_1\#Y_2,*}$.
\end{itemize}
Since $Y_1$ and $Y_2$ are both $h$-negative, we have $a_i\leq -2h(Y_1) - 1$ and $c_j\leq -2h(Y_2) - 1$. Also note that $b_i$, $d_j\geq 1$. It is now easy to check that all the summands in $\ker p_{Y_1\#Y_2,*}$ are supported in degrees at most $-2h(Y_1) - 2h(Y_2) - 1 = -2h(Y_1\# Y_2) - 1$. Therefore, $Y_1\# Y_2$ is $h$-negative, and the lemma is proved.
\end{proof}

\begin{cor}\label{C:first part}
Let $Y_1$ be an $h$-positive (respectively, $h$-negative) $\mathbb{Z}/2$ homology sphere, and suppose that $h(Y_1)\neq \rho(Y_1)\pmod 2$. Then $Y_1$ is not $\mathbb{Z}/2$ homology cobordant to any $h$-negative (respectively, $h$-positive) $\mathbb{Z}/2$ homology sphere $Y_2$.
\end{cor}

\begin{proof}
Suppose to the contrary that $Y_1$ is $\mathbb{Z}/2$ homology cobordant to an $h$-negative $\mathbb{Z}/2$ homology sphere $Y_2$. Since both $h$ and $\rho$ are invariants of $\mathbb{Z}/2$ homology cobordism, 
\[
h(Y_1) = h(Y_2)\quad\text{and}\quad \rho(Y_1) = \rho(Y_2) \pmod 2
\]
and in particular $h(Y_2) \neq \rho(Y_2)\pmod 2$. Let $h = h(Y_1) = h(Y_2)$ then $S(Y_1)\,\subset\, [-2h,+\infty)$ and $S(Y_2)\,\subset\, (-\infty,-2h-1]$ so that $S(Y_1)\,\cap\, S(Y_2) = \emptyset$. This contradicts Proposition \ref{cobordism obstruction}. 
\end{proof}

\begin{cor}\label{h-positive not cobordism to h-negative}
Let $Y_1,Y_2,\cdots,Y_n$ be $h$-positive $\mathbb{Z}/2$ homology spheres, and $a_1,a_2,\cdots, a_n$ positive integers. Suppose that at least one of the $Y_{j}$ satisfies the condition $h(Y_j) \neq \rho(Y_j) \pmod 2$. Then the connected sum
\begin{equation}\label{connected sum manifold}
(\#_{a_1} Y_1)\,\# \ldots \#\,(\#_{a_n}Y_n)
\end{equation}
cannot be $\mathbb{Z}/2$ homology cobordant to $S^3$. A similar result holds for $h$-negative homology spheres.
\end{cor}

\begin{proof}
Suppose to the contrary that the connected sum (\ref{connected sum manifold}) is $\mathbb{Z}/2$ homology cobordant to $S^3$. Without loss of generality, we may assume that $h(Y_1)\neq \rho(Y_1) \pmod 2$. Then the manifold $-Y_1$, which is $h$-negative by Lemma \ref{orientation reversal}, is $\mathbb{Z}/2$ homology cobordant to the manifold 
\[
(\#_{a_1-1}Y_1)\,\# \ldots \#\,(\#_{a_n} Y_n),
\]
which is $h$-positive by Lemma \ref{close under connected sum}. This contradicts Corollary \ref{C:first part}. 
\end{proof}

\begin{proof}[Proof of Theorem \ref{T:Theta}]
Suppose to the contrary that $Y$ is of finite order in $\Theta^{3}_{\mathbb{Z}}$. Then $h(Y) = 0\neq \rho(Y) \pmod 2$. This contradicts Corollary \ref{h-positive not cobordism to h-negative} since a $\mathbb{Z}$ homology cobordism is also a $\mathbb{Z}/2$ homology cobordism.
\end{proof}

\begin{proof}[Proof of Theorem \ref{T:Theta2}]
Suppose to the contrary that $Y$ is of finite order in $\Theta^3_{\Z/2}/\Theta^3_L$. Then there exists an integer $n>0$ and an $L$-space $Y_{1}$ which is a $\mathbb{Z}/2$ homology sphere, such that $(\#_{n}Y)\# Y_{1}$ is $\mathbb{Z}/2$ homology cobordant to $S^{3}$. This contradicts Corollary \ref{h-positive not cobordism to h-negative} since $Y_{1}$ is both $h$-positive and $h$-negative. 

\end{proof}

\begin{proof}[Proof of Corollary \ref{C: thin group infinite index}]
Using the spectral sequences of Ozsv\'ath-Szab\'o \cite{oz: spectral sequence} and Bloom \cite{bloom: spectral sequence} one can easily see that, for any Khovanov-homology thin knot $K_{1}$, the double branched cover $\Sigma(K_{1})$ is an $L$-space over $\mathbb{Z}/2$. The universal coefficient theorem then implies that $\Sigma(K_{1})$ is also an $L$-space over $\mathbb{Q}$. Note that the double branched cover of $I\times S^3$ with branch set a smooth concordance between two knots is a $\Z/2$ homology cobordism between the double branched covers of the knots. The result now follows from Theorem \ref{T:Theta2}.  
\end{proof}

%%%%%%%%%%%%%%%%%%%%%%%%%%%%%%%%%%%%%%%%%%%%%%%

\subsection{Surgery on knots}
In this section, we will use the rational surgery formula of Ozsv\'ath and Szab\'o~\cite{oz:q-surgery} to obtain a sufficient condition for a surgered manifold to be $h$-positive. Corollary \ref{C: surgery infinite order} will be proved in the next section by checking this condition and applying Theorem \ref{T:Theta}.

Let $K$ be a knot in $S^3$. Given co-prime integers $p$ and $q$, denote by $S^3_{p/q} (K)$ the manifold obtained by the $p/q$ surgery on $K$. For any $p > 0$, the manifold $S^3_p (K)$ is a rational homology sphere. It admits exactly $p$ distinct $\spin^c$ structures, which can be naturally identified \cite{oz:z-surgery} with the elements of $\Z/p$. The $\spin^c$ structure on $S^3_p (K)$ corresponding to an integer $s$ will be denoted by $[s]$.

\begin{thm}[Ozsv\'ath-Szab\'o \cite{oz:knots}, Rasmussen \cite{rasmussen:hfk}]\label{stable heegaard floer} For all sufficiently large $p > 0$ and all $s\in \Z$ with $|s|\leq p/2$ the Heegaard Floer homology group $HF^+ (S^3_p (K),[s])$, viewed as a relatively graded $\mathbb{Q}[U]$-module, is independent of $p$.
\end{thm}

\begin{pro}\label{surgery h-positive}
Suppose that, for all sufficiently large $p > 0$ and all $s\in \Z$ with $|s|\leq p/2$, the rational homology sphere $(S^3_p (K),[s])$ is $h$-positive. Then, for any positive integer $n$, the integral homology sphere $S^3_{1/n} (K)$ is $h$-positive.
\end{pro}

\begin{proof}
This is a straightforward corollary of Ozsv\'ath-Szab\'o's rational surgery formula \cite{oz:q-surgery}. For the sake of completeness, we will sketch the argument here and refer the reader to \cite{ni-wu:cosmetic} for a concise summary. (See also \cite{Hom-Karakurt-Lidman}, which treats a similar situation as here). 

For a sufficiently large $p > 0$ and any $s \in \Z$, consider $n$ copies $A^+_{s,1},\ldots, A^+_{s,n}$ of the Heegaard Floer homology $HF^+ (S^3_p (K), [s])$ (this notation is justified by Theorem \ref{stable heegaard floer}), and $n$ copies $B^+_{s,1},\ldots, B^+_{s,n}$ of the module $\T^+ = \Q [U,U^{-1}]\big/(U\cdot \Q[U])$. By \cite[Theorem 1.1]{oz:q-surgery} and \cite[Remark 2.3]{Hom-Karakurt-Lidman}, one can recover the Heegaard Floer homology $HF^+ (S^3_{1/n} (K))$ as the homology of the mapping cone of a certain map
\[
\Phi_{1/n}: \left(\mathop{\bigoplus}\limits_{s\in \Z,\,1\leq i\leq n} A^+_{s,i}\right)\;\longrightarrow\; \left(\mathop{\bigoplus}\limits_{s\in \Z,\,1\leq i\leq n} B^+_{s,i}\right).
\]
In practice, one can take a large enough integer $N$ and instead consider the mapping cone of the truncated map
\[
\Phi^{N}_{1/n}:A^{+}_{-N-1,n}\,\oplus\;\left(\mathop{\bigoplus}\limits_{-N\leq s\leq N,\,1\leq i\leq n} A^+_{s,i}\right)\;\longrightarrow\; \left(\mathop{\bigoplus}\limits_{-N\leq s\leq N,\ 1\leq i\leq n}B^+_{s,i}\right).
\]
This map is surjective so one has an isomorphism
\[
\ker \Phi^N_{1/n} \; \cong \; HF^{+}(S^3_{1/n} (K)).
\]
Furthermore, one can impose suitable absolute gradings on $A^+_{s,i}$ such that the above isomorphism preserves the absolute grading. Recall that $A^+_{s,i}$ admits a splitting $A^+_{s,i} = \T^+\,\oplus\, A^{\red}_{s,i}$. Let $a_{s,i}$ be the absolute grading of the bottom term in $\mathcal{T}^{+}\subset A^{+}_{s,i}$. Then 
\begin{align*}
& a_{0,i}\; =\; d(S^3_{1/n}(K))\quad\text{for all $1\leq i\leq n$, and} \\
& a_{s,i}\; \geq\; d(S^3_{1/n}(K))\quad\text{for all $s\in \Z$ and all $1\leq i\leq n$}
\end{align*}
by \cite[(2.1)$\, - \,$(2.4)]{Hom-Karakurt-Lidman}. Using the assumption that the $\spin^c$ structure $[s]$ is $h$-positive for all sufficiently large $p > 0$, we conclude that $A^+_{s,i}$ is supported in degrees $\geq a_{s,i}$. This implies that $\ker \Phi^N_{1/n}$ is supported in degrees greater than or equal to
\[
\mathop{\inf}\limits_{s,i}\, (a_{s,i})\; =\; d(S^3_{1/n}(K)).
\]
Therefore, the integral homology sphere $S^3_{1/n} (K)$ is $h$-positive.
\end{proof}
\subsection{$L$-space knots and thin knots}
In this section, we will apply Proposition \ref{surgery h-positive} to the classes of $L$-space knots and Floer homology thin knots, and then prove Corollary \ref{C: surgery infinite order}. 

Recall that a knot $K$ is called an $L$-space knot (over the rationals) if there is a rational number $p/q > 0$ such that the manifold $S^3_{p/q} (K)$ is an $L$-space, that is, 
\[
\widehat{HF}_{p/q}\,(S^3_{p/q}(K))\, =\, \Q^{\,p}.
\]
This condition actually implies that $S^3_{p'/q'} (K)$ is an $L$-space for all $p'/q'\geq p/q$. Proposition \ref{surgery h-positive} has the following corollary.

\begin{cor}\label{surgery on L-space knot h-positive}
Let $K$ be an $L$-space knot. Then $S^3_{1/n} (K)$ is $h$-positive for all $n>0$.
\end{cor}

\begin{rmk}
Using Corollary \ref{surgery on L-space knot h-positive}, one can derive a result similar to Corollary \ref{C: surgery infinite order} for all $L$-space knots. However, this can be proved directly using the Heegaard Floer correction term.
\end{rmk}

Now we turn to the case of surgeries on Floer homology thin knots. We will need a number of constructions involving Heegaard Floer homology of knots, for which we refer to the original paper~\cite{oz:knots}, as well as to the survey \cite{manolescu:knot-survey}.

 Recall that, for an even integer $\tau$, a knot $K$ is called Floer homology $\tau$-thin (over the rationals) if the bigraded knot Floer homology group $\widehat{HFK}_{*}(K,* )$ satisfies the  condition 
\begin{equation}\label{thin condition}
\widehat{HFK}_i (K,j ) = 0\quad \text{unless}\quad i = j + \tau/2.
\end{equation}
The following lemma summarizes properties of thin knots that are useful for the application we have in mind.

\begin{lem}\label{property of thin knots}
Let $\tau$ be an even integer, $K$ a Floer homology $\tau$-thin knot, and $n$ a positive integer.  Then
\begin{enumerate}
\item $d (S^3_{1/n} (K)) = 2\min(0,-\lceil -\tau/4 \rceil )$. In particular, $d (S^3_{1/n} (K))\neq 0$ if and only if $\tau<0$;
\item If $\tau>0$ then for all sufficiently large $p > 0$ and all integers $s$ such that $|s| \leq p/2$ the rational homology sphere $(S^3_p (K),[s])$ is $h$-positive.
\end{enumerate}
\end{lem}

\begin{proof}
Recall that the knot Floer complexes are generated by triples $[x,i,j]$ satisfying various conditions, where $i$ and $j$ are integers, and $x$ is an intersection point between Lagrangian tori in the symmetric product of the Heegaard surface. For any $a,b\in \Z$, we will denote by $C\{i=a,j=b\}$ the complex generated by triples $[x,i,j]$ with $i=a$, $j=b$. We will use similar notations for the other complexes. It follows from (\ref{thin condition}) that $H(C\{i=a,j=b\})$ is supported in degree (absolute Maslov grading) $a + b + \tau/2$. Via a basic spectral sequence argument, this implies that
\begin{align*}
& H(C\{\min(i,j-s)\}\geq 0)\;\;\text{is supported in degrees $\geq s + \tau/2$,\; and} \\
& H(C\{\max(i,j-s)\}\leq -1)\,\;\text{is supported in degrees $\leq s + \tau/2 - 2$}.
\end{align*}
With these two facts established, we can prove that 
\[
d(S^{3}_{1}(K))\; =\; 2\min(0,-\lceil -\tau/4 \rceil)
\]
by repeating word for word the proof of \cite[Corollary 1.5]{oz:alternating} (which deals with the special case of an alternating knot $K$). Since $d(S^{3}_{1/n}(K))=d(S^{3}_{1}(K))$ for any $n>0$ (see \cite[Proposition 1.6]{ni-wu:cosmetic}), claim (1) is proved. 

We now turn to claim (2). Since the spin$^{c}$ structures $[s]$ and $[-s]$ are conjugate to each other, one has an isomorphism $HF^{+}(S^3_p (K),[s]) \cong HF^{+}(S^3_p (K),[-s])$. Therefore, it is sufficient to consider the case of $s\geq 0$. Recall from \cite{oz:knots, rasmussen:hfk} that there is an isomorphism
\[
HF^{+}(S^3_p (K),[s])\, \cong\, H(C\{\max(i,j-s)\geq 0\}).
\]
of relatively graded $\mathbb{Q}[U]$-modules. For any integer $a$, denote by $\T^{+}_{a}$ the graded module $(\mathbb{Q}[U,U^{-1}]/U\cdot \mathbb{Q}[U])[-a]$ (cf. (\ref{U-tail})). Then we have a decomposition of absolutely graded $\mathbb{Q}[U]$-modules,
\[
H(C\{\max(i,j-s)\geq 0\})\, \cong\, \T^{+}_{e}\oplus V
\]
for some integer $e$ and a finite dimensional $\mathbb{Q}$-vector space $V$. Consider the short exact sequence
\[
0\rightarrow C\{\max(i,j-s)\geq 0\}\longrightarrow C\{i\geq 0\}\oplus C\{j\geq s\}\longrightarrow C\{\min(i,j-s)\geq 0\}\rightarrow 0.
\] 
Since $H(C\{\min(i,j-s)\geq 0)$ is supported in degrees $\geq s + \tau/2 \geq 1$, we obtain 
\[
H_{\leq -1}(C\{\max(i,j-s)\geq 0\})\, \cong\, H_{\leq -1}(C\{i\geq 0\}) \oplus H_{\leq -1}(C\{j\geq s\})=0,
\]
with the last equality following from the isomorphisms
\[
H(C\{i\geq 0\})\cong HF^{+}(S^{3})\cong \T^{+}_{0}\quad\text{and}\quad H(C\{j\geq s\})\cong HF^{+}(S^{3})[-2s]\cong \T^{+}_{2s}.
\]
Therefore, $H(C\{\max(i,j-s)\geq 0\})$ is supported in degrees $\geq 0$. The proof will be complete once we show that $e=0$. To this end, consider another short exact sequence
\[
\begin{CD}
0\rightarrow C\{\min(-1-i,j-s)\geq 0\} \longrightarrow C\{\max(i,j-s)\geq 0\} @> p >> C\{i\geq 0\}\rightarrow 0.
\end{CD}
\]
Since $C\{\max(-1-i,j-s)\geq 0\}$ is finite dimensional, for any sufficiently large integer $m$ we have isomorphisms 
\[
p_{*}:H_{2m}( C\{\min(i,j-s)\geq 0\})\,\cong\, H_{2m}( C\{i\geq 0\})\,\cong\, \mathbb{Q}.
\]
Let $\xi \in H_{2m}( C\{\min(i,j-s)\geq 0\})$ be any non-zero element. Since $H(C\{i\geq 0\})\cong \T_{0}^{+}$, we have $p_{*}(U^{m}\xi)=U^{m}p_{*}(\xi)\neq 0$. This implies that $U^{m}\xi\neq 0$ and 
\[
e= 2m-2\cdot \max\{n\in \mathbb{Z}\mid U^{n}\xi \neq 0\}\; \leq\; 2m-2m=0.
\]
Because $H(C\{\max(i,j-s)\geq 0\})$ is supported in degrees $\geq 0$ we conclude that $e=0$, which completes the proof. 
\end{proof}

\begin{cor}\label{surgery on thin knot is h-positive}
Let $K$ be a Floer homology $\tau$-thin knot with $\tau > 0$. Then $S^3_{1/n}(K)$ is $h$-positive for all $n > 0$.
\end{cor}

\begin{proof}
This is immediate from Lemma \ref{property of thin knots} (2) and Proposition \ref{surgery h-positive}.
\end{proof}

\begin{proof}[Proof of Corollary \ref{C: surgery infinite order}] 
Since $\arf(K)=1$ and $n$ is odd, it follows from the surgery formula for the Rohlin invariant~\cite{gonzalez-acuna:dehn,saveliev:spheres} that  $\rho(S^3_{1/n}(K)) = 1\pmod 2$. Claim (1) now follows from Example \ref{seifert and figure eight}. To prove claim (2), consider the mirror image $\widebar K$ of the knot $K$. Since $\widebar K$ is quasi-alternating and $S^3_{-1/n}(K) = -S^3_{1/n}(\widebar{K})$, it is sufficient to consider the case of $n > 0$. According to \cite{manolescu-ozsvath:quasi}, any quasi-alternating knot $K$ is Floer homology $\sigma(K)$-thin over $\Z/2$, where $\sigma (K)$ stands for the knot signature. By the universal coefficient theorem, this implies that $K$ is also Floer homology $\sigma(K)$-thin over $\Q$. If $\sigma(K) < 0$, it follows from Lemma \ref{property of thin knots} (1) that $d(S^3_{1/n}(K)) \neq 0$ and hence $S^3_{1/n}(K)$ has infinite order in $\Theta^{3}_{\mathbb{Z}}$. If $\sigma(K) > 0$, it follows from Corollary \ref{surgery on thin knot is h-positive} that $S^3_{1/n}(K)$ is $h$-positive and hence it has infinite order in $\Theta^{3}_{\mathbb{Z}}$.
\end{proof}

%%%%%%%%%%%%%%%%%%%%%%%%%%%%%%%%%%%%%%%%%%%%%%%

\section{The correction term}\label{S:correction}
In this section, we will prove Theorem \ref{T:Eta long neck}. The index theory that will go into our proof is not specific to dimension four, therefore, we will work in more generality than strictly necessary.

Let $X$ be a connected smooth spin compact manifold of dimension $n \equiv 0 \pmod 4$ with a primitive cohomology class $\gamma \in H^1 (X;\mathbb Z)$. Let $Y \subset X$ be a connected manifold Poincar\'e dual to $\gamma$ with the induced spin structure $\s$. Choose a metric $g$ on $X$ which takes the form $g = dt^2 + h$ in a product region $[-\ep,\ep] \times Y$, $\ep > 0$. We will assume that $(Y, \s)$ is a spin boundary and that the $\hat A$--genus of $X$ vanishes; both of these conditions are automatic when $X$ is a homology $S^1 \times S^3$. Given a real number $R > 0$, construct the spin manifold
\begin{equation}\label{E:XRn}
X_R = W\,\cup\,([0,R] \times Y)
\end{equation}
as in \eqref{E: XR} by cutting $X$ open along $\{ 0\} \times Y$ and gluing in the cylinder $[0,R] \times Y$ along the two copies of $Y$. The metric $g$ defines a metric $g_R$ on $X_R$, which lifts to a metric $g_R$ on the infinite cyclic cover of $X_R$ determined by $\gamma$. Following \eqref{E:zp}, denote by $\zp(X_R)$ the manifold with \emph{periodic} end modeled on this infinite cyclic cover, and by $\zp$ and $\wp$ the manifolds with \emph{product} ends modeled on the product $\mathbb R \times Y$ with metric $dt^2 + h$. Note that $\wp$ has two ends, corresponding to the two boundary components of $W$. The metrics will often be suppressed in our notations.

\begin{thm}\label{thm}
Assume that the spin Dirac operator $\D^+(\wp): L^2_1\,(\wp;\,\S^+) \to L^2 (\wp;\,\S^-)$ is an isomorphism. Then for all sufficiently large $R$ the end-periodic operator $\D^+(\zp(X_R))$ is Fredholm of index
\[
\ind \D^+(\zp(X_R)) = \ind \D^+(\zp).
\]
\end{thm}

The existence of metrics on $\wp$ making the operator $\D^+(\wp)$ invertible is addressed in Theorem \ref{T:two}. When applied to a spin $4$-manifold $X$ with the rational homology of $S^1\times S^3$, Theorem \ref{T:Eta long neck} is a straightforward corollary of Theorem \ref{thm}.

%%%%%%%%%%%%%%%%%%%%%%%%%%%%%%%%%%%%%%%%%%%%%%

\subsection{Preliminaries}
We begin by proving two technical lemmas which will be used later in the argument.

\begin{lem}\label{norm control}
Suppose $A: H\rightarrow H'$ is a surjective bounded operator between Hilbert spaces. Then there exists a constant $C > 0$ such that for any vector $v\in H'$ one can find a vector $u\in A^{-1}(v)$ with $\|u\|_H \leq C\cdot \|v\|_{H'}$.
\end{lem}

\begin{proof}
By the open mapping theorem the map $H/\ker A \to H_1$ is an isomorphism.
\end{proof}

\begin{lem}\label{kernel to domain}
Let $A_1: H \to H_1$ and $A_2: H \to H_2$ be bounded linear operators between Hilbert spaces, and assume that $A_2$ is surjective. Then the operator $A = (A_1,A_2): H\rightarrow H_{1}\oplus H_{2}$ is Fredholm if and only if the operator $A_{1}|_{\ker A_{2}}$ is Fredholm and
\[
\ind A = \ind \left(A_1|_{\ker A_2}\right).
\]
\end{lem}

\begin{proof}
The projection map $\im A \to H_2$ can be included in the short exact sequence $0 \to \im \left(A_1|_{\ker A_2}\right) \to \im A \to H_2 \to 0$, which is naturally a subsequence of the short exact sequence $0 \to H_1 \to H_1 \oplus H_2 \to H_2 \to 0$. The quotient sequence $0 \to H_1/\im (A_1|_{\ker A_2}) \to (H_1 \oplus H_2)/\im A \to 0 \to 0$ is exact by the snake lemma, which proves the equality of the cokernels of the two operators in question. The equality of their kernels is clear.
\end{proof}

We will find it convenient to introduce the notation $M = [0,R]\,\times\,Y$ and write $X_R = W\,\cup\, M$ and
\[
\zp(X_R) = Z\cup_{Y^{-}_{1}}M_{1}\cup_{Y^{+}_{1}}W_{1}\cup_{Y^{-}_{2}} M_{2}\cup_{Y^{+}_{2}}W_{2}\cup \ldots
\]
with $W_n = W$ and $M_n = M$ for all $n \ge 1$. Each of the manifolds $Y_n^{\pm}$ is just a copy of $Y$ but the notations are chosen so that $M_n$ is a cobordism from $Y_n^-$ to $Y_n^+$ while $W_n$ is a cobordism from $Y_n^+$ to $Y_{n+1}^-$.

The spin Dirac operator $\D(Y)$ is a self-adjoint elliptic operator on a compact manifold hence it has a discrete spectrum with real eigenvalues of finite multiplicity. Denote by $V_{\pm}(Y)\subset L^{2}_{1/2}\,(Y;\S)$ the subspaces spanned by the eigenspinors of $\D(Y)$ with respectively the positive and the negative eigenvalues. The $L^2$ orthogonal projections onto these subspaces will be denoted by $\pi_{\pm}$.

\begin{lem}\label{assum}
The operator $\D^+(\wp)$ of Theorem \ref{thm} is invertible if and only if the following two conditions are satisfied:
\begin{enumerate}
\item the Dirac operator $\D(Y)$ has zero kernel, and
\item the Dirac operator
$$
(\D^{+}(W),\pi_{+}\circ r^-, \pi_{-}\circ r^+):L^{2}_{1}\,(W;\S^{+})\rightarrow L^{2}\,(W;\S^{-})\oplus V_{+}(Y)\oplus V_{-}(Y)
$$
with the Atiyah--Patodi--Singer boundary conditions is an isomorphism. Here, $r^{\pm}: L^2_1(W;\S^+) \to L^2_{1/2} (Y;\S)$ denote the restriction maps to the boundary components $Y = \p^- W$ and $Y = \p^+ W$ of $W$.
\end{enumerate}
\end{lem}

\begin{proof}
The condition on $\D(Y)$ to have zero kernel is equivalent to the condition on $\D^+(\wp)$ to be Fredholm. The relation between $\D^+(\wp)$ and the operator $\D^+(W)$ with the Atiyah--Patodi--Singer boundary conditions is well known; see \cite[Proposition 3.11]{aps:I}.
\end{proof}

From now on, we will assume that the operator $\D^+(\wp)$ is invertible or, equivalently, that the conditions (1) and (2) of Lemma \ref{assum} are satisfied.

Given a family of Hilbert spaces $H_i$, $i \ge 1$, their direct sum $\bigoplus H_i$ is the Hilbert space which consists of all the sequences $(u_1,u_2,\ldots)$ of vectors $u_i \in H_i$ such that $\sum\, \|u_i\|^{2}_{H_i} < \infty$, the inner product of sequences $(u_1,u_2,\ldots)$ and $(v_1,v_2,\ldots)$ being $\sum\, (u_i, v_i)_{H_i}$. Any uniformly bounded family of bounded operators $T_i: H_i \to H'_i$ gives rise to a well defined bounded operator
\[
\bigoplus A_i: \bigoplus H_i \longrightarrow \bigoplus H'_i
\]
of norm $\sup \|A_i\|$. An application of this abstract construction to the above splitting of $\zp (X_R)$ yields the following result (we suppress spinor bundles in our notations).

\begin{lem}\label{break the sobolev space}
The natural restriction maps provide Hilbert space isomorphisms
$$
L^2(\zp(X_R)) = L^2(Z)\,\oplus\,\left(\bigoplus L^2(W_i)\right)\,\oplus\,\left(\bigoplus L^2(M_i)\right)
\quad\text{and}\quad
L^2_1(\zp(X_R))=\ker r,
$$
where $r$ is the restriction map
$$
L^2_1(Z)\oplus\left(\bigoplus L^2_{1}(W_i)\right)\oplus\left(\bigoplus L^2_1(M_i)\right)\rightarrow \left(\bigoplus L^2_{1/2}(Y^{-}_{i})\right)\oplus\left(\bigoplus L^2_{1/2}(Y^{+}_{i})\right)
$$
which sends
$
\,\phi_0\oplus (\phi_{1},\phi_{2},\ldots)\oplus (\psi_{1},\psi_{2},\ldots)
$
to
$$
(\phi_{0}|_{Y^{-}_{1}}-\psi_{1}|_{Y^{-}_{1}},\phi_{1}|_{Y^{-}_{2}}-\psi_{2}|_{Y^{-}_{2}},\ldots)\oplus (\phi_{1}|_{Y^{+}_{1}}-\psi_{1}|_{Y^{+}_{1}},\phi_{2}|_{Y^{+}_{2}}-\psi_{2}|_{Y^{+}_{2}},\ldots).
$$
\end{lem}

\begin{proof}
Claim (1) is straightforward. To prove (2), observe that there is an obvious norm preserving  inclusion of $L^2_1 (\zp(X_R))$ into $\ker r$. The result now follows from the fact that all spinors in $\ker r$ belong to $L^2_{1,\rm{loc}}(\zp(X_R))$, see for instance Manolescu \cite[Lemma 3]{Manolescu1}.
\end{proof}

\begin{lem}\label{surjective}
(1) The Dirac operator  $\D^+(M): L^2_1(M)\rightarrow L^2(M)$ is surjective.

\noindent (2) The operator $\ker \D^+ (M)\rightarrow V_+(Y^-)\oplus V_-(Y^+)$ sending $\psi$ to $(\pi_{+}(\psi|_{Y^{-}}),\pi_{-}(\psi|_{Y^{+}}))$ is an isomorphism.

\noindent (3) The restriction maps $r^{\pm}: L^2_1(W;\S^{+})\rightarrow L^2_{1/2}(Y;\S)$ of Lemma \ref{assum} are surjective.

\noindent (4) The restriction map $r$ of Lemma \ref{break the sobolev space} is surjective.

\noindent (5) The operator $\bigoplus \D^+(M_i): \bigoplus L^2_1(M_i)\rightarrow \bigoplus L^2(M_i)$ is surjective.
\end{lem}

\begin{proof}
Claim (1) is proved in \cite[Corollary 17.1.5]{KM}; (2) can be easily verified using the spectral decomposition of $\D(Y)$ and the fact that $\ker \D(Y) = 0$; (3) is a standard fact about Sobolev spaces; (4) follows from (3) and Lemma \ref{norm control}; (5) follows from (1) and Lemma \ref{norm control}.
\end{proof}

%%%%%%%%%%%%%%%%%%%%%%%%%%%%%%%%%%%%%%%%%%%%

\subsection{Proof of Theorem \ref{thm}}
The proof will essentially be a repeated application of Lemma \ref{kernel to domain} to the Dirac
operator
$$
\D_0 = \D^{+}(\zp(X_R)): L^2_1(\zp(X_R);\,\S^+)\rightarrow L^2 (\zp(X_R);\,\S^-).
$$

%%%%%%%%%%%%%%%%%%%%%%%%%%%%%%%%%%%%%%%%%%%%

\medskip\noindent\textbf{Step 1.} Consider the operator
\begin{multline}\notag
\D_{1}: L^2_1 (Z)\oplus \left(\bigoplus L^2_1(W_i)\right)\oplus \left(\bigoplus L^2_1(M_i)\right) \\
\longrightarrow L^{2}(Z)\oplus \left(\bigoplus L^2(W_i)\right)\oplus \left(\bigoplus L^2(M_i)\right)\oplus \left(\bigoplus L^2_{1/2}(Y^-_i)\right)\oplus \left(\bigoplus L^2_{1/2}(Y^+_i)\right)
\end{multline}
sending $\phi_{0}\oplus (\phi_{1},\phi_{2},\ldots)\oplus (\psi_{1},\psi_{2},\ldots)$
to
\begin{align*}
\D^+ \phi_{0}\oplus (\D^+\phi_{1},\D^+\phi_{2},\ldots)\oplus (\D^+\psi_{1},\D^+\psi_{2},\ldots)
&\oplus (\phi_{0}|_{Y^{-}_{1}}-\psi_{1}|_{Y^{-}_{1}},\phi_{1}|_{Y^{-}_{2}}-\psi_{2}|_{Y^{-}_{2}},\ldots) \\ & \oplus (\phi_{1}|_{Y^{+}_{1}}-\psi_{1}|_{Y^{+}_{1}},\phi_{2}|_{Y^{+}_{2}}-\psi_{2}|_{Y^{+}_{2}},\ldots)
\end{align*}
It follows from Lemma \ref{break the sobolev space}, Lemma \ref{surjective} (4) and Lemma \ref{kernel to domain} that $\D_0$ is Fredholm if and only if $\D_1$ is Fredholm, and $$\ind \D_0 = \ind \D_1.$$

%%%%%%%%%%%%%%%%%%%%%%%%%%%%%%%%%%%%%%%%%%%%

\medskip\noindent\textbf{Step 2.} Observe that the kernel of $\bigoplus \D^+(M_i)$ equals $\bigoplus \ker \D^+(M_i)$ and consider the operator
\begin{multline}\notag
\D_2: L^2_1 (Z)\oplus \left(\bigoplus L^2_1(W_i)\right)\oplus \left(\bigoplus \ker \D^+(M_i)\right) \\
\longrightarrow L^2(Z)\oplus\left(\bigoplus L^2(W_i)\right)\oplus\left(\bigoplus L^2_{1/2}(Y^-_i)\right)\oplus\left(\bigoplus L^2_{1/2}(Y^+_i)\right)
\end{multline}
sending $\phi_{0}\oplus (\phi_{1},\phi_{2},\ldots)\oplus (\psi_{1},\psi_{2},\ldots)$
to
\begin{align*}
\D^+\phi_{0}\oplus (\D^+\phi_1,\D^+\phi_2,\ldots) & \oplus (\phi_{0}|_{Y^{-}_{1}}-\psi_{1}|_{Y^{-}_{1}},\phi_{1}|_{Y^{-}_{2}}-\psi_{2}|_{Y^{-}_{2}},\ldots)\\ & \oplus (\phi_{1}|_{Y^{+}_{1}}-\psi_{1}|_{Y^{+}_{1}},\phi_{2}|_{Y^{+}_{2}}-\psi_{2}|_{Y^{+}_{2}},\ldots)
\end{align*}
It follows from Lemma \ref{kernel to domain} and Lemma \ref{surjective} (5) that the operator $\D_1$ is Fredholm if and only if $\D_2$ is Fredholm, and $$\ind \D_1 = \ind \D_2.$$

%%%%%%%%%%%%%%%%%%%%%%%%%%%%%%%%%%%%%%%%%%%%

\medskip\noindent\textbf{Step 3.} Using the subspaces $V_+ (Y)$ and $V_-(Y)$ spanned by the positive and negative eigenspinors of the operator $\D(Y)$, and the respective $L^2$ orthogonal projections $\pi_+$ and $\pi_-$, the operator $\D_2$ can be written as the operator
\begin{multline}\notag
\D_{3}: L^2_1(Z)\oplus \left(\bigoplus L^2_1(W_i)\right) \oplus \left(\bigoplus \ker \D^+(M_i)\right)\longrightarrow L^2 (Z) \oplus \left(\bigoplus L^2(W_i)\right) \\ \oplus\left(\bigoplus V_-(Y_i^-)\right)\oplus \left(\bigoplus V_+(Y_i^+)\right)\oplus\left(\bigoplus V_+(Y_i^-)\right)\oplus \left(\bigoplus V_-(Y_i^+)\right)
\end{multline}
sending $\phi_{0}\oplus (\phi_{1},\phi_{2},\ldots)\oplus (\psi_{1},\psi_{2},\ldots)$
to
\begin{equation*}
\begin{split}
\D^+\phi_0\oplus (\D^+\phi_{1},\D^+\phi_2,\ldots)&\oplus (\pi_{-}\phi_{0}|_{Y^{-}_{1}}-\pi_{-}\psi_{1}|_{Y^{-}_{1}},\;\pi_{-}\phi_{1}|_{Y^{-}_{2}}-\pi_{-}\psi_{2}|_{Y^{-}_{2}},\ldots)\\&\oplus (\pi_{+}\phi_{1}|_{Y^{+}_{1}}-\pi_{+}\psi_{1}|_{Y^{+}_{1}},\;\pi_{+}\phi_{2}|_{Y^{+}_{2}}-\pi_{+}\psi_{2}|_{Y^{+}_{2}},\ldots)\\&\oplus (\pi_{+}\phi_{0}|_{Y^{-}_{1}}-\pi_{+}\psi_{1}|_{Y^{-}_{1}},\;\pi_{+}\phi_{1}|_{Y^{-}_{2}}-\pi_{+}\psi_{2}|_{Y^{-}_{2}},\ldots)\\&\oplus (\pi_{-}\phi_{1}|_{Y^{+}_{1}}-\pi_{-}\psi_{1}|_{Y^{+}_{1}},\;\pi_{-}\phi_{2}|_{Y^{+}_{2}}-\pi_{-}\psi_{2}|_{Y^{+}_{2}},\ldots)
\end{split}
\end{equation*}
Since the operators $\D_2$ and $\D_3$ are isomorphic, we have $$\ind \D_2 = \ind \D_3.$$

%%%%%%%%%%%%%%%%%%%%%%%%%%%%%%%%%%%%%%%%%%%%

\medskip\noindent\textbf{Step 4.} By Lemma \ref{surjective} (2), for each $i \ge 1$ we have an isomorphism
$$
\left(\bigoplus V_+ (Y^-_i)\right)\oplus\left(\bigoplus V_- (Y_i^+)\right) = \bigoplus \ker \D^+(M_i).
$$
Compose this isomorphism with the restrictions to respective boundary components and spectral projections to obtain the operator
$$
\left(\bigoplus V_+ (Y^-_i)\right)\oplus\left(\bigoplus V_- (Y_i^+)\right) \longrightarrow \bigoplus \ker \D^+(M_i) \longrightarrow \left(\bigoplus V_-(Y_i^-)\right) \oplus \left(\bigoplus V_+ (Y_i^+)\right)
$$
sending $(s_{1,+},s_{2,+},\ldots) \oplus (s_{1,-},s_{2,-},\ldots)$ to
$$
(e^{R\D} s_{1,-},e^{R\D} s_{2,-},\ldots)\oplus (e^{-R\D} s_{1,+},e^{-R\D} s_{2,+},\ldots).
$$
Here, we used the notation $\D = \D(Y)$. Note that $e^{-R\D}$ is a smoothing operator on $V_+$ while $e^{R\D}$ is a smoothing operator on $V_-$. The operator $\D_3$ can now be written as
\begin{multline}\notag
\D_{4}: L^2_1(Z)\oplus \left(\bigoplus L^2_1(W_i)\right)\oplus \left(\bigoplus V_+(Y_i^-)\right)\oplus \left(\bigoplus V_-(Y_i^+)\right) \longrightarrow
L^2(Z)\oplus \left(\bigoplus L^2(W_i)\right) \\ \oplus \left(\bigoplus V_-(Y_i^-)\right)\oplus \left(\bigoplus V_+(Y_i^+)\right)\oplus \left(\bigoplus V_+(Y_i^-)\right)\oplus \left(\bigoplus V_-(Y_i^+)\right)
\end{multline}
sending $\phi_{0}\oplus (\phi_{1},\phi_{2},\ldots)\oplus (s_{1,+},s_{2,+},\ldots)\oplus (s_{1,-},s_{2,-},\ldots)$ to
\begin{equation*}
\begin{split}
\D^{+}\phi_{0}\oplus (\D^{+}\phi_{1},\D^{+}\phi_{2},\ldots)&\oplus (\pi_{-}\phi_{0}|_{Y^{-}_{1}}\;-\; e^{R\D}s_{1,-},\,\pi_{-}\phi_{1}|_{Y^{-}_{2}}\;-\;e^{R\D}s_{2,-},\ldots)\\&\oplus (\pi_{+}\phi_{1}|_{Y^{+}_{1}}-e^{-R\D}s_{1,+},\,\pi_{+}\phi_{2}|_{Y^{+}_{2}}-e^{-R\D}s_{2,+},\ldots)\\&\oplus (\pi_{+}\phi_{0}|_{Y^{-}_{1}}-s_{1,+},\,\pi_{+}\phi_{1}|_{Y^{-}_{2}}-s_{2,+},\ldots)\\&\oplus (\pi_{-}\phi_{1}|_{Y^{+}_{1}}-s_{1,-},\,\pi_{-}\phi_{2}|_{Y^{+}_{2}}-s_{2,-},\ldots)
\end{split}
\end{equation*}
Since the operators $\D_3$ and $\D_4$ are isomorphic, we again conclude that
$$
\ind \D_3 = \ind \D_4.
$$

%%%%%%%%%%%%%%%%%%%%%%%%%%%%%%%%%%%%%%%%%%%%

\medskip\noindent\textbf{Step 5.} Consider the last two components of $\D_4$, that is, the operator
$$
L^2_1(Z)\oplus \left(\bigoplus L^2_1(W_i)\right)\oplus \left(\bigoplus V_+(Y_i^-)\right)\oplus \left(\bigoplus V_-(Y_i^+)\right) \rightarrow \left(\bigoplus V_+(Y_i^-)\right)\oplus \left(\bigoplus V_-(Y_i^+)\right)
$$
sending $\phi_0\oplus (\phi_1,\phi_2,\ldots)\oplus (s_{1,+},s_{2,+},\ldots)\oplus (s_{1,-},s_{2,-},\ldots)$ to
$$
(\pi_{+}\phi_{0}|_{Y^{-}_{1}}-s_{1,+},\,\pi_{+}\phi_{1}|_{Y^{-}_{2}}-s_{2,+},\ldots)\oplus (\pi_{-}\phi_{1}|_{Y^{+}_{1}}-s_{1,-},\,\pi_{-}\phi_{2}|_{Y^{+}_{2}}-s_{2,-},\ldots).
$$
This operator is obviously surjective. Therefore, we can apply Lemma \ref{kernel to domain} to the first four components of $\D_4$ restricted to the kernel of the last two components. The resulting operator
$$
\D_{5}: L^2_1 (Z)\oplus \left(\bigoplus L^2_1(W_i)\right)\rightarrow L^2 (Z)\oplus \left(\bigoplus L^2(W_i)\right)\oplus \left(\bigoplus V_-(Y_i^-)\right)\oplus \left(\bigoplus V_+(Y_i^+)\right)
$$
sends $\phi_{0}\oplus (\phi_{1},\phi_{2},\ldots)$ to
\begin{equation*}
\begin{split}
\D^+\phi_0\oplus (\D^+\phi_1,\D^+\phi_2,\ldots)&\oplus (\pi_{-}\phi_{0}|_{Y^{-}_{1}}\;-\;e^{R\D}\pi_{-}\phi_{1}|_{Y_{1}^{+}},\,\pi_{-}\phi_{1}|_{Y^{-}_{2}}\;-\;e^{R\D}\pi_{-}\phi_{2}|_{Y_{2}^{+}},\ldots)\\&\oplus (\pi_{+}\phi_{1}|_{Y^{+}_{1}}-e^{-R\D}\pi_{+}\phi_{0}|_{Y_{1}^{-}},\,\pi_{+}\phi_{2}|_{Y^{+}_{2}}-e^{-R\D}\pi_{+}\phi_{1}|_{Y_{2}^{-}},\ldots)
\end{split}
\end{equation*}
It follows from Lemma \ref{kernel to domain} that the operator $\D_4$ is Fredholm if and only if $\D_5$ is Fredholm, and $$\ind \D_4 = \ind \D_5.$$

%%%%%%%%%%%%%%%%%%%%%%%%%%%%%%%%%%%%%%%%%%%%

\medskip\noindent\textbf{Step 6.} The operator $\D_5$ splits as $\D_5 = \D_6 + K$, where the operator $\D_6$ sends $\phi_0 \oplus \allowbreak (\phi_1,\phi_2,\ldots)$ to
$$
\D^+\phi_0\oplus (\D^+\phi_1,\D^+\phi_{2},\ldots)\oplus (\pi_-\phi_0|_{Y^-_1},\pi_-\phi_1|_{Y^-_2},\ldots)\oplus (\pi_+\phi_1|_{Y^+_1},\pi_+\phi_2|_{Y^+_2},\ldots)
$$
and the operator $K$ sends $\phi_0\oplus (\phi_1,\phi_2,\ldots)$ to
$$
0\oplus 0\oplus (-e^{R\D}\pi_-\phi_1|_{Y_1^+}, -e^{R\D}\pi_-\phi_2|_{Y_2^+},\ldots)\oplus (-e^{-R\D}\pi_+\phi_0|_{Y_1^-},-e^{-R\D}\pi_+\phi_1|_{Y_2^-},\ldots).
$$
According to Lemma \ref{assum}, the operator $\D = \D(Y)$ has zero kernel. Denote by $\mu > 0$ the smallest absolute value of the eigenvalues of $\D(Y)$ then the operator norm of $K$ does not exceed $C\cdot e^{-\mu R}$, where $C$ is a constant independent of $R$. Therefore, if $R$ is sufficiently large, the operator $\D_5$ is Fredholm if %and only if 
$\D_6$ is Fredholm, and in this case
$$
\ind \D_5 = \ind \D_6.
$$
The operator $\D_6$ further splits as a direct sum of the Dirac operator with the Atiyah--Patodi--Singer boundary conditions,
$$
\D_6^0: L^2_1(Z)\rightarrow L^{2}(Z)\oplus V_-(Y_1^-),\quad \phi_0 \to (\D^+\phi_{0}, \pi_-\phi_0|_{Y^-_1}),
$$
and an infinite family of operators
\smallskip
$$
\D_6^i: L^2_1(W_i)\rightarrow L^2(W_i)\oplus V_+(Y_i^+) \oplus V_-(Y_{i+1}^-),\quad \phi_i \to (\D^+\phi_i,\pi_+\phi_i|_{Y^+_i},\pi_-\phi_i|_{Y^-_{i+1}}),
$$

\smallskip\noindent
for $i = 1, 2,\ldots$. By Lemma \ref{assum}, each of the operators $\D_6^i$ with $i \ge 1$ is an isomorphism. Therefore, the operator $\D_6$ is Fredholm if and only if $\D^0_6$ is Fredholm, and
$$
\ind \D_6 = \ind \D_6^0.
$$
The operator $\D^0_6$ is precisely the Dirac operator $\D^+(Z)$ with the Atiyah--Patodi--Singer boundary conditions. Since $\ker \D(Y) = 0$, the operator $\D^+(\zp)$ is a Fredholm operator of index
\[
\ind \D^0_6 = \ind \D^+ (\zp),
\]

\smallskip\noindent
see Atiyah--Patodi--Singer \cite[Proposition 3.11]{aps:I}. This completes the proof of Theorem \ref{thm}.

%%%%%%%%%%%%%%%%%%%%%%%%%%%%%%%%%%%%%%%%%%%%

\section{First eigenvalue estimate}\label{S:eigenvalue}
In this section, we continue the study of manifolds $X_R$ defined in \eqref{E:XRn} by stretching the neck of a spin manifold $X$ of dimension $n \equiv 0 \pmod 4$. We will be interested in estimating the first eigenvalue of $\D^-(X_R)\,\D^+(X_R)$ as $R \to \infty$. This estimate will be used in the compactness argument in Section \ref{S:compact}.

\begin{pro}\label{eigenvalue estimate}
Let us assume that the spin Dirac operator $\D^+(\wp): L^2_2\,(\wp;\,\S^+) \to L^2_1\,(\wp;\S^-)$ is an isomorphism. Then there exist constants $R_0 > 0$ and $\epsilon_1 > 0$ such that for any $R\geq R_0$, the operator
$$
\Delta_R = \D^{-}(X_{R})\, \D^{+}(X_{R}):\, L^{2}_{2}\,(X_{R};\S^{+})\rightarrow L^{2}\,(X_{R};\S^{+})
$$
has no eigenvalues in the interval $[0,\epsilon_1^2)$.
\end{pro}

\begin{proof}
For the purpose of this proof, we will view $W$ as a cobordism from $Y_1$ to $Y_2$ with $Y_1 = Y_2 = Y$. The manifold $X_R = W\cup\, ([0,R]\times Y)$ is then obtained from $W$ by gluing $\{0\} \times Y$ to $Y_2$ and $\{R\}\times Y$ to $Y_{1}$. Denote by $V_{\pm}(Y_1)\subset L^{2}_{3/2}\,(Y_{1};\S)$ and $V_{\pm}(Y_2)\subset L^{2}_{3/2}\,(Y_{2};\S)$ the subspaces spanned by the eigenspinors of $\D(Y_{1})$ and $\D(Y_{2})$ with respectively  positive and  negative eigenvalues. The $L^2$ orthogonal projections onto these subspaces will be denoted by $\pi_{\pm}$, and the restriction maps $L^2_2\, (W;\S^+) \to L^2_{3/2}\,(Y_i;\S)$ will be denoted by $r_i$ with $i = 1, 2$. As in Lemma \ref{assum}, the operator $\D^+(\wp)$ is an isomorphism if and only if the following two conditions are satisfied:
\begin{enumerate}
\item the Dirac operator $\D(Y)$ has zero kernel, and
\item the Dirac operator
$$
(\D^{+}(W),\pi_{+}\circ r_1, \pi_{-}\circ r_2):L^{2}_{2}\,(W;\S^{+})\rightarrow L^{2}_{1}\,(W;\S^{-})\oplus V_{+}(Y_1)\oplus V_{-}(Y_2)
$$
with the Atiyah--Patodi--Singer boundary conditions is an isomorphism.
\end{enumerate}
The operator $\Delta_R$ is a non-negative self-adjoint elliptic differential operator hence all of its eigenvalues have the form $\lambda^{2}$ with a real $\lambda \ge 0$. Using the fact that $\D^+ (X_R)$ has zero index, one can easily check that $\lambda^2$ is an eigenvalue of $ \Delta_{R}$ if and only if the operator
\begin{gather}
\D_{1}:L^{2}_{2}\,(X_{R};\S^{+}\oplus \S^{-})\longrightarrow L^{2}_{1}\,(X_{R};\S^{+}\oplus \S^{-}), \notag \\
\D_1(\psi^{+},\psi^{-}) = (\D^{-}\psi^{-}-\lambda\psi^{+},\,\D^{+}\psi^{+}-\lambda\psi^{-}),\notag
\end{gather}
has non-zero kernel. We denote the restriction of $(\psi^{+},\psi^{-})$ to $W$, respectively, $[0,R]\times Y$ by $(\phi^{+},\phi^{-})$, respectively, $(\tilde{\phi}^{+},\tilde{\phi}^{-})$. Supposing that $(\psi^{+},\psi^{-})$ belongs to the kernel of $\D_{1}$, then the following conditions are satisfied:
\begin{enumerate}[(i)]
\item $\D^{+}\phi^{+}=\lambda\phi^{-}$ and $\D^{-}\phi^{-}=\lambda\phi^{+}$ on $W$;
\item $\D^{+}\tilde{\phi}^{+}=\lambda\tilde{\phi}^{-}$ and $\D^{-}\tilde{\phi}^{-}=\lambda\tilde{\phi}^{+}$ on $[0,R] \times Y$;
\item $(\pi_{+}\phi^{+}|_{Y_{1}},\pi_{-}\phi^{+}|_{Y_{2}},\pi_{-}\phi^{-}|_{Y_{1}},\pi_{+}\phi^{-}|_{Y_{2}}) = (\pi_{+}\tilde{\phi}^{+}|_{Y_{1}},\pi_{-}\tilde{\phi}^{+}|_{Y_{2}},\pi_{-}\tilde{\phi}^{-}|_{Y_{1}},\pi_{+}\tilde{\phi}^{-}|_{Y_{2}})$;
\item $(\pi_{-}\phi^{+}|_{Y_{1}},\pi_{+}\phi^{+}|_{Y_{2}},\pi_{+}\phi^{-}|_{Y_{1}},\pi_{-}\phi^{-}|_{Y_{2}}) = (\pi_{-}\tilde{\phi}^{+}|_{Y_{1}},\pi_{+}\tilde{\phi}^{+}|_{Y_{2}},\pi_{+}\tilde{\phi}^{-}|_{Y_{1}},\pi_{-}\tilde{\phi}^{-}|_{Y_{2}}).$
\end{enumerate}

\begin{lem}\label{from one end to the other} (1) There exists a linear operator $T_+(\lambda,R): V_+(Y_2)\oplus V_+(Y_1)\rightarrow V_+(Y_2)\oplus V_+(Y_1)$ such that, for any $(\tilde{\phi}^{+},\tilde{\phi}^{-})$ satisfying (ii),
$$
T_+(\lambda,R)(\pi_{+}\tilde{\phi}^{+}|_{Y_{2}},\pi_{+}\tilde{\phi}^{-}|_{Y_{1}})=(\pi_{+}\tilde{\phi}^{-}|_{Y_{2}},\pi_{+}\tilde{\phi}^{+}|_{Y_{1}}).
$$

\noindent (2) There exists a linear operator $T_{-}(\lambda,R): V_- (Y_1)\,\oplus\, V_- (Y_2) \rightarrow V_- (Y_1)\,\oplus\, V_- (Y_2)$ such that, for any $(\tilde{\phi}^{+},\tilde{\phi}^{-})$ satisfying (ii),
$$
T_{-}(\lambda,R)(\pi_{-}\tilde{\phi}^{+}|_{Y_{1}},\pi_{-}\tilde{\phi}^{-}|_{Y_{2}})=(\pi_{-}\tilde{\phi}^{-}|_{Y_{1}},\pi_{-}\tilde{\phi}^{+}|_{Y_{2}}).
$$

\noindent
(3) For any $\epsilon>0$, there exist constants $R_{0} > 0$ and $\epsilon_{2} > 0$ such that, for any $R \geq R_{0}$ and $0\leq \lambda<\epsilon_{2}$,
$$
| T_{\pm}(\lambda,R)| < \epsilon.
$$
\end{lem}

\begin{proof}
We focus on the case of $T_{+}(\lambda,R)$ since the other case is similar. Over $[0,R] \times Y$, use Clifford multiplication with $\p/\p t$ to identify the bundles $\S^{+}$ and $\S^{-}$ with each other and with the pull back of the bundle $\S$. This identifies the operators $\D^+ ([0,R]\times Y)$ and $\D^- ([0,R]\times Y)$ with the operators $\p/\p t + \D(Y)$ and $\p/\p t - \D(Y)$, respectively.

Choose a complete system of orthonormal eigenspinors $\varphi_{i}$ for $\D(Y)$, with corresponding eigenvalues $\lambda_{i}$, $i \ge 1$. Let $V_i\,(Y_1)$ (respectively, $V_i\,(Y_2)$) denote the vector space spanned by $\varphi_{i}$, treated as a section over $Y_{1}$ (respectively, $Y_{2}$). It is sufficient to define $T_{+}(\lambda,R)$ on each of the spaces $V_i\, (Y_2) \oplus V_i\,(V_1)$ with $\lambda_{i} > 0$, which we will do next.

Let us write $\tilde{\phi}^{+} (t,y) = \sum\, a_{i}(t)\,\varphi_{i}(y)$ and $\tilde{\phi}^{-} (t,y) = \sum\, b_{i}(t)\,\varphi_{i}(y)$ then condition (ii) takes the form
\begin{equation}\label{ode}
\frac{d}{dt}\begin{pmatrix}
  a_{i}   \\
  b_{i}
\end{pmatrix}\;=\;\begin{pmatrix}
  -\lambda_{i} & \lambda\\
  \lambda & \lambda_{i}
\end{pmatrix}\begin{pmatrix}
  a_{i}    \\
  b_{i}
\end{pmatrix},
\end{equation}

\medskip\noindent
for all $i$ and $t\in [0,R]$. The matrix of this system will be denoted by $A_{i}$. Recall that $Y_{2}$ is identified with $\{0\}\times Y$ and $Y_{1}$ is identified with $\{R\}\times Y$, and express $(b_{i}(0),a_{i}(R))$ in terms of $(a_{i}(0),b_{i}(R))$. The computation that follows is elementary if a bit tedious.

The eigenvalues of $A_{i}$ are $\pm\,\omega_{i}$, where $\omega_{i} = \sqrt{\,\lambda_i^{2} + \lambda^{2}}$, corresponding to the eigenvectors $(\lambda,\lambda_{i}\pm \omega_{i})$. The solutions of (\ref{ode}) are explicitly given by the formula
$$
x\cdot e^{\omega_{i}t}\cdot \begin{pmatrix}
  \lambda    \\
  \lambda_{i}+\omega_{i}
\end{pmatrix}+y\cdot e^{\omega_{i}(R-t)}\cdot \begin{pmatrix}
  \lambda    \\
  \lambda_{i}-\omega_{i}
\end{pmatrix}
\quad \text{with} \quad x,y\in \mathbb{C},
$$
from which we obtain
$$
\begin{pmatrix}
  b_{i}(0)    \\
  a_{i}(R)
\end{pmatrix}=\begin{pmatrix}
  \lambda_{i}+\omega_{i}& e^{\omega_{i}R}(\lambda_{i}-\omega_{i})   \\
  e^{\omega_{i}R}\lambda & \lambda
\end{pmatrix}\begin{pmatrix}
  x    \\
  y
\end{pmatrix}\;\,
$$
and
$$
\begin{pmatrix}
  a_{i}(0)    \\
  b_{i}(R)
\end{pmatrix}=\begin{pmatrix}
  \lambda& e^{\omega_{i}R}\lambda   \\
  e^{\omega_{i}R}(\lambda_{i}+\omega_{i}) & \lambda_{i}-\omega_{i}
\end{pmatrix}\begin{pmatrix}
  x    \\
  y
\end{pmatrix}.
$$

\medskip\noindent
Therefore,
$$
\begin{pmatrix}
  b_{i}(0)    \\
  a_{i}(R)
\end{pmatrix}\; =\; B_{i}\, \begin{pmatrix}
  a_{i}(0)    \\
  b_{i}(R)
\end{pmatrix},
$$
where
\begin{equation*}
B_i = \frac{1}{(\lambda_{i}-\omega_{i})-(\lambda_{i}+\omega_{i})e^{2\omega_{i}R}}\cdot \begin{pmatrix}
\lambda (e^{2\omega_{i}R} - 1)  &  -2\omega_{i}e^{\omega_{i}R}  \\
-2\omega_{i}e^{\omega_{i}R}     &  \lambda (1 - e^{2\omega_{i}R})
\end{pmatrix}.
\end{equation*}

\medskip\noindent
We define $T_+ (\lambda,R)$ on each of the spaces $V_i\,(Y_2)\,\oplus\,V_i (Y_1)$ by the respective matrix $B_{i}$. To derive estimate (3) we let $\lambda_{0}$ be the smallest absolute value of the eigenvalues of the operator $\D(Y)$. Note that $\lambda_0$ is positive by our assumption on the kernel of $\D(Y)$ and that $\omega_i \ge \lambda_0$ for all $i$. For any $\lambda_i > 0$ we obviously have
$$
\left|(\lambda_{i}-\omega_{i})-(\lambda_{i}+\omega_{i})e^{2\omega_{i}R}\right| > (\lambda_{i}+\omega_{i})(e^{2\omega_{i}R}-1) > \omega_{i}(e^{2\omega_{i}R}-1),
$$
hence
\medskip
$$
\left|\frac{2\omega_{i}e^{\omega_{i}R}}{(\lambda_{i}-\omega_{i})-(\lambda_{i}+\omega_{i}) e^{2\omega_{i}R}}\right|\; \leq\; 2\cdot \frac{e^{\omega_i R}}{e^{2\omega_{i}R}-1}\; \leq\; 2\cdot \frac{e^{\lambda_0 R}}{e^{2\lambda_0 R}-1}
$$

\medskip\noindent
and
\medskip
$$
\left|\frac{\lambda(e^{2\omega_{i}R} - 1)}{(\lambda_{i}-\omega_{i})-(\lambda_{i}+\omega_{i})e^{2\omega_{i}R}}\right| \;\leq\; \frac{\lambda}{\omega_i}\;\leq\; \frac{\lambda}{\lambda_0}.
$$

\bigskip\noindent
Therefore, the norms of $B_i$ approach zero uniformly over $i$ as $R \to \infty$ and $\lambda \to 0$. This proves claim (3) for the operator $T_{+}(\lambda,R)$.
\end{proof}

We now return to the proof of Proposition \ref{eigenvalue estimate}. It follows from Lemma \ref{from one end to the other} together with conditions (iii) and (iv) that
\begin{multline}\notag
(\pi_{+}\phi^{-}|_{Y_{2}},\pi_{+}\phi^{+}|_{Y_{1}},\pi_{-}\phi^{-}|_{Y_{1}},\pi_{-}\phi^{+}|_{Y_{2}})= \\
(T_{+}(\lambda,R)(\pi_{+}\phi^{+}|_{Y_{2}},\pi_{+}\phi^{-}|_{Y_{1}}),T_{-}(\lambda,R)(\pi_{-}\phi^{+}|_{Y_{1}},\pi_{-}\phi^{-}|_{Y_{2}})).
\end{multline}
Therefore, the pair $(\phi^{+},\phi^{-})$ belongs to the kernel of the operator $\D_{2}=\D_{3}-K$,
where the operators
$$
\D_{3}, K: L^{2}_{2}(W;\S^{+}\oplus \S^{-})\rightarrow L^{2}_{1}(W;\S^{+}\oplus \S^{-})\oplus V_+ (Y_2)\oplus V_+ (Y_1)\oplus V_- (Y_1)\oplus V_- (Y_2)
$$
are given by the formulas
\begin{gather}
\D_{3}(\phi^{+},\phi^{-})=(\D^{-}\phi^{-},\D^{+}\phi^{+},\pi_{+}\phi^{-}|_{Y_{2}},\pi_{+}\phi^{+}|_{Y_{1}},\pi_{-}\phi^{-}|_{Y_{1}},\pi_{-}\phi^{+}|_{Y_{2}})\quad\text{and} \notag \\
K(\phi^{+},\phi^{-}) = (\lambda \phi^{+},\lambda\phi^{-},T_{+}(\lambda,R)(\pi_{+}\phi^{+}|_{Y_{2}},\pi_{+}\phi^{-}|_{Y_{1}}),T_{-}(\lambda,R)(\pi_{-}\phi^{+}|_{Y_{1}},\pi_{-}\phi^{-}|_{Y_{2}})). \notag
\end{gather}
One can easily see that $\D_{3}$ is isomorphic to the operator $\D_{0}\,\oplus\,\D^{*}_{0}$ hence its kernel is zero by our assumption on the kernel of $\D_0$ (note that the operator $\D_0$ has zero index). Therefore, there exists a constant $C_0$ such that the operator $\D_2 = \D_3 - K$ has zero kernel as long as $\|K\|\leq C_0$, and so does the operator $\D_1$. The proposition now follows from Lemma \ref{from one end to the other} (3).
\end{proof}

The following result, which will be used in Section \ref{S:eta}, is a straightforward extension of Proposition \ref{eigenvalue estimate} to the holomorphic family of operators
\[
\D^{\pm}_z (X_R) = \D^{\pm}(X_R) - \ln z\cdot df,\quad z \in \mathbb C^*,
\]
where $f: X \to S^1$ is an arbitrary smooth function such that $[df] = \gamma \in H^1 (X;\mathbb Z)$. Note that the operators $\D^+_z (X_R)$ and $\D^-_z (X_R)$ are adjoint to each other whenever $|z| = 1$.

\begin{pro}\label{eigenvalue estimate_z}
Let us assume that the spin Dirac operator $\D^+(\wp): L^2_2\,(\wp;\,\S^+) \to L^2_1\,(\wp;\S^-)$ is an isomorphism. Then there exist constants $R_{0} > 0$ and $\epsilon_{1}>0$ such that for any $R\geq R_{0}$, the operators
$$
\D^{-}_z (X_{R})\, \D^{+}_z (X_{R}):\, L^{2}_{2}\,(X_{R};\S^{+})\rightarrow L^{2}\,(X_{R};\S^{-}),\quad |z| = 1,
$$
have no eigenvalues in the interval $[0,\epsilon_{1}^{2})$.
\end{pro}

\begin{proof}
The above proof can easily be adapted by introducing an extra parameter $z$ into the matching of the spinor bundles over $Y_1$. This preserves the conditions (i) and (ii) but replaces the conditions (iii) and (iv) with
\smallskip
\begin{itemize}
\item[(iii)] $(\pi_{+}\phi^{+}|_{Y_{1}},\pi_{-}\phi^{+}|_{Y_{2}},\pi_{-}\phi^{-}|_{Y_{1}},\pi_{+}\phi^{-}|_{Y_{2}}) = (z\,\pi_{+}\tilde{\phi}^{+}|_{Y_{1}},\pi_{-}\tilde{\phi}^{+}|_{Y_{2}},z\,\pi_{-}\tilde{\phi}^{-}|_{Y_{1}},\pi_{+}\tilde{\phi}^{-}|_{Y_{2}})$;
\item[(iv)] $(\pi_{-}\phi^{+}|_{Y_{1}},\pi_{+}\phi^{+}|_{Y_{2}},\pi_{+}\phi^{-}|_{Y_{1}},\pi_{-}\phi^{-}|_{Y_{2}}) = (z\,\pi_{-}\tilde{\phi}^{+}|_{Y_{1}},\pi_{+}\tilde{\phi}^{+}|_{Y_{2}},z\,\pi_{+}\tilde{\phi}^{-}|_{Y_{1}},\pi_{-}\tilde{\phi}^{-}|_{Y_{2}})$.
\end{itemize}

\smallskip\noindent
The new operators $T_+ (\lambda,R)$ and $T_- (\lambda, R)$ that show up in the formula for $K$ are obtained from the old ones by multiplying them on the left by
$$
\begin{pmatrix}
1 & 0 \\ 0 & z
\end{pmatrix}
\quad\text{and}\quad
\begin{pmatrix}
z & 0 \\ 0 & 1
\end{pmatrix},
$$
respectively. Since $|z| = 1$, this does not change the operator norm of $K$, and the rest of the proof goes through with no change.
\end{proof}

%%%%%%%%%%%%%%%%%%%%%%%%%%%%%%%%%%%%%%%%%%%%

\section{Compactness}\label{S:compact}
The proof of Theorem \ref{T: SW=Lefschetz} naturally divides into two steps: compactness and gluing. In this section, we provide the necessary compactness results; the proof of Theorem \ref{T: SW=Lefschetz} will be completed in Section \ref{S:gluing}.

%%%%%%%%%%%%%%%%%%%%%%%%%%%%%%%%%%%%%%%%%%%%

\subsection{Notations}\label{S:notations}
Let $X$ be a connected smooth spin compact 4-manifold with a primitive cohomology class $\gamma \in H^1 (X;\mathbb Z)$. We will assume that the Poincar\'e dual of $\gamma$ is realized by a rational homology 3-sphere $Y \subset X$ and choose a metric $g$ on $X$ which takes the form $g = dt^2 + h$ in a product region $[-\ep,\ep] \times Y$. Given a real number $T > 0$, consider the manifold with long neck
\[
X_T = W\,\cup\,([-T,T] \times Y)
\]
obtained by cutting $X$ open along $\{ 0\} \times Y$ and gluing in the cylinder $[-T,T] \times Y$ along the two copies of $Y$. This differs from the notation $X_R$ used in Section \ref{S:correction} by a simple re-parametrization. In addition, for any $0 < T' < T < \infty$ we will write $X_T = W_{T'}\,\cup\,I_{T',T}$, where
\[
W_{T'} = ([-T',0]\times Y) \cup W \cup ([0,T'] \times Y)\quad\text{and}\quad I_{T',T} = [-T+T',T-T']\times Y.
\]
We will find it convenient to extend these notations to the case of $T = \infty$ by letting $I_{\infty}$ be the disjoint union $(-\infty,0] \times Y \cup [0,+\infty) \times Y$ and using $X_{\infty}$ to denote the manifold $\wp$ with infinite product ends, as in Section \ref{S:correction}. When $T = \infty$ and $T'$ is finite, the notation $I_{T',\infty}$ will mean $\wp - \operatorname{int} (W_{T'})$.

%%%%%%%%%%%%%%%%%%%%%%%%%%%%%%%%%%%%%%%%%%%%%%%

\subsection{Perturbations and regularity of moduli spaces}\label{S: perturbation} 
Recall that, in order to define the monopole Floer homology of a rational homology sphere $Y$ in Section \ref{S:HM}, we introduced perturbations $\mathfrak{q}$. We will assume that our perturbation $\mathfrak{q}$ satisfies Assumption \ref{small perturbation} with respect to a constant $\epsilon_0$ satisfying $0 < \epsilon_0 < \epsilon_1$, where $\epsilon_1 > 0$ is the constant from Proposition \ref{eigenvalue estimate}.

To define the morphisms \eqref{E:W*} induced on the Floer homology of $Y$ by the spin cobordism $W$, we will need to introduce further perturbations. To this end, consider a collar neighborhood  $U = [0,1]\times \partial W$, with $\{1\}\times \partial W$ identified with the actual boundary $\p W = -Y\,\cup\,Y$. Let $\zeta$ be a cut-off function which equals $1$ near $t=1$ and equals $0$ near $t=0$, and let $\zeta_{0}$ be a bump function with compact support in $(-1,0)$. Pick another perturbation $\mathfrak{p}_{0}$ as in Section \ref{S:HM}, and let
\begin{equation}\label{E: mixed perturbation}
\hat{\mathfrak{p}}\; =\; \zeta\cdot\hat{\mathfrak{q}}\,+\,\zeta_0\cdot\hat{\mathfrak{p}}_0,
\end{equation}
where $\hat{\mathfrak{q}}$ and $\hat{\mathfrak{p}}_{0}$ are the 4-dimensional perturbations corresponding to $\mathfrak{q}$ and $\mathfrak{p}_{0}$ respectively; see \cite[Definition 10.1.1]{KM}. This is a perturbation on $W$, supported in $U$. By gluing the perturbations $\hat{\mathfrak{p}}$ on $W$ and $\hat{\mathfrak{q}}$ on $I_{0,T}$ together, we obtain a perturbation $\hat{\mathfrak{p}}_{T}$ on $X_{T}$. Similarly, we define a perturbation $\hat{\mathfrak{p}}_{\infty}$ on $\wp$ by gluing together the perturbations $\hat{\mathfrak{p}}$ on $W$ and $\hat{\mathfrak{q}}$ on $I_{0,\infty}$. These perturbations give rise to the perturbed Seiberg--Witten equations, whose solutions will be referred to as monopoles.

Let $\L_q$ be the perturbed Chern--Simons--Dirac functional as in Section \ref{S:HM}. Downstairs, the gauge equivalence classes of its critical points form the finite set
\[
\widetilde{\mathfrak C}\; =\; \mathfrak C^*\, \cup\; [\theta]
\]
where $[\theta]$ is the unique reducible class and $\mathfrak C^*$ consists of the irreducible classes. Given $[\alpha], [\beta] \in \widetilde{\mathfrak{C}}$, consider the following moduli spaces:
\begin{enumerate}
\item the moduli space $\breve\M\,([\alpha],[\beta])$ of unparameterized (downstairs) trajectories  of the perturbed Chern--Simons--Dirac gradient flow (that is, monopoles on $\mathbb R \times Y$) limiting to $[\alpha]$ and $[\beta]$ at minus and plus infinity, in other words, the quotient of $\M\,([\alpha],[\beta])$ by translations, excluding the constant trajectory if $[\alpha] = [\beta]$, and
\item the moduli space $\M\,(W_{\infty}, [\alpha], [\beta])$ of (downstairs) monopoles on $W_{\infty}$ limiting to $[\alpha]$ and $[\beta]$ at minus and plus infinity. We will write $\M (W_{\infty}, [\alpha])$ for $\M (W_{\infty}, [\alpha], [\alpha])$.
\end{enumerate}
 
Since $\mathfrak{q}$ is admissible by Assumption \ref{small perturbation}, the moduli space $\breve\M\,([\alpha],[\beta])$ is always regular. The regularity of the moduli space $\M\,(\wp, [\alpha], [\beta])$ is proved in the following lemma. 

\begin{lem}\label{L:p_0}
For any nice admissible perturbation $\mathfrak{q}$, there exists a nice perturbation $\mathfrak{p}_{0}$ such that, for the perturbation (\ref{E: mixed perturbation}), the following conditions are satisfied:
\begin{itemize}
\item the various moduli spaces of upstairs monopoles on $W_{\infty}$ are all regular. As a result, the cobordism induced map $W_{*}: \hmred(Y,\s)\to \hmred(Y,\s)$ can be defined using this perturbation;
\item  the moduli space $\M\,(W_{\infty}, [\alpha], [\beta])$ is regular for all $[\alpha],[\beta]\in \widetilde{\mathfrak{C}}$. 
\end{itemize}
\noindent 
Furthermore, we may assume that $\mathfrak{p}_{0}$ is nice and that it satisfies the estimate
\begin{equation}\label{E: derivative of perturbation small 2}
\|D_{(B_{0},0)}\,\mathfrak{p}_{0}^{1}(0,\psi)\|_{L^2(Y)}\;\leq\; \frac{1}{4}\,\epsilon_0\cdot \|\psi\|_{L^2(Y)}
\end{equation}
for any $\psi\in L^2_{k-1/2}(Y;\S)$, where $B_{0}$ is the product connection and $\epsilon_0 > 0$ is the constant fixed in the beginning of this section.
\end{lem} 

\begin{proof}
The proof is a careful check that the arguments of \cite[Proposition 24.4.7]{KM} hold in our situation. We first introduce a large Banach space $\mathcal{P}$ of \emph{nice} perturbations and form the parametrized moduli space 
\[
\M_{\mathcal{P}}(*,*) = \mathop{\bigcup}\limits_{\mathfrak{p}_{0}\in \mathcal{P}}  \M_{\mathfrak{p}_{0}}(*,*).
\]
After proving the regularity of $\M_{\mathcal{P}}(*,*)$, we can apply the Sard--Smale lemma to find a residual subset $U\subset\mathcal{P}$ with the property that $\M_{\mathfrak{p}_{0}}(*,*)$ is regular for any $\mathfrak{p}_{0}\in U$. In particular, we can choose a $\mathfrak{p}_{0}\in U$ satisfying the estimate (\ref{E: derivative of perturbation small 2}). There is one new feature in this argument: since $\mathcal{P}$ only consists of nice perturbations, at a reducible monopole, we can only obtain the transversality in the spinor direction by repeating the arguments in \cite{KM}. This does not cause a problem for the following reason: At a reducible monopole, the linearization of the curvature equation $F^{+}_{A^{t}}=0$ is  the operator $$d^{+}:L^{2}_{k}(W_\infty;iT^{*}W_{\infty})\rightarrow L^{2}_{k-1}(W_\infty;i\Lambda^{2}_{+}T^{*}W_{\infty}).$$ Since $b^{+}_{2}(W)=0$, this operator is surjective (without any perturbation). As a result, the transversality in directions tangent to the space of connections is automatically satisfied.\end{proof}

%%%%%%%%%%%%%%%%%%%%%%%%%%%%%%%%%%%%%%%%%%%%%%%%

\subsection{Statement of the theorem}\label{S:statement}
From now on, we will fix a perturbation $\mathfrak{q}$ satisfying Assumption \ref{small perturbation} and a perturbation $\mathfrak{p}_0$ as in Lemma \ref{L:p_0}. The following compactness theorem is the main result of this section.

\begin{thm}\label{compactness}
Let $T_n$ be a sequence of positive real numbers such that $T_n \to \infty$. Then for any sequence $[(A_{n},\phi_{n})]\in \M(X_{T_{n}})$ there exist $[\alpha]\in \mathfrak{C}^{*}$ and $[(A_{\infty},\phi_{\infty})]\in \M(W_{\infty}, [\alpha])$ such that, after passing to a subsequence, $[(A_{n},\phi_{n})]$ converges to $[(A_{\infty},\phi_{\infty})]$ in the sense of Definition \ref{convergence} below.
\end{thm}

\begin{defi}\label{convergence}
Let $T_n$ be a sequence of positive real numbers such that $T_n \to \infty$. We will say that $[(A_{n},\phi_{n})]\in \M(X_{T_{n}})$ converges to $[(A_{\infty},\phi_{\infty})]\in \M(W_{\infty},[\alpha])$ if the following two conditions hold:
\begin{enumerate}
\item there exists a sequence of $L^2_{k+1}$ gauge transformation $u_{n}:X_{T_{n}}\rightarrow S^{1}$ such that
\[
u_{n}\cdot (A_{n},\phi_{n})\longrightarrow (A_{\infty},\phi_{\infty})\quad\text{in}\quad L^{2}_{k,\loc}(\wp),
\]
where by $L^{2}_{k,\loc}(\wp)$ convergence we mean $L^{2}_{k}$ convergence on compact subsets of $W_{\infty}$, and
\item for any $\epsilon > 0$, there exist a real number $T > 0$ and an integer $N > 0$ such that, for all $n > N$, we have $T_n > T$ and there is a sequence of $L^2_{k+1}$ gauge transformations $v_n: I_{T,T_n} \to S^1$ such that
\[
\left\|\,v_{n}\cdot(A_{n},\phi_{n})|_{ I_{T,T_n}}-\gamma_{\alpha}\,\right\|_{L^2_k ( I_{T,T_n})}\; <\; \epsilon.
\]
Here, $\gamma_{\alpha}$ stands for the constant trajectory of the Chern--Simons--Dirac gradient flow (that is, a translation invariant monopole on the cylinder $I_{T,T_n}$) connecting the critical point $\alpha$ to itself.
\end{enumerate}
\end{defi}

\noindent
Condition (2) roughly says that, up to a gauge transformation, $(A_{n},\phi_{n})$ is a ``near-constant trajectory'' when restricted to the middle of the long neck. Note that the definition of convergence given on page 486 of \cite{KM} only includes condition (1). Actually, condition (2) follows from condition (1) in our case but proving this would require some additional work. Instead of doing this, we simply include condition (2) in our definition of convergence.

The rest of this section will be dedicated to the proof of Theorem \ref{compactness}. We begin with some preparations.

%%%%%%%%%%%%%%%%%%%%%%%%%%%%%%%%%%%%%%%%%%%%%

\subsection{Topological energy}
The perturbed topological energy of a configuration $(A,\phi)$ on a $4$-manifold was defined by Kronheimer and Mrowka in \cite[Definition 24.6.3]{KM}. On a cylinder $[a,b]\times Y$, the perturbed topological energy of $(A,\phi)$ is given by
\[
\E^{\top}_{\mathfrak{q}}(A,\phi) = 2 \left(\L_{\mathfrak{q}}\left((A,\phi)|_{\{a\}\times Y}\right)-\L_{\mathfrak{q}}\left((A,\phi)|_{\{b\}\times Y}\right)\right).
\]
More generally, the topological energy of a configuration $(A,\phi)$ on the cobordism $W_T$ is given by
\[
\E_{\mathfrak{q}}^{\top}(A,\phi) = 2 \left(\L_{\mathfrak{q}}\left((A,\phi)|_{\{T\}\times Y}\right) - \L_{\mathfrak{q}}\left((A,\phi)|_{\{-T\}\times Y}\right)\right).
\]

\begin{lem}\label{energy bound}
(1) Let $(A,\phi)$ be a monopole on a cylinder $[a,b]\times Y$ then $\E_{\mathfrak{q}}^{\top}(A,\phi)\geq 0$, with equality if and only if $(A,\phi)$ is gauge equivalent to a constant trajectory.

(2) There exists a constant $C$ such that $\,\E_{\mathfrak{q}}^{\top}(A,\phi)\,\geq\,C$ for any monopole $(A,\phi)$ on $W_T$.
\end{lem}

\begin{proof}
Claim (1) is clear because any monopole on a cylinder is gauge equivalent to a downward gradient flow line of $\mathcal{L}_{\mathfrak{q}}$. Claim (2) is a direct consequence of \cite[Lemma 24.5.1]{KM} and the equality of the topological and analytic energies for monopoles, see \cite[Definition 4.5.4]{KM}.
\end{proof}

\begin{lem}\label{local compactness}
(1) Let $(A_{n},\phi_{n})$ be a sequence of monopoles on $[a,b] \times Y$ satisfying a uniform bound $\E_{\mathfrak{q}}^{\top}(A_{n},\phi_{n}) \le M$. Then, after passing to a subsequence, there exist gauge transformations $u_{n}:[a,b]\,\times\,Y\rightarrow S^{1}$ and a monopole $(A_{\infty},\phi_{\infty})$ on $[a,b]\,\times\,Y$ such that the sequence $u_{n}\cdot (A_{n},\phi_{n})$ converges to $(A_{\infty},\phi_{\infty})$  in the $L^{2}_{k}$ norm on every interior domain in $[a,b]\times Y$.

(2) Let $I_{n}=[a_{n},b_{n}]$ be a sequence of intervals with $\lim b_{n} = -\infty$ and $\lim a_{n}= +\infty$, and $(A_{n},\phi_{n})$ a sequence of monopoles on $I_n\times Y$ satisfying a uniform bound $\E_{\mathfrak{q}}^{\top}(A_{n},\phi_{n}) \le M$. Then, after passing to a subsequence, there exist gauge transformations $u_{n}: [a_{n},b_{n}]\,\times\,Y\to S^{1}$ and a monopole $[(A_{\infty},\phi_{\infty})] \in \M([\alpha],[\beta])$ with $[\alpha], [\beta]\in \mathfrak{C}$ such that the sequence $u_{n}\cdot (A_{n},\phi_{n})$ converges to $(A_{\infty},\phi_{\infty})$ in $L^{2}_{k,\loc}(\R \times Y)$.

(3) Let $T_{n}$ be a sequence of positive real numbers with $\lim T_{n}=+\infty$, and $(A_{n},\phi_{n})$ a sequence of monopoles on $W_{T_{n}}$ satisfying a uniform bound $\E_{\mathfrak{q}}^{\top}(A_{n},\phi_{n}) \le M$. Then, after passing to a subsequence, there exist gauge transformations $u_{n}:W_{T_{n}}\to S^{1}$ and a monopole $[(A_{\infty},\phi_{\infty})]\in \M(\wp,[\alpha],[\beta])$ with $[\alpha], [\beta]\in \mathfrak{C}$ such that the sequence $u_{n}\cdot (A_{n},\phi_{n})$ converges to $(A_{\infty},\phi_{\infty})$ in $L^{2}_{k,\loc}(\wp)$.
\end{lem}

\begin{proof}
This follows from \cite[Theorem 10.7.1 and Theorem 24.5.2]{KM}.
\end{proof}

%%%%%%%%%%%%%%%%%%%%%%%%%%%%%%%%%%%%%%%%%%%%%%

\subsection{Near-constant trajectories}\label{near constant trajectories}
For every $[\alpha]\in \widetilde{\mathfrak{C}}$, choose an open neighborhood $U_{[\alpha]}\subset \B (Y)$ such that $U_{[\alpha]}\cap U_{[\beta]} = \emptyset$ when $[\alpha]\neq [\beta]$. In addition, for every $[\alpha] \in \widetilde{\mathfrak C}$, choose an open neighborhood  $U_{[\gamma_{\alpha}]}$ of the constant trajectory $[\gamma_{\alpha}]\in \B ([0,1] \times Y)$ such that $[(A,\phi)|_{\{t\}\times Y}]\in U_{[\alpha]}$ for all $t\in [0,1]$ and all $[(A,\phi)]\in U_{[\gamma_{\alpha}]}$. Here $\B ([0,1] \times Y)$ denotes the space of gauge equivalent classes of $L^{2}_{k}$ configurations over $[0,1]\times Y$.

\begin{lem}\label{small energy near constant}
There exists a constant $\epsilon_0 > 0$ such that for any monopole $(A,\phi)$ on $[-1,2]\times Y$ of energy $\E_{\mathfrak{q}}^{\top}(A,\phi)\leq 3 \epsilon_0$ we have $[(A,\phi)|_{[0,1]\times Y}]\subset U_{[\gamma_{\alpha}]}$ for some $[\alpha]\in \widetilde{\mathfrak C}$.
\end{lem}

\begin{proof}
Suppose that this is not true. Then we can find a sequence of monopoles $(A_{n},\phi_{n})$ on $[-1,2] \times Y$ such that $\lim\,\E_{\mathfrak{q}}^{\top}(A_{n},\phi_{n}) = 0$ as $n \to \infty$ but $[(A_{n},\phi_{n})|_{[0,1]\times Y}]$ does not belong to any $U_{[\gamma_{\alpha}]}$. By Lemma \ref{local compactness} (1), after passing to a subsequence, $[(A_{n},\phi_{n})|_{[0,1]\times Y}]$ will converge to a monopole of zero energy, which by Lemma \ref{energy bound} (1) must be of the form $[\gamma_{\alpha}]$ for some $[\alpha]\in \mathfrak{C}$. This leads to a contradiction.
\end{proof}

\begin{lem}\label{the middle of a long neck}
For any real numbers $M > 0$ and $\epsilon > 0$ there exists a constant $T' > 0$ with the following significance: Let $T > T'$ and let $(A,\phi)$ be a monopole on $[-T,T] \times Y$ such that
\[
\E_{\mathfrak{q}}^{\top}(A,\phi)\leq M\quad\text{and}\quad [(A,\phi)|_{\{t\}\times Y}] \in U_{[\alpha]}\quad\text{for all}\quad t\in [-T,T].
\]
Then for any interval $[t,t+4] \subset [-T+T',T-T']$, there exists a gauge transformation $u: [t,t+4]\,\times\, Y\to S^1$ such that
\[
\left\|\, u\cdot (A,\phi)|_{[t,t+4]\times Y}-\gamma_{\alpha}\, \right\|_{L^2_k([t,t+4]\times Y)} < \epsilon.
\]
\end{lem}

\begin{proof}
%Suppose there exists $C$ and $\epsilon$ such that no $T'$ exists. Then we can find a sequence of solutions $\gamma_{n}: [-S_{n},S_{n}] \rightarrow \B(Y)$ satisfying two conditions in the lemma and a number $t_{n}\in [-S_{n},S_{n}]$ with $T_{n}=\min(t_{n}+S_{n},S_{n}-t_{n})\to +\infty$ such that $\gamma_{n}([t_{n},t_{n}+4])$ can not be gauge transformed to a solution that is $\epsilon$ closed to the constant solution. Now we re-parameterize $\gamma_{n}|_{[t_{n}-T_{n},t_{n}+T_{n}]}$ to a trajectory $\gamma_{n}: [-T_{n},T_{n}]\rightarrow \B(Y)$ such that the single piece $[0,4]\times Y$ violates the inequality.

Suppose that this is not true. Then there exist a sequence $T'_n \to \infty$ and a sequence $(A_{n},\phi_{n})$ of monopoles on $[-T_n,T_n] \times Y$ with $T_n > T'_n$ such that $\E_{\mathfrak{q}}^{\top}(A_n,\phi_n)\leq M$ and $[(A_n,\phi_n)|_{\{t\}\times Y}] \in U_{[\alpha]}$ for all $t \in [-T_n,T_n]$ but, for some interval $[t_n, t_n+4] \subset [-T_n + T'_n, T_n - T'_n]$, we have
\[
\left\|\,u\cdot (A_{n},\phi_{n})|_{[t_n,t_n+4]\times Y}-\gamma_{\alpha}\,\right\|_{L^2_k([t_n,t_n+4]\times Y)}\;\geq\;\epsilon\; > 0
\]
for all $n$ and all gauge transformations $u: [t_n, t_n+4]\,\times\, Y\to S^1$. Using the translation invariance of the Seiberg--Witten equations on the cylinder, re-parameterize $(A_n, \phi_n)$ to obtain a monopole on $[-T_n - t_n, T_n - t_n]$, called again $(A_n,\phi_n)$, such that
\[
\left\|\,u\cdot (A_{n},\phi_{n})|_{[0,4]\times Y}-\gamma_{\alpha}\,\right\|_{L^2_k([0,4]\times Y)}\;\geq\;\epsilon\; > 0
\]
for all $n$ and all gauge transformations $u: [0,4]\,\times\, Y\to S^1$. Note that $\lim\,(-T_n - t_n) = -\infty$ and $\lim\,(T_n - t_n) = \infty$. Therefore, by Lemma \ref{local compactness} (2), after passing to a subsequence, we can find gauge transformations $u_n: [-T_n,T_n] \times Y \to S^{1}$ such that $u_n\cdot (A_{n},\phi_{n}) \to (A_{\infty},\phi_{\infty})$ in $L^2_{k,\loc}(\R \times Y)$, where $(A_{\infty},\phi_{\infty})$ is a monopole on $\R \times Y$ limiting to critical points in $\widetilde{\mathfrak C}$ at plus and minus infinity. Since the gauge equivalence class of the restriction of $(A_{\infty},\phi_{\infty})$ to each slice $\{t\} \times Y$ is contained in $U_{[\alpha]}$, and since $[\alpha]$ is the only critical point in $U_{[\alpha]}$, the monopole $(A_{\infty},\phi_{\infty})$ much be gauge equivalent to $\gamma_{\alpha}$ (note that if $Y$ were not a rational homology sphere, we would need to impose the extra condition that $U_{[\alpha]}$ is contractible). Without loss of generality, we may assume that $(A_{\infty},\phi_{\infty}) = \gamma_{\alpha}$. But then the $L^2_{k,\loc}(\R \times Y)$ convergence implies that
\[
\left\|\,u_n\cdot (A_n,\phi_n)|_{[0,4]\times Y}-\gamma_{\alpha}\,\right\|_{L^2_k([0,4]\times Y)}\to 0,
\]
which leads to a contradiction.
\end{proof}

\begin{lem}\label{global norm estimate}
For any $\epsilon > 0$ there exists $\delta > 0$ with the following significance: For any $T > 5$ and any irreducible critical point $[\alpha]\in \mathfrak C^*$, let $(A,\phi)$ be a monopole on $[-T,T]\times Y$ with the property that, for any $[t,t+4] \subset [-T,T]$, there exist a gauge transformation $u_t: [t,t+4]\times Y\rightarrow S^1$ such that
\[
\|u_{t}\cdot (A,\phi)|_{[t,t+4]\times Y}-\gamma_{\alpha}\|_{L^2_k([t,t+4]\times Y)}\leq \delta.
\]
Then there exists a gauge transformation $u: [-T+5,T-5] \times Y\rightarrow S^1$ such that
\[
\|u\cdot (A,\phi)|_{[-T+5,T-5]\times Y}-\gamma_{\alpha}\|_{L^2_k([-T+5,T-5]\times Y)}\leq \epsilon.
\]
\end{lem}

\begin{proof}
This is essentially Lemma 19.3.2 of \cite{KM}. %Note that \cite{KM} makes use of $L^{2}_{k}$ norm upstairs instead of downstairs. However, since $[\alpha]$ is irreducible, these two norms are equivalent to each other. Actually, a careful check of the proof in the book shows that the similar result holds downstairs even for the reducible critical point $[\alpha^{\text{red}}]$. However, this will not be needed in our argument.
In fact, the patching argument in the proof of Lemma \ref{converging to reducible} will be extracted, as in the proof of this lemma, from~\cite[Lemma 13.6.]{KM}.
\end{proof}

%%%%%%%%%%%%%%%%%%%%%%%%%%%%%%%%%%%%%%%%%%%%%%%

\subsection{Broken trajectories on $W_{\infty}$}\label{S:broken}
Let $\gr^{\Q}: \mathfrak C^*\to \mathbb Q$ be the absolute grading function~\cite[page 587]{KM} on the irreducible critical points, and extend it to the unique reducible $[\theta]$ by the formula
\[
\gr^{\mathbb{Q}}([\theta])\, = -2 n(Y,h,\s)
\]
(compare with the grading in Lemma \ref{L: reducible grading}). Then, for any $[\alpha], [\beta]\in \mathfrak C^*$, the expected dimensions (denoted by $\text{e.dim}$) of the moduli spaces are as follows:
\begin{enumerate}
\item $\text{e.dim}(\breve\M([\alpha],[\beta]))\, =\, \gr^{\mathbb{Q}}([\alpha]) - \gr^{\mathbb{Q}} ([\beta]) - 1$;
\item $\text{e.dim}(\breve\M([\alpha],[\,\theta]))\, =\, \gr^{\mathbb{Q}}([\alpha]) - \gr^{\mathbb{Q}}([\theta]) - 1$;
\item $\text{e.dim}(\breve\M([\,\theta],[\alpha]))\, =\, \gr^{\mathbb{Q}}([\theta]) - \gr^{\mathbb{Q}}([\alpha]) - 2$;
\item $\text{e.dim}(\M (\wp,[\alpha],[\beta])) = \gr^{\mathbb{Q}}([\alpha]) - \gr^{\mathbb{Q}}([\beta])$;
\item $\text{e.dim}(\M (\wp,[\alpha],[\,\theta])) = \gr^{\mathbb{Q}}([\alpha]) - \gr^{\mathbb{Q}}([\theta])$;
\item $\text{e.dim}(\M (\wp,[\,\theta],[\alpha])) = \gr^{\mathbb{Q}}([\,\theta]) - \gr^{\mathbb{Q}}([\alpha])-1$;
\item $\text{e.dim}(\M (\wp,[\,\theta],[\theta]))=-1$.
\end{enumerate}
By our regularity assumption, the actual dimensions of the moduli spaces are equal to their expected dimensions except in case (7): in this case, there is always one dimensional cokernel of the corresponding Fredholm operator. %(This phenomenon is called the boundary obstruction.)
As a result, $\M (\wp,[\theta],[\theta])$ is a zero-dimensional manifold which only contains reducible monopoles. Note that the moduli space $\breve\M([\theta],[\theta])$ is empty because we only allow non-constant trajectories in our definition of the moduli spaces $\breve\M([\alpha],[\alpha])$.

\begin{lem}\label{broken at most once}
Let $d \ge 2$ and suppose that $[\alpha_1]$, \ldots, $[\alpha_{d-1}]$ are critical points such that the moduli spaces $\breve\M([\alpha_1],[\alpha_2])$, \ldots, $\breve\M([\alpha_{d-2}],[\alpha_{d-1}])$ and $\M(\wp,[\alpha_{d-1}],[\alpha_{1}])$ are all non-empty. Then $d = 2$.
\end{lem}

\begin{proof}
This follows from our regularity assumption by a simple dimension count.
\end{proof}

%%%%%%%%%%%%%%%%%%%%%%%%%%%%%%%%%%%%%%%%%%%%%

\subsection{Proof of Theorem \ref{compactness}}
We will follow closely the argument of \cite[Section 16.2]{KM}. Let $(A_n,\phi_n)$ be a sequence of monopoles as in the statement of Theorem \ref{compactness}. Since
\begin{multline}\notag
\E_{\mathfrak{q}}^{\top}\left((A_n,\phi_n)|_{[-T_n,T_n] \times Y}\right) = 2\left(\L_{\mathfrak{q}}\left((A_n,\phi_n)|_{\{-T_n\}\times Y}\right) - \L_{\mathfrak{q}}\left((A_n,\phi_n)|_{\{T_n\}\times Y}\right)\right) \\ = - \E_{\mathfrak{q}}^{\top}\left((A_n,\phi_n)|_W\right),
\end{multline}
we conclude from Lemma \ref{energy bound} that 
\[
\E_{\mathfrak{q}}^{\top}\left((A_n,\phi_n)|_{W}\right)\,\leq 0\quad\text{and}\quad \E_{\mathfrak{q}}^{\top}\left((A_n,\phi_n)|_{[-T_n,T_n] \times Y}\right)\,\leq M
\]
for some constant $M \ge 0$ independent of $n$. As a result, we obtain the uniform bounds
\begin{equation}\label{energy control}
\E_{\mathfrak{q}}^{\top}\left((A_n,\phi_n)|_{W_T}\right) \leq M\quad \text{and}\quad\E_{\mathfrak{q}}^{\top}\left((A_n,\phi_n)|_{I\times Y}\right)\leq M
\end{equation}
for any $0 < T < T_n$ and any interval $I \subset [-T_n,T_n]$, which will allow us to apply Lemma \ref{local compactness}.

Choose neighborhoods $U_{[\alpha]}$ and $U_{[\gamma_{\alpha}]}$ as in Section \ref{near constant trajectories}, and let $\epsilon_0 > 0$ be the constant provided by Lemma \ref{small energy near constant}. Restricting $(A_n,\phi_n)$ to the slices $\{t\} \times Y$ gives rise to a path
$\gamma_n: [-T_n,T_n] \to \B (Y)$. For each $n$, consider the set
\[
S_n = \left\{\,p\in \mathbb Z\; |\; [p,p+1] \subset [-T_n,T_n]\;\;\text{and}\;\;\E_{\mathfrak{q}}^{\top}\left((A_n,\phi_n)|_{[p,p+1]\times Y}\right)\geq \epsilon_0 > 0\,\right\}.
\]
This set contains at most $M/\epsilon_0$ elements. By passing to a subsequence, we may find an integer $j$ such that every set $S_n$ contains exactly $j$ elements, $p^{n}_{1}<p^{n}_{2}<\cdots<p^{n}_{j}$. Also introduce the integers $p_{0}^{n}=\lceil -T_{n}\rceil$ and $p_{j+1}^{n}=\lfloor T_{n}\rfloor-1$. After passing to a subsequence one more time, we may assume that, for each integer $m$ between $0$ and $j$, either $\lim\limits_{n\to\infty}(p^{n}_{m+1}-p^{n}_{m})=+\infty$ or $p^{n}_{m+1}-p^{n}_{m}$ is independent of $n$. On the set $\{0,1,\cdots,j+1\}$, define an equivalence relation
\[
m_{1}\sim m_{2}\quad\text{iff}\quad \lim\limits_{n\rightarrow \infty}\left|\;p^{n}_{m_{1}}-p^{n}_{m_{2}}\,\right| < \infty.
\]
Denote by $d$ the number of the equivalence classes; since $0$ is not equivalent to $j$, we must have $d \ge 2$. Pick representatives $m_{0}<m_{1}<m_{2}<\cdots <m_{d-1}$, one for each equivalence class, and let
\[
a^{n}_{i} = \min_m \left\{\,p^{n}_{m}\;|\;m\sim m_i \right\}\quad \text{and}\quad b^{n}_{i} = \max_m \left\{\,p^{n}_{m}\;|\;m\sim m_i \right\}\quad\text{for}\quad 0 \leq i \leq d-1.
\]
Then
\[
\lceil -T_n\rceil=a^n_0\; \le\; b^n_0\; <\; a^n_1\; \le\; b^n_1 < \cdots <a_{d-1}^n\; \le\; b^n_{d-1}=\lfloor T_n\rfloor-1,
\]
the difference $b^n_m - a^n_m$ is independent of $n$ for $0\leq m\leq d-1$, and $\lim\limits_{n\to \infty} (a^n_{m+1}-b^n_m) = \infty$ for $0 \leq m\leq d-2$.
Using Lemma \ref{small energy near constant} together with the translation invariance of the Seiberg--Witten equations, we obtain
\begin{equation}\label{eq 1}
\gamma_{n}([b^{n}_{m}+2, a^{n}_{m+1}-1])\subset U_{[\alpha_{m+1}]}
\end{equation}
 for some critical point $[\alpha_{m+1}]$. Here, we passed to a subsequence again to ensure that $[\alpha_{m+1}]$ is independent of $n$.  Using Lemma \ref{local compactness} and passing to a subsequence if necessary we conclude the following:

\smallskip

\begin{enumerate}
\item[$\bullet$\;] There exist gauge transformations $u_n: X_{T_n}\to S^1$ such that $u_n\cdot(A_n,\phi_n)$ converges in $L^2_{k,\loc}(\wp)$ to a monopole $(A_{\infty},\phi_{\infty})$ on $W_{\infty}$. Using (\ref{eq 1}), it is not difficult to see that $[(A_{\infty},\phi_{\infty})] \in \M (\wp,[\alpha_{d-1}],[\alpha_{1}])$.
\item[$\bullet$\;] Let $\tau_t$ denote the translation on trajectories defined by $(\tau_t \cdot \gamma) (x) = \gamma(x+t)$. Then, for every $n \geq 0$ and $1 \leq m\leq d-2$, we have
\[
\tau_{(a^n_m + b^n_m)/2}\,\cdot \gamma_{n}\; \longrightarrow\; \gamma_{\infty,m}\quad\text{in}\quad L^2_{k,\loc}(\mathbb R \times Y),
\]
where $\gamma_{\infty,m}$ is a trajectory on $\mathbb{R}\times Y$. Since the topological energy of $\gamma_n ([a^n_m, a^n_m+1])$ is greater than or equal to $\epsilon_0 > 0$, we conclude that $\gamma_{\infty,m}$ is not a constant trajectory. Using (\ref{eq 1}), it is not difficult to see that $\gamma_{\infty,m}$ represents a monopole in $\M([\alpha_m],[\alpha_{m+1}])$.
\end{enumerate}

\smallskip\noindent
Since all the moduli spaces $\M([\alpha_{1}],[\alpha_{2}]), \cdots, \M([\alpha_{d-2}],[\alpha_{d-1}])$ and $\M (\wp, [\alpha_{d-1}],[\alpha_{1}])$ are non-empty, it follows from Lemma \ref{broken at most once} that $d = 2$. Keeping this in mind, our earlier discussion implies the following two results:
\smallskip
\begin{enumerate}
\item[(A)] There exist gauge transformations $u_n: X_{T_n} \to S^1$ such that $u_n\, \cdot\, (A_n,\phi_n) \to (A_{\infty},\phi_{\infty})$ in $L^2_{k,\loc}(\wp)$ for some $[(A_{\infty},\phi_{\infty})] \in \M (\wp,[\alpha])$, where $[\alpha] = [\alpha_1]$;
\item[(B)] There exists a constant $T' > 0$ such that $\gamma_{n}([-T_n + T',T_n - T']) \subset U_{[\alpha]}$ for all $n$ (this is implied by (\ref{eq 1})).
\end{enumerate}

\smallskip\noindent
Using Lemma \ref{the middle of a long neck}, we can replace (B) by the following:

\smallskip

\begin{enumerate}
\item[(C)] For any $\epsilon > 0$, there exists a constant $C > 0$ with the following significance: for any $n$ large enough so that $T_n > C + 4$ and any interval $[t, t+4] \subset [-T_n + C,T_n - C]$, there exists a gauge transformation $u_{n,t}: [t,t+4]\times Y \to S^1$ such that
\[
\|u_{n,t}\cdot (A_{n},\phi_{n})|_{[t,t+4]\times Y}-\gamma_{\alpha}\|_{L^2_k([t,t+4]\times Y)}\,\leq\, \epsilon.
\]
\end{enumerate}

\begin{lem}\label{converging to reducible}
The critical point $[\alpha]$ is irreducible.
\end{lem}

This lemma will be proved in Section \ref{S:reducible}. For now, let us assume it and finish the proof of Theorem \ref{compactness}. It is clear that condition (1) of Definition \ref{convergence} follows from (A) and Lemma \ref{converging to reducible}. Therefore, we only need to verify condition (2) of Definition \ref{convergence}. For any $\epsilon > 0$, let $C$ be the constant from (C). Choose $T = C + 5$ and let $N$ be an integer large enough so that $T_n > T$ for all $n \geq N$. A straightforward application of Lemma \ref{global norm estimate} finishes the proof.

%%%%%%%%%%%%%%%%%%%%%%%%%%%%%%%%%%%%%%%%%%%%%%

\subsection{Convergence to reducible}\label{S:reducible}
In this subsection we prove Lemma \ref{converging to reducible}. The proof will be based on the following three lemmas.

\begin{lem}
The moduli space $\M(\wp,[\theta])$ contains a single point $[(A_0,0)]$, where $A_0$ is a trivial connection on $\wp$.
\end{lem}

\begin{proof}
Since both $\mathfrak{q}$ and $\mathfrak{p}_{0}$ are nice, we may disregard the perturbations when studying reducible monopoles downstairs. We saw in Section \ref{S:broken} that the moduli space $\M (\wp,[\theta])$ contains only reducible monopoles $[(A,0)]$ with $F^+_A = 0$. Write $A = A_0 + a$, where $a$ is an $L^2_k$ differential 1-form on $\wp$ with coefficients in $i\mathbb R$ satisfying $d^{+}a=0$. Integration by parts shows that
\smallskip
\[
\int_{\wp}\langle da,da\rangle = 0.
\]

\smallskip\noindent
Therefore, the 1-form $a$ is closed. Since $H^1 (\wp,\mathbb R) = 0$, there exists $\xi: \wp \to \mathbb R$ such that $a = id\xi$, and $(A,0)$ is gauge equivalent to $(A_0,0)$ via the gauge transformation $u = e^{i\xi} \in L^2_{k+1,\loc}(\wp)$. Now we use \cite[Definition 24.2.1]{KM} to conclude that $[(A,0)] = [(A_0,0)]$.
\end{proof}

\begin{lem}\label{L:four}
For any $\epsilon > 0$ there exists a positive integer $N$ such that, for any $n \ge N$,
\begin{enumerate}
\item $T_{n}\geq 4$
\item There exists a gauge transformation $u_{n}:W_{4}\rightarrow S^{1}$ such that
\[
\|u_{n}\cdot(A_{n},\phi_{n})|_{W_{4}}-(A_0,0)\|_{L^2_k (W_4)} < \epsilon
\]
\item For any interval $[t,t+4] \subset [-T_n,T_n]$ there exists a gauge transformation $u_{n,t}: [t,t+4]\times Y \to S^1$ such that
\[
\|u_{n,t}\cdot(A_{n},\phi_{n})|_{[t,t+4]\times Y}-(A_0,0)\|_{L^2_k([t,t+4]\times Y)} < \epsilon.
\]
\end{enumerate}
\end{lem}

\begin{proof}
Claim (1) follows trivially from $T_n \to \infty$. Since $W_4$ is a compact subset of $\wp$, claim (2) follows from (A). To prove claim (3), let $C > 0$ be the constant from (C). If $[t,t+4]$ belongs to $[-T_n + C,T_n - C]$, claim (3) follows from (C). Otherwise, $[t,t+4]$ belongs to either $[-T_n,-T_n+C+4]$ or $[T_n-C-4,T_n]$. For every $n$, these are identified with the fixed  compact subsets $[0,C+4]$ and $[-C-4,0]$ of $\wp$ hence the result follows from (A).
\end{proof}

\begin{lem}\label{L: L2 norm of perturbation}
Let $\epsilon_0 > 0$ be the constant fixed in the beginning of Section \ref{S: perturbation}. Then there exists an integer $N > 0$ such that, for any $n\geq N$, we have 
\[
\|\D^+_{A_n}(\phi_n)\|_{L^2(X_{T_n})}\; \leq\; \frac 12\,\epsilon_0\cdot\|\phi_n\|_{L^2 (X_{T_n})}.
\]
\end{lem}

\begin{proof}
Since $(A_n,\phi_n)$ solves the perturbed Seiberg--Witten equations, we have the equality
\[
\D^+_{A_n}(\phi_n) = \hat{\mathfrak{p}}^{1}_{T_{n}}(A_n,\phi_n),
\]
where $\hat{\mathfrak{p}}^{1}_{T_{n}}(A_{n},\phi_{n})$ denotes the spinor component of the perturbation term $\hat{\mathfrak{p}}_{T_{n}}(A_{n},\phi_{n})$; see Section \ref{S: perturbation}. It is supported in $I_{0,T_{n}}\cup U$, where $U$ is a collar neighborhood of $\partial W$. By our definition of $\hat{\mathfrak{p}}_{T_{n}}$,
\[
\hat{\mathfrak{p}}^{1}_{T_{n}}(A_{n},\phi_{n})\big|_{\{t\}\times Y}=\mathfrak{q}^{1}\left((A_{n},\phi_{n})\big|_{\{t\}\times Y}\right)\quad\text{for all}\;\; t \in [-T_n,T_n].
\] 
Since $\mathfrak{q}$ satisfies Assumption \ref{small perturbation}, it follows from (\ref{E: derivative of perturbation small}) that there exists a neighborhood $V$ of $[\theta]\in \B(Y)$ such that 
\[
\|\mathfrak{q}^1(B,\psi)\|_{L^2(Y)}\;\leq\; \frac 12\,\epsilon_0\cdot \|\psi\|_{L^2(Y)}
\]
for any configuration $(B,\psi)$ with $[(B,\psi)]\in V$. But then, by Lemma \ref{L:four}, there exists a positive integer $N$ such that for any $n\geq N$, we have $[(A_{n},\phi_{n})|_{\{t\}\times Y}]\in V$ for all $t\in [-T_n,T_n]$. Therefore, we have the estimate
\[
\|\D^{+}_{A_{n}}(\phi_{n})\big|_{\{t\}\times Y}\|_{L^2(Y)}\;\leq\; \frac 1 2\,\epsilon_0\cdot \|\phi_{n}\big|_{\{t\}\times Y}\|_{L^2(Y)}\quad\text{for all\; $t \in [-T_{n},T_{n}]$},
\]
which implies that 
\[
\|\D^{+}_{A_{n}}(\phi_{n})\|_{L^2(I_{0,T_n})}\; \leq\; \frac 1 2\,\epsilon_0\cdot \|\phi_{n}\|_{L^{2}(I_{0,T_{n}})}.
\]

\smallskip\noindent
A similar argument involving estimate \eqref{E: derivative of perturbation small 2} shows that 
\[
\|\D^{+}_{A_{n}}(\phi_{n})\|_{L^{2}(U)}\; \leq\; \frac 1 2\,\epsilon_0\cdot \|\phi_{n}\|_{L^{2}(U)}.
\]
This completes the proof of the lemma because $\D^{+}_{A_{n}}(\phi_{n})$ is supported on $I_{0,T_{n}}\cup U$.
\end{proof}

\begin{lem}\label{L:global}
For any $\epsilon > 0$ there exists a positive integer $N$ such that, for any $n \ge N$, there is a global gauge transformation $\tilde{u}_n: X_{T_{n}} \to S^{1}$ such that
\[
\|\,A_{n}-\tilde{u}^{-1}_{n}d\tilde u_n\,\|_{C^{0}(X_{T_n})}\; <\; \epsilon.
\]
\end{lem}

Before we go on to prove this lemma, we will show how it implies Lemma \ref{converging to reducible}. We know from Proposition \ref{eigenvalue estimate} that as $n\rightarrow \infty$ the smallest eigenvalue of the operator  $\D^- \,\D^+$ on $X_{T_n}$ is bounded from below by $\epsilon^2_1 > 0$. Therefore, for all sufficiently large $n$ and all positive spinors $\psi$,
\[
\|\,\D^+ \psi\,\|_{L^2(X_{T_n})}\; \ge \; \epsilon_0\cdot \|\psi\|_{L^2(X_{T_n})}.
\]
On the other hand, consider the sequence of gauge transformations $\tilde  u_n$ from Lemma \ref{L:global} and the sequence of spinors $\psi_n = \tilde u_n \phi_n$ with $\epsilon = \epsilon_0/4$. Lemma \ref{L: L2 norm of perturbation} implies that  
\[
\|\D^{+}_{A_{n}-\tilde u^{-1}_n d\tilde u_n}(\psi_{n})\|_{L^2 (X_{T_n})}\; \leq\; \frac 1 2\,\epsilon_0\cdot \|\psi_{n}\|_{L^2 (X_{T_n})}
\]
We then conclude using Lemma \ref{L:global} that
\[
\|\,\D^+\psi_{n}\,\|_{L^2 (X_{T_n})} = \|\D^{+}_{A_{n}-\tilde u^{-1}_n d\tilde u_n}(\psi_{n})-\, (A_n - \tilde u^{-1}_n d\tilde u_n)\cdot \psi_{n}\,\|_{L^2 (X_{T_n})}\; < \epsilon_0\cdot\|\psi_{n}\|_{L^2 (X_{T_n})}
\]
for all sufficiently large $n$, which gives a contradiction.

\begin{proof}[Proof of Lemma \ref{L:global}]
Let $N$ be as in Lemma \ref{L:four} then, for any $n \ge N$, there exists an integer $m_{n} \geq 1$ such that $-T_{n}+2m_{n}\in [T_{n}-4,T_{n}-2)$. For $1\leq j\leq m_{n}-1$, we denote the gauge transformation $u_{n,-T_{n}+2j}$ of Lemma \ref{L:four} by $u(n,j)$. We wish to glue the gauge transformations $u_n, u(n,1),u(n,2),\cdots,u(n,m_{n}-1)$ together with the help of cutoff functions.

First, we pick a base point $o\in Y$. After multiplying by suitable constant gauge transformations,  we may assume that
\[
u(n,j)(\{-T_{n}+2j+2\}\times o)=u(n,j+1)(\{-T_{n}+2j+2\}\times o).
\]
Since $Y$ is a rational homology sphere, we have
\[
u(n,j)|_{[-T_{n}+2j+2,-T_{n}+2j+4]\times Y}=e^{i\xi(n,j)}\cdot u(n,j+1)|_{[-T_{n}+2j+2,-T_{n}+2j+4]\times Y},
\]
where
\[
\xi(n,j):[-T_{n}+2j+2,-T_{n}+2j+4]\times Y\rightarrow \mathbb{R}
\]
satisfies
\begin{equation}\label{base point normalization}
\xi(n,j)(\{-T_{n}+2j+2\}\times o) = 0.
\end{equation}
Now, since
\[
\|A_{n}-u(n,j)^{-1}du(n,j)\,\|_{L^{2}_{k}([-T_{n}+2j+2,-T_{n}+2j+4]\times Y)} < \epsilon
\]
and
\[
\|A_{n}-u(n,j+1)^{-1}du(n,j+1)\,\|_{L^{2}_{k}([-T_{n}+2j+2,-T_{n}+2j+4]\times Y)} < \epsilon,
\]
we have
\begin{multline}\notag
\|d\xi(n,j)\|_{L^{2}_{k}([-T_{n}+2j+2,-T_{n}+2j+4]\times Y)} = \\ \|u(n,j)^{-1}du(n,j)-u(n,j+1)^{-1}du(n,j+1)\|_{L^{2}_{k}([-T_{n}+2j+2,-T_{n}+2j+4]\times Y)} < 2\epsilon.
\end{multline}
Together with (\ref{base point normalization}), this implies that there exists a constant $C_2 > 0$ such that
\begin{equation}
\|\xi(n,j)\|_{L^{2}_{k+1}([-T_{n}+2j+2,-T_{n}+2j+4]\times Y)}\; <\; C_2\cdot \epsilon.
\end{equation}
Next, choose a bump function $\tau: [0,2]\rightarrow [0,1]$ such that $\tau([0,1])=1$ and $\tau([3/2,2])=0$. We let $\tau_{t}:[t,t+2]\times Y\rightarrow [0,1]$ be the function defined by the formula $\tau_t (s) = \tau(s-t)$. Extend the function $\tau_{-T_{n}+2j+2}\cdot \xi(n,j)$ defined on
\[
[-T_{n}+2j+2,-T_{n}+2j+4]\times Y
\]
by zero to obtain a function
\[
\tilde{\xi}(n,j): [-T_{n}+2j+2,-T_{n}+2j+5]\times Y\rightarrow \mathbb{R}.
\]
We have $\|\tilde{\xi}(n,j)\|_{L^{2}_{k+1}([-T_{n}+2j+2,-T_{n}+2j+5])} < C_3\, \epsilon$ for a constant $C_3 > 0$, which implies that
\begin{equation}
\|d\tilde{\xi}(n,j)\|_{C^{0}([-T_{n}+2j+2,-T_{n}+2j+5] \times Y)}\; <\; C_4\cdot \epsilon
\end{equation}
for another constant $C_4 > 0$. Now, for $2\leq j\leq m_{n}-1$, consider the gauge transformations
\[
\bar{u}(n,j) = e^{i\tilde{\xi}(n,j-1)}\cdot u(n,j): [-T_{n}+2j,-T_{n}+2j+3]\times Y\rightarrow S^{1}.
\]
One can easily see that $\bar{u}(n,j)$ equals $u(n,j-1)$ over $[-T_{n}+2j,-T_{n}+2j+1]\times Y$ and equals $u(n,j)$ over $[-T_{n}+2j+2,-T_{n}+2j+3]\times Y$. Therefore, the functions $\bar{u}(n,2)$, $\bar{u}(n,3), \cdots, \bar{u}(n,m_{n}-1)$, together with the functions $u(n,1)|_{[-T_{n}+2,-T_{n}+5]\times Y}$ and $ u(n,m_{n}-1)|_{[-T_{n}+2m_{n},-T_{n}+2m_{n}+2]\times Y}$, agree with each other on the overlaps. As a result, we can glue them together to obtain a gauge transformation
\[
\bar{u}_{n}: [-T_{n}+2,-T_{n}+2m_{n}+2]\times Y\rightarrow S^{1}.
\]
Since $\bar{u}_n$ is obtained from $u (n,j)$ by multiplying by $e^{i\tilde{\xi}(n,j-1)}$, there is a constant $C_5 > 0$ such that
\[
\|A_{n}-\bar{u}^{-1}_{n}d\bar{u}^{-1}_n\,\|_{C^{0}([-T_{n}+2,-T_{n}+2m_{n}+2]\times Y)}\; <\; C_5\cdot \epsilon.
\]
In our next step, we introduce $\hat u_n = e^{i(a+bt)}\,\bar u_n$, where $t$ is the coordinate in the cylindrical direction, and $a \in [0,2\pi]$ and $b \in [0,\pi/(T_{n}-2)]$ are chosen so that
\[
\hat{u}_{n}(\{-T_{n}+2\}\times o)=u_{n}(\{-T_{n}+2\}\times o)
\]
and
\[
\hat{u}_{n}(\{-T_{n}+2m_{n}+2\}\times o)=u_{n}(\{-T_{n}+2m_{n}+2\}\times o).
\]
There is a constant $C_6 > 0$ such that
\[
\|A_{n}-\hat{u}^{-1}_{n}d\hat{u}^{-1}_{n}\,\|_{C^{0}([-T_{n}+2,-T_{n}+2m_{n}+2]\times Y)}\;\leq\; C_{6}\cdot \epsilon + \pi/(T_{n}-2),
\]
where, by choosing $N > 0$ large enough, we may assume $\pi/(T_{n}-2) < \epsilon$ for all $n\geq N$. Arguing as before, we can find
\[
\xi_n: ([-T_n,-T_n+2]\cup [-T_n+2m_n+2, -T_n+2m_n+4])\times Y\rightarrow \mathbb R
\]
such that $e^{i\xi_{n}}\hat{u}_{n}$ equals $u_{n}$ over the domain of $\xi_{n}$ and
\[
\xi_{n}(\{-T_{n}\}\times o)= \xi_{n}(\{-T_{n}+2m_{n}+4\}\times o) = 0.
\]
As before, we have the estimate
\[
\|d\xi_{n}\|_{L^{2}_{k}}\; < \; 4\epsilon +b\; < \; 5\epsilon
\]
on the domain of $\xi_n$.
Let $\tau_n$ be a cut-off function on the domain of $\xi_n$ which equals $1$ when restricted to $([-T_{n}+1,-T_{n}+2]\cup [-T_{n}+2m_{n}+2, -T_{n}+2m_{n}+3])\times Y$ and equals $0$ when restricted to $([-T_{n},-T_{n}+1/2]\cup [-T_{n}+2m_{n}+7/2, -T_{n}+2m_{n}+4])\times Y$. Assume that the $C^{\infty}$ norm of $\tau_n$ is uniformly bounded for all $n\geq N$ and extend $\tau_n\cdot \xi_n$ by zero to a function $\tilde{\xi}_{n}: W_4\rightarrow \mathbb R$. The gauge transformations $e^{i\tilde{\xi}(n)}\,u_{n}$ and $\hat{u}_{n}|_{[-T_{n}+1,-T_{n}+2m_{n}+3]\times Y}$ match on the overlap of their domains, therefore, we can glue them together to the desired gauge transformation $\tilde{u}_{n}: X_{T_n} \to S^1$. Since $\tilde{u}_{n}$ is obtained by modifying $u_{n}$ and $\bar{u}_{n,j}$ using the cutoff functions $\tilde{\xi}_{n}$, $\tilde{\xi}_{n,j}$ and the function $e^{i(a+bt)}$ the estimate of the lemma can be easily verified.
\end{proof}

%%%%%%%%%%%%%%%%%%%%%%%%%%%%%%%%%%%%%%%%%%%%%%%

\section{Gluing results}\label{S:gluing}
In this section, we will finish the proof of Theorem \ref{T: SW=Lefschetz} by first establishing a bijective correspondence between monopoles on $X_R$ and monopoles on $\wp$ for all sufficiently large $R$, and then matching the signs to identify $\# \M(X_R,g_{R},\hat{\mathfrak{p}}_R)$ with the Lefschetz number in the monopole chain complex $C^o$. To simplify notations, we will continue writing $\M(X_R)$ for $\M(X_R,g_{R},\hat{\mathfrak{p}}_{R})$.

\begin{thm}\label{SW=Lefschetz}
Assume that the spin Dirac operator $\D^+(\wp): L^2_1\,(\wp;\,\S^+) \to L^2\,(\wp;\S^-)$ is an isomorphism. Then, for all sufficiently large $R > 0$, the moduli space $\M (X_R)$ is regular, and there exists a homeomorphism
\begin{equation}\label{E:rho}
\rho:\; \M (X_R)\longrightarrow \mathop{\bigcup}\limits_{[\alpha]\in\,\mathfrak{C}^*} \M (W_{\infty}, [\alpha]).
\end{equation}
\end{thm}

%\begin{cor}\label{C:SW=Lefschetz}
%Under the above assumptions, $\#\,\M(X_R,g_R,\beta) = - \Lef\,(W_*: C^o \to C^o)$ for all sufficiently large $R > 0$
%\end{cor}

We will first prove Theorem \ref{SW=Lefschetz} by adopting the gluing techniques from \cite{KM} to the non-separating case at hand.  Theorem \ref{T: SW=Lefschetz} will be proved at the end of this section.  

%%%%%%%%%%%%%%%%%%%%%%%%%%%%%%%%%%%%%%%%%%%%%%%

\subsection{Fiber products}
We will be using notations from Section \ref{S:notations}. Denote by $\M^*(X_T)$ the moduli space of irreducible monopoles on $X_T$. It follows from Proposition \ref{eigenvalue estimate_z} that $\M^*(X_T) = \M(X_T)$ for all sufficiently large $T > 0$. The similarly defined moduli spaces $\M^*(W_{T'})$ and $\M^*(I_{T',T})$ will be infinite dimensional Hilbert manifolds because both $W_{T'}$ and $I_{T',T}$ have non-empty boundary but we are not imposing any boundary conditions. By the unique continuation theorem \cite[Section 10.8]{KM}, the restriction of an irreducible monopole to the boundary is irreducible, therefore, for all $0 < T' < T < \infty$ we have well defined restriction maps
\[
R^-_{T'}: \M^*(W_{T'}) \to \B^*(Y)\times \B^*(Y)\quad\text{and}\quad R^+_{T',T}: \M^*(I_{T',T}) \to \B^* (Y)\times \B^*(Y),
\]
where $\B^*(Y)$ consists of irreducible configurations in $\B (Y)$. One can show that these maps are embeddings of Hilbert manifolds. It will be convenient to extend these notations to the case of $T = \infty$. Recall that in Section \ref{S:notations} we defined $I_{T',\infty} = \wp - \operatorname{int} (W_{T'})$. Let $\M^* (I_{T',\infty})$ be the moduli space of monopoles $(A,\phi)$ on $I_{T',\infty}$ satisfying
\[
\lim\limits_{t \to +\infty}\left[(A,\phi)|_{\{t\}\times Y}\right]\; =\; \lim\limits_{t \to -\infty}\left[(A,\phi)|_{\{t\}\times Y}\right]\; \in\; \mathfrak C^*.
\]
Then the restriction map
\[
R^+_{T',\infty}: \M^* (I_{T',\infty}) \to \B^*(Y)\times \B^*(Y)
\]
is still well defined and is an embedding of Hilbert manifolds. To unify the notations, we will write $X_{\infty} = \wp$ and
\[
\M^*(X_{\infty}) = \bigcup\limits_{[\alpha]\in \mathfrak{C}^{*}} \M^*(\wp, [\alpha]).
\]
Then, for all $0 < T' < T \le \infty$, we have the following commutative diagram whose unmarked arrows are given by restriction to submanifolds

\begin {equation*}
\begin{tikzpicture}
\draw (4,4.5) node (a) {$\B^*(Y) \times \B^*(Y)\qquad\;$};
\draw (8,6) node (b) {$\M^*(W_{T'})$};
\draw (8,3) node (c) {$\M^*(I_{T',T})$};
\draw (12,4.5) node (d) {$\M^*(X_T)$};
\draw[->](b)--(a) node [midway,above](TextNode){$R^-_{T'}$\;};
\draw[->](c)--(a) node [midway,below](TextNode){$R^+_{T',T}$};
\draw[->](d)--(b) node [midway,above](TextNode){};
\draw[->](d)--(c) node [midway,above](TextNode){};
\end{tikzpicture}
\end {equation*}

\begin{lem}\label{solution as fiber product}
For all $0 < T' < T \le \infty$, the above diagram is a Cartesian square, that is, $\M^*(X_T)$ is homeomorphic to the fiber product
\[
\Fib\,(R^-_{T'},R^+_{T',T}) = \left\{(x,y) \in \M^*(W_{T'})\times \M^*(I_{T',T}) \mid R^-_{T'}(x) = R^+_{T',T}(y) \right\}.
\]
\end{lem}

\begin{proof}
The proof is identical to that of \cite[Lemma 19.1.1]{KM}, which deals with the separating case, and will be omitted. Note that some new issues would appear if we were to glue reducible monopoles in the blown-up moduli space but we do not deal with them here.
\end{proof}

In less formal terms, Lemma \ref{solution as fiber product} asserts that the moduli space $\M^*(X_T)$ is the intersection of the moduli spaces $\M^*(W_{T'})$ and $\M^*(I_{T',T})$ viewed as submanifolds of $\B^*(Y)\,\times\,\B^*(Y)$. We will prove Theorem \ref{SW=Lefschetz} by showing that all of these intersections occur in small neighborhoods of constant trajectories $\gamma_{\alpha}$, where the intersection points for sufficiently large $T$ can be matched with those for $T = \infty$ using implicit function theorem.

%%%%%%%%%%%%%%%%%%%%%%%%%%%%%%%%%%%%%%%%%%%%%

\subsection{Technical lemmas}
The central role in our proof will be played by the following theorem, which is a special case of \cite[Theorem 18.2.1]{KM}. When $T$ is finite, $I_T$ will stand for $[-T,T]\times Y$ and $R^+_T$ for the corresponding restriction map.

\begin{thm}\label{near constant solution on finite cylinder}
There exists a constant $T_1 > 0$ such that for all $T\geq T_1$ and $[\alpha] \in \mathfrak C^*$, there exist smooth maps
\begin{gather*}
u_{[\alpha]}(T,-): B([\alpha]) \to \M^* (I_T) \\
u_{[\alpha]}(\infty,-): B([\alpha])\to \M^* (I_{\infty})
\end{gather*}
which are diffeomorphisms from an open neighborhood $B([\alpha])\subset \B^*(Y)$, which is independent of $T$, onto neighborhoods of the constant solutions $[\gamma_{\alpha}]$. Moreover, the maps
\[
\mu_{[\alpha],T}\; =\; R^+_T\circ u_{[\alpha]} (T,-)\,: B([\alpha]) \to \B^*(Y)\times \B^*(Y)
\]
are smooth embeddings for all $T \in [T_1,\infty]$, and we have a $C^{\infty}_{\loc}$ convergence
\[
\mu_{[\alpha],T}\; \longrightarrow\;\mu_{[\alpha],\infty}\quad\text{as}\quad T \to \infty.
\]
Finally, there exists a constant $\eta > 0$, independent of $T$, such that the image of the map $u_{[\alpha]}(T,-)$ contains all the trajectories $[\gamma]\in \M^* (I_T)$ with $\|\gamma-\gamma_{\alpha}\|_{L^2_k(I_T)}\leq \eta$.
\end{thm}

\begin{rmk}\label{separating}
In addition, we will assume that, for all $S,T\in [T_1,\infty]$ and all $[\alpha]\neq [\beta] \in \mathfrak C^*$,
$$\operatorname{im}(u_{[\alpha]}(S,-))\cap \operatorname{im}(u_{[\beta]}(T,-))=\emptyset,\quad\operatorname{im}(\mu_{[\alpha]}(S,-))\cap \operatorname{im}(\mu_{[\beta]}(T,-))=\emptyset.$$
\end{rmk}

\begin{lem}\label{long neck near constant}
Let $\eta > 0$ be as in Theorem \ref{near constant solution on finite cylinder}. Then one can find constants $0 < T_2 < T_3 < \infty$ with the following significance: for any $T \in [T_3,\infty]$, any element of $\M^*(X_T)$ can be represented by a monopole $(A,\phi)$ such that $\|(A,\phi)|_{I_{T_2,T}}-\gamma_{\alpha}\|_{L^2_k (I_{T_2,T})}\, \leq\, \eta/2$ for some $[\alpha]\in \mathfrak C^*$.
\end{lem}

\begin{proof}
The case $T = \infty$ should be clear once we remember that, for $T = \infty$, the notation $I_{T_2,\infty}$ means $W_{\infty} - \operatorname{int} (W_{T_2})$. Let us now assume that $T < \infty$ and suppose to the contrary that the constants $T_2$ and $T_3$ do not exist. Then we can find two sequences of real numbers $T'_n < T_n$, both going to infinity as $n \to \infty$, and a sequence of monopoles $(A_n,\phi_n)$ on $X_{T_n}$ such that, for any gauge transformation $u_n: I_{T_n', T_n}\times Y \to S^1$ and any $\alpha \in \mathfrak C^*$, we have
\[
\|u_n\cdot (A_n,\phi_n)|_{I_{T_n', T_n}} - \gamma_{\alpha}\|_{L^2_k(I_{T_n', T_n})}\; >\; \eta/2
\]
for all $n$. According to Theorem \ref{compactness}, after passing to a subsequence, we may assume that $[(A_{n},\phi_{n})]$ converges to an element in $\M (\wp,[\alpha'])$ for some $[\alpha']\in\mathfrak C^*$. This leads to a contradiction with part (2) of Definition \ref{convergence} of convergence.
\end{proof}

\begin{lem}\label{solution as fiber product II}
Let $T_2$ and $T_3$ be the constants from Lemma \ref{long neck near constant}. Then for any $T\in [T_{3},+\infty]$, there exists a homeomorphism
\[
\M^*(X_T)\; = \bigcup\limits_{[\alpha]\in \mathfrak C^*}\Fib\, (R^-_{T_2},\,\mu_{[\alpha],T-T_2})
\]
between $\M^*(X_{T})$ and a disjoint union of the fiber products. Moreover, the moduli space $\M^*(X_{T})$ is regular if and only if, for any $[\alpha]\in \mathfrak C^*$, the images of the maps $R^-_{T_{2}}$ and $\mu_{[\alpha],T-T_2}$ intersect transversely in $\B^*(Y)\times \B^*(Y)$.
\end{lem}

\begin{proof}
The first assertion is a straightforward corollary of Theorem \ref{near constant solution on finite cylinder}, Lemma \ref{solution as fiber product}, and Lemma \ref{long neck near constant}. The second assertion is essentially \cite[Theorem 19.1.4]{KM}.
\end{proof}

\begin{lem}\label{convergence on the boundary}
Let $B([\alpha]) \subset \B^*(Y)$ be an open neighborhood as in Theorem \ref{near constant solution on finite cylinder}, and let $x_n \in B([\alpha])$ be a sequence such that $\mu_{[\alpha],T_n}(x_n)\to \mu_{[\alpha],\infty}(x)$ for some $T_n \to \infty$ and some $x \in B([\alpha])$. Then $x_n \to x$ in the topology of $\B^*(Y)$.
\end{lem}

\begin{proof}
This will be clear once we recall the construction of the map $u_{[\alpha]}(T,-)$ from Section 18.4 of \cite{KM}. The tangent space $T_{[\alpha]}\B^*(Y)$, denoted by $\K$, has a decomposition $\K = \K_{\mathfrak{q}}^+ \oplus \K_{\mathfrak{q}}^-$ given by the spectral decomposition of the Hessian of $\L_{\mathfrak{q}}$. We will denote by $\Pi^{\pm}: \K \to \K_{\mathfrak{q}}^{\pm}$ the corresponding projections. We will also identify small open balls $B_{\epsilon}([\alpha]) \subset \K$ of radius $\epsilon > 0$ with open neighborhoods $B([\alpha])\subset\B^*(Y)$. According to \cite[Section 18.4]{KM}, for any sufficiently small $\epsilon > 0$, there exist constants $\epsilon' > 0$ and $T_1 > 0$ with the following significance: For any $T\in [T_1,\infty]$ and any $(a,b)\in B_{\epsilon} ([\alpha])\, \subset\, \K_{\mathfrak{q}}^+\oplus \K_{\mathfrak{q}}^-$, there exists a unique
\[
[\gamma^T_{(a,b)}]\;\in\; (R^+_T)^{-1}(B([\alpha])\times B([\alpha]))\;\subset\; \M^*(I_T)
\]
such that $\|\gamma^T_{(a,b)} - \gamma_{\alpha}\|_{L^2_k(I_T)}\,<\,\epsilon'$ and $(\Pi_+,\Pi_-)(R^+_T\, ([\gamma^{T}_{(a,b)}])) = (a,b)$. The map $u_{[\alpha]}(T,-)$ is then defined by the formula
\[
u_{[\alpha]}(T,(a,b)) = [\gamma^T_{(a,b)}].
\]
With this definition in place, write $x_n = (a_n,b_n)\in \K_{\mathfrak{q}}^+\oplus\K_{\mathfrak{q}}^-$ and similarly $x_{\infty}=(a_{\infty},b_{\infty})\in \K_{\mathfrak{q}}^+\oplus\K_{\mathfrak{q}}^-$. Then $\mu_{[\alpha],T_n}(x_n) = (a_n,*,*,b_n)$ and $\mu_{[\alpha],\infty}(x_{\infty}) = (a_{\infty},*,*,b_{\infty})$, where we embedded $B([\alpha])\times B([\alpha])$ into $\K_{\mathfrak{q}}^+\oplus\K_{\mathfrak{q}}^-\oplus\K_{\mathfrak{q}}^+\oplus\K_{\mathfrak{q}}^-$. From this we clearly see that $\mu_{[\alpha],T_n}(x_n)\to \mu_{[\alpha],\infty}(x)$ implies $x_n\to x_{\infty}$.
\end{proof}

%%%%%%%%%%%%%%%%%%%%%%%%%%%%%%%%%%%%%%%%%%%%%%

\subsection{Proof of Theorem \ref{SW=Lefschetz}}
Let $T_2$ and $T_3$ be the constants from Lemma \ref{long neck near constant}. Using Lemma \ref{solution as fiber product II} and our regularity assumption on $\M(\wp)$ we can claim that, for any $[\alpha] \in \mathfrak C^*$, the images of the maps $R^-_{T_2}$ and $\mu_{[\alpha],\infty}$ intersect each other transversely and we have a homeomorphism (bijection)
\[
\M^*(X_{\infty})\; = \bigcup\limits_{[\alpha]\in \mathfrak C^*}\Fib\, (R^-_{T_2},\,\mu_{[\alpha],\infty}).
\]
By Theorem \ref{near constant solution on finite cylinder}, the maps $\mu_{[\alpha],T}$ converge to $\mu_{[\alpha],\infty}$ in the $C^{\infty}_{\loc}$ topology as $T \to \infty$. The implicit  function theorem now implies that there exists a constant $T_4 \ge T_3$ such that any $(y,x) \in \Fib\,(R^-_{T_2},\,\mu_{[\alpha],\infty})$ has an open neighborhood $U(y,x) \subset \M^* (W_{T_2}) \times B_{[\alpha]}$ with the following significance: for any $T \ge T_4$, the images of embeddings $R^-_{T_2}$ and $\mu_{[\alpha],T-T_2}$ intersect each other in exactly one point in $U(y,x)$, and this intersection is transverse. Therefore, our proof will be finished once we prove the following lemma.

\begin{lem}
There exists a constant $T_5 \ge T_4$ such that, for any $T\geq T_5$ and any element of $\M^*(X_T)$ represented by
\[
(y',x')\; \in \bigcup\limits_{[\alpha]\in \mathfrak C^*} \Fib\, (R^-_{T_2},\,\mu_{[\alpha],T-T_2}),
\]
there exists a point $(y,x) \in \Fib\,(R^-_{T_2},\,\mu_{[\alpha],\infty})$ such that $(y',x')\in U(y,x)$.
\end{lem}

\begin{proof}
Suppose to the contrary that this is not the case. Then there is a sequence
\[
(y_n, x_n)\;\in\; \Fib\,(R^-_{T_2},\mu_{[\alpha_n],T_n}),\quad T_n \to \infty,
\]
representing monopoles $(A_n,\phi_n)$ on $X_{T_n}$ such that $(y_n,x_n) \notin U(y,x)$ for any $(y,x)$ in the fiber product $\Fib\,(R^-_{T_2},\,\mu_{[\alpha],\infty})$. By Theorem \ref{compactness}, after passing to a subsequence, we may assume that $[(A_{n},\phi_{n})]$ converges to a monopole $[(A_\infty, \phi_\infty)]$ on $X_{\infty}$. Represent the latter by $(y_{\infty},x_{\infty})\in \Fib\,(R^-_{T_2}, \mu_{[\alpha], \infty})$. Then
\begin{itemize}
\item[(a)] $y_{n}\rightarrow y_{\infty}$. This follows from Part (1) of Definition \ref{convergence} of convergence because $W_{T_2}$ is a compact subset of $W_{\infty}$ and
\[
y_n = [(A_n,\phi_n)|_{W_{T_2}}]\quad\text{and}\quad y_{\infty} = [(A_{\infty},\phi_{\infty})|_{W_{T_2}}].
\]
\item[(b)] $x_{n}\rightarrow x_{\infty}$. This follows from Remark \ref{separating}, which implies that $[\alpha_{n}] = [\alpha]$ for all large enough $n$, and from Lemma \ref{convergence on the boundary} applied to the convergent sequence
\[
\mu_{[\alpha_n],T_n}(x_n)\, =\, R^-_{T_2}(y_{n})\;\longrightarrow\; R^-_{T_2}(y_{\infty})\, =\, \mu_{[\alpha],\infty}(x_{\infty}).
\]
\end{itemize}
Therefore, $(y_n,x_n) \in U(y_{\infty},x_{\infty})$ for all sufficiently large $n$. This gives a contradiction.
\end{proof}

%%%%%%%%%%%%%%%%%%%%%%%%%%%%%%%%%%%%%%%%%%%%

\subsection{Proof of Theorem \ref{T: SW=Lefschetz}}
All we need to do is compare the signs with which the mono\-poles corresponding to each other under the map \eqref{E:rho} are counted in $\#\,\M(X_R)$ and $\Lef (W_*: C^o \to C^o)$. This is done in \cite[Proposition 3]{F} under a different $\mathbb{Z}/2$ grading convention. Since our setting here is slightly different, we give an alternative argument using excision principle.

In the product case, $X = S^1 \times Y$, the orientation transport argument of \cite{MRS1} (see also \cite{RS2} in the instanton setting) can be used to show that 
\begin{equation}\label{E:sign}
\#\,\M(X)\, =\, \pm\, \chi\,(C^o)
\end{equation}
up to an overall sign which is independent of the choice of $Y$. The sign is determined by the sign fixing condition
\begin{equation}\label{E:casson}
\lsw (X) = - \lambda (Y)
\end{equation}
of \cite[Section 11.2]{MRS1}, where $\lambda(Y)$ is the Casson invariant normalized so that $\lambda(\Sigma(2,3,5)) = -1$. To calculate the sign in \eqref{E:sign}, we will let $X = S^1 \times Y$, where $Y = \Sigma(2,3,7)$ is the Brieskorn homology sphere endowed with a natural metric $h$ realizing its Thurston geometry, and compute 
\begin{equation}\label{E:lsw}
\lsw (X)\; =\; \#\,\M(X,g)\, -\, w(X,g)
\end{equation}
with respect to the product metric $g = dt^2 + h$. The correction term in this formula equals 
\[
w(X,g)\; =\; \ind \D^+ (Z_{\infty})\, +\, \frac 1 8\,\sign Z, 
\]
where $Z$ can be any smooth compact spin manifold with boundary $Y$. Let us choose $Z$ to be the plumbed manifold with the intersection form isomorphic to $E_8\,\oplus\,H$, where $E_8$ is positive definite. According to \cite[Section 6]{RS:ded}, the index of $\D^+(Z_{\infty})$ vanishes, therefore, $w(X,g) = 1$. Since $\lambda (Y) = -1$ we conclude from formulas \eqref{E:casson} and \eqref{E:lsw} that $\#\,\M(X,g) = 2$. This needs to be compared to the Euler characteristic of the chain complex $C^o$. The latter complex is known \cite{moy} to have exactly two generators of the same grading. Note that 
\[
-2h(Y)=0 \; > \; \gr^{\mathbb{Q}}([\theta_{0}]) = -2w(X,g) = -2.
\]
By \cite[Lemma 2.9]{lin:sw-psc}, the boundary map $\partial ^{o}_{s}: C^{o}\rightarrow C^{s}$ must be non-zero. Therefore, both generators of $C^{o}$ must be of odd grading, and the overall sign in formula \eqref{E:sign} is a minus.

The general case now follows by using the excision principle for determinant bundles as in \cite[Section 25]{KM}. 

%%%%%%%%%%%%%%%%%%%%%%%%%%%%%%%%%%%%%%%%%%%%

\section{Generic metrics}\label{S:generic}
This section contains two results about metrics with no harmonic spinors on two types of spin manifolds: compact manifolds with product regions and non-compact manifolds with cylindrical ends. These results are similar to those of Amman, Dahl, and Humbert \cite{ADH} and are proved by a  modification of their argument.

%%%%%%%%%%%%%%%%%%%%%%%%%%%%%%%%%%%%%%%%%%%%

\subsection{Manifolds with product regions}
Let $X$ be a connected smooth spin compact manifold of dimension $n \equiv 0 \pmod 4$ and $f: X \to S^1$ a smooth map with the primitive cohomology class $[df] \in H^1 (X; \Z)$. Let $Y \subset X$ be a connected manifold Poincar\'e dual to $[df]$, and introduce Riemannian metrics $h$ on $Y$ and $g$ on $X$ so that $g = dt^2 + h$ in a product region $[-\ep,\ep] \times Y \subset X$. According to \cite[Theorem 1.1]{ADH}, the spin Dirac operator $\D(Y,h)$ is invertible in the Sobolev $L^2$ completion for a generic choice of $h$.

\begin{thm}\label{T:one}
Let $X$ be a manifold as above with a product region $[-\ep,\ep] \times Y$ and suppose that the metric $h$ is such that $\ker \D(Y,h) = 0$. If $X$ is spin cobordant to zero then there exists a metric $g$ on $X$ such that $g = dt^2 + h$ in the product region and $\D^+ (X,g)$ is invertible in the Sobolev $L^2$ completion.
\end{thm}

\begin{proof}
We will use the method of transporting invertibility of the operator $\D^+ (X,g)$ across a spin cobordism as in \cite{ADH} but we will be more specific in choosing the cobordism. 

Start with the manifold $X' = S^1 \times Y$ with the product metric $g' = dt^2 + h$ and the spin Dirac operator $\D^+ (X',g') = dt\,\cdot\,(\p/\p t + \D (Y,h))$. Since the operator $\D (Y,h)$ is invertible, so is the operator $\D^+ (X',g')$ by a direct calculation.

We claim that there is a spin cobordism from $X'$ to $X$ that contains a product region $I \times [-\ep,\ep] \times Y$, and furthermore has handles of index at most $n - 1$.  To prove this, write 
\[
X = W\, \cup_{\{-\ep,\ep\}\times Y}\, ([-\ep,\ep] \times Y)\quad\text{and}\quad X' = V \cup_{\{-\ep,\ep\}\times Y} [-\ep,\ep] \times Y;
\]
see Figure \ref{F:spin}. We assume of course that $V$ has a product structure as well.  Let us consider
\[
U = W\, \cup\, ([0,1] \times \{-\ep,\ep\}\times Y)\, \cup\, V
\]
where $\{0\} \times \{-\ep,\ep\}\times Y$ is identified with $\partial V$ and $\{1\} \times \{-\ep,\ep\}\times Y$ is identified with $\partial W$.  
After rounding corners, $U$ is a spin manifold diffeomorphic to $X$ hence $U$ is the boundary of a connected spin manifold $R$ of dimension $n+1$. View $R$ as a cobordism from $V$ to $W$, relative to the product region $[0,1] \times  \partial V$, and give it a handle decomposition. 

\smallskip

\begin{figure}[h]
\labellist
\small\hair 0mm
\pinlabel {$\{0\} \times [-\ep,\ep] \times Y$}  at 150 -10
 \pinlabel {$[0,1] \times [-\ep,\ep] \times Y$}  at 150 50
 \pinlabel {$\{1\} \times [-\ep,\ep] \times Y$}  at 150 107
 \pinlabel {$W$}  at 20 130
 \pinlabel {$V$}  at 15 11
 \pinlabel {$R$}  at 250 70
\endlabellist
\centering
\includegraphics[scale=1]{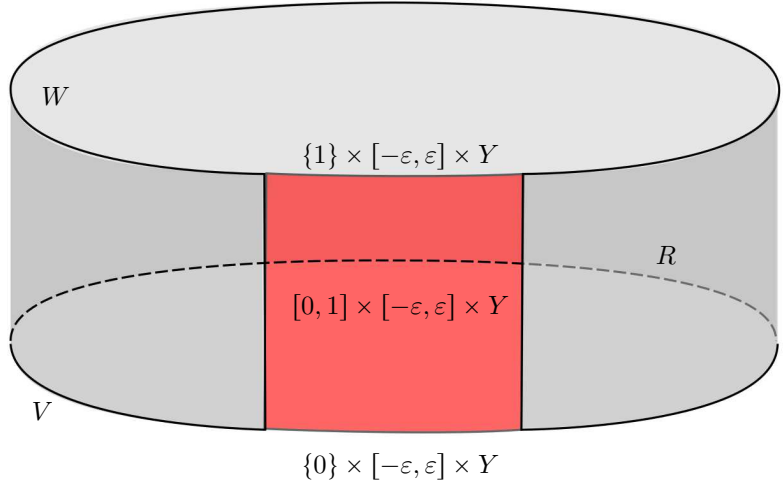}
\vspace*{1ex}
\caption{The spin cobordism}
\label{F:spin}
\end{figure}

By surgeries on generators of $\pi_1(R)$, preserving the spin condition, we can make $R$ simply connected, and cancelling the $0$-handles and $(n+1)$-handles, we may assume that all the handles have index between $1$ and $n$.  Then a standard handle trading argument can be used to eliminate the handles of index $1$ and $n$.  Gluing this to the product region $I \times [-\ep,\ep] \times Y$ gives the desired cobordism.

According to \cite[Theorem 1.1]{ADH}, the manifold $X$ admits a metric $g$ such that $\ker \D (X,g) = 0$. This explicitly constructed metric matches the original product metric on $X'$ away from arbitrarily thin tubular neighborhoods of the surgery spheres. By construction, these spheres lie outside of the product region, so that the metric remains a product there. Since $X$ is spin cobordant to zero, we conclude that $\hat{A}(X) =0$. Therefore, the operator $\D^+ (X,g)$ has zero index and must be invertible.
\end{proof}

\begin{rmk}
The existence of just a single metric as in Theorem \ref{T:one} implies that the set of such metrics is in fact generic (that is, $C^0$ open and $C^{\infty}$ dense) in the space of all metrics on $X$ with a fixed product metric on $[-\ep,\ep] \times Y$. The proof of \cite[Proposition 3.1]{M} goes through after one notes that fixing a product metric on $[-\ep,\ep] \times Y$ defines a convex subset in the space of all metrics.
\end{rmk}

%%%%%%%%%%%%%%%%%%%%%%%%%%%%%%%%%%%%%%%%%%%%%%%%

\subsection{Manifolds with product ends}
Let $X$ be a manifold from the previous section with a product region $[-\ep,\ep] \times Y$ and a metric $g$ which in the product region takes the form $g = dt^2 + h$. Cut $X$ open along $Y = \{0\}\times Y$ into a cobordism $W$, and attach infinite product ends to $W$. This results in the non-compact manifold 
\[
\wp = ((-\infty,0]\times Y)\,\cup\,W\, \cup\, ([0,\infty) \times Y).
\]

\begin{thm}\label{T:two}
Let us suppose that $\ker \D(Y,h) = 0$ and that $X$ is spin cobordant to zero. Then one can find a metric $g$ on $\wp$ such that $g = dt^2 + h$ on the product ends and $\ker \D^+ (\wp,g) = 0$ in the Sobolev $L^2$ completion.
\end{thm}

\begin{rmk}\label{R:two}
Note that the conditions of this theorem are obviously satisfied for spin 4-manifolds $X$ with integral homology of $S^1 \times S^3$ because the spin cobordism class of such a manifold is determined by its $\hat A$--genus, 
\[
\hat A (X)\; =\; -\,\frac 1 8 \sign (X)\, =\, 0.
\]
\end{rmk}

The rest of this section will be dedicated to the proof of Theorem \ref{T:two}. We showed in the proof of Theorem \ref{T:one} that $X$ can be obtained from $X' = S^1 \times Y$ by performing surgery along spheres disjoint from the product region $[-\ep,\ep] \times Y$. In the language of manifolds with product ends, this implies that $\wp$ can be obtained from the product manifold $\mathbb R \times Y$ by performing surgeries inside a compact region in $\mathbb R \times Y$. Since $\ker \D(Y,h) = 0$, the Sobolev $L^2$ completion of the operator $\D^+ (X')$ is obviously invertible. Therefore, we can proceed with the construction of a desired metric on $\wp$ exactly as in \cite{ADH} making some changes along the way to account for the non-compactness of $\wp$.

The first change comes up in the proof of \cite[Lemma 3.4]{ADH} which uses Rellich Lemma to conclude that an $L^2_1$ bounded sequence of spinors contains a strongly convergent subsequence in $L^2$. While Rellich Lemma fails on non-compact manifolds, the sequence of \emph{harmonic} spinors in Lemma 3.4 still admits a strongly convergent subsequence: we first apply Rellich Lemma on the compact manifold $([-1,0]\times Y)\,\cup\,W\, \cup\, ([0,1] \times Y)$, and then use the following estimate.

\begin{lem}\label{L}
Let $X = [\,0,\infty) \times Y$ be a manifold with product metric $g = dt^2 + h$ such that $\ker \D(Y,h) = 0$. Then there exists a constant $C > 0$ such that, for any harmonic $L^2$--spinor $\varphi$ on $X$,
\[
\|\varphi\|_{L^2([\,0,\infty)\times Y)}\;\le\;C\cdot \|\varphi\|_{L^2([0,1]\times Y)}.
\]
\end{lem}

\begin{proof}
Let $\{\,\psi_{\lambda}\}$ be an orthonormal basis of eigenspinors of $\D(Y,h)$ and let $\lambda_0 > 0$ stand for the smallest positive $\lambda$. Since $\varphi$ is an $L^2$--spinor in the kernel of the operator $\D^+(X,g) = dt\cdot (\p/\p t + \D(Y,h))$, it takes the form
\[
\varphi (t,y)\; =\; \sum_{\lambda > 0}\; a_{\lambda}\cdot e^{-\lambda t}\cdot \psi_{\lambda} (y).
\]
A direct calculation with this formula gives
\[
\|\varphi\|^2_{L^2([\,0,\infty)\times Y)}\; =\; \int_0^{\infty}\;\|\varphi\|^2_{L^2(Y)}\,dt\; =\; \sum_{\lambda > 0}\; \frac {\;\,|a_{\lambda}|^2}{2\lambda}
\]
and 
\[
\|\varphi\|^2_{L^2([0,1]\times Y)}\; =\; \int_0^1\;\|\varphi\|^2_{L^2(Y)}\,dt\; =\; \sum_{\lambda > 0}\; \frac {\;\,|a_{\lambda}|^2}{2\lambda}\cdot\left(1 - e^{-2\lambda}\right).
\]
This leads to the desired estimate with the constant $C = 1/\sqrt{\,1 - e^{-2\lambda_0}}$.
\end{proof}

\noindent
The rest of the proof of Lemma 3.4 goes through using exhaustion of the complement of the surgery sphere $S$ in $\wp$ by the compact sets 
\[
([-1/\ep,0]\times Y)\,\cup\,W\, \cup\, ([0,1/\ep] \times Y) - U_S (\ep),
\]
where $U_S (\ep)$ is the open tubular neighborhood of $S$ of radius $\ep > 0$. All the lemmas used in that proof are already proved in \cite[Section 2]{ADH} without the compactness assumption.

The second change comes up in the proof of Step 2 on page 537 of \cite{RS}. That step goes through using exhaustion by the compact sets
\[
([-Z,0]\times Y)\,\cup\,W\, \cup\, ([0,Z] \times Y) - U_S (1/Z)
\]
for positive integers $Z$. In Step 3, convergence in $C^1_{\rm{loc}}(\wp - S)$ of a sequence of harmonic spinors implies only $L^2$ convergence on compact subsets of $\wp - U_S (s)$. To obtain the desired $L^2$ convergence on the entire $\wp - U_S (s)$, we use Lemma \ref{L} one more time.

%%%%%%%%%%%%%%%%%%%%%%%%%%%%%%%%%%%%%%%%%%%%

\section{Periodic $\eta$-invariants}\label{S:eta}
Let $X$ be a connected smooth spin compact manifold of dimension $n \equiv 0 \pmod 4$ and $f: X \to S^1$ a smooth map such that the cohomology class $[df] \in H^1 (X; \mathbb Z)$ is primitive. Choose a connected manifold $Y \subset X$ Poincar\'e dual to $[df]$. We will assume that the manifold $Y$ with the induced spin structure is a spin boundary and that the $\hat A$--genus of $X$ vanishes; both of these conditions are automatic when $X$ is a homology $S^1 \times S^3$. Define the Riemannian manifold $X_R$ with long neck as in \eqref{E:XRn}. Recall that the metric $g_R$ on $X_R$ takes the form $g_R = dt^2 + h$ along the neck.
%which is a product in a normal neighborhood $[-1,1] \times Y$. For any $R \ge 1$, define the Riemannian manifold $X_R$ by replacing $[-1,1] \times Y$ with $[-R,R]\times Y$. Let $W$ be the cobordism from $Y$ to itself obtained by cutting $X$ open along $\{0\} \times Y$.
Consider the holomorphic family
\[
\D^{\pm}_z\; =\; \D^{\pm} (X_R,g_R) - \ln z\cdot df,\quad z \in \mathbb C^*.
\]
Under the assumption that $\ker \D^+_z = 0$ for all $z$ on the unit circle $|z| = 1$, the periodic $\eta$--invariant $\eta (X_R)$ was defined in \cite{MRS3} by the formula

\begin{equation}\label{E:eta}
\eta (X_R) = \frac 1 {\pi i}\,\int_0^{\infty}\oint_{|z| = 1}\;
\Tr \left(df\cdot \D^+_z \exp (-t \D^-_z \D^+_z)\right)\,\frac {dz} z\,dt.
\end{equation}

\medskip\noindent
It follows from \cite[Theorem A and Remark 5.4]{MRS3} that the so defined periodic $\eta$--invariant is independent of the choice of $f$ as long as $df$ is supported in the product region in $X_R$, which we will assume from now on.

\begin{thm}\label{T:eta}
Let $\eta(Y)$ be the Atiyah--Patodi--Singer $\eta$--invariant \cite{aps:I} of the Dirac operator $\D = -\D(Y,h)$.\footnote{The change in sign is dictated by the different conventions for the spin Dirac operator with respect to the product metric $g = dt^2 + h$ used in this paper and in \cite{MRS3}.} Assume that the operator $\D$ has zero kernel, as does the $L^2$--completion of the operator $\D^+(\wp)$ on the manifold $\wp$ obtained from $W$ by attaching infinite product ends. Then the invariants $\eta (X_R)$ are well--defined and, for all sufficiently large $R$,
\[
\eta(X_R) = \eta(Y).
\]
\end{thm}

\begin{proof}
The well--definedness of $\eta(X_R)$ for all sufficiently large $R$ follows from Proposition \ref{eigenvalue estimate_z}. According to \cite[Theorem A and Remark 5.4]{MRS3}, for any spin manifold $\zp(X_R)$ whose periodic end is modeled on $(X_R, g_R)$ we have
\[
\ind \D^+ (\zp(X_R))\; =\; \int_Z \widehat A(Z)\; -\; \frac 1 2\,\eta\,(X_R).
\]
A similar formula applied to an end-periodic manifold $\zp$ whose end is modeled on $(S^1 \times Y, dt^2 + h)$ yields
\[
\ind \D^+ (\zp)\; =\; \int_{Z} \widehat A(Z)\; -\; \frac 1 2\,\eta\,(S^1 \times Y).
\]
We know from \cite[Section 6.3]{MRS3} that $\eta\,(S^1 \times Y) = \eta\,(Y)$. By subtracting the above formulas from each other, we conclude that $\eta (X_R)$ must differ from $\eta (Y)$ by an even integer. The statement of the theorem will follow as soon as we prove that $\eta (X_R)$ and $\eta (S^1 \times Y)$ can be made arbitrarily close by choosing sufficiently large $R$. The proof of this will occupy the rest of this section.
\end{proof}

\begin{rmk}
For any $X$ which is spin cobordant to zero, all of the conditions of Theorem \ref{T:eta} are satisfied for the right choice of metric; see Theorem \ref{T:two}.
\end{rmk}

\begin{rmk}
Because of the periodic index theorem \cite{MRS3} the statements of Theorem \ref{thm} and Theorem \ref{T:eta} are essentially equivalent to each other. What follows is an independent proof of Theorem \ref{T:eta} using heat kernel techniques. We chose to include this proof, inspired by \cite{DW}, because it may be of interest in its own right.
\end{rmk}

%%%%%%%%%%%%%%%%%%%%%%%%%%%%%%%%%%%%%%%%%%%%%%

\subsection{Heat kernel estimates on $S^1_r \times Y$}
Let $S^1_r$ a circle of length $r$ and consider a smooth function $f$ on its universal cover, the real line, such that $f(t+r) = f(t) + 1$. The cohomology class of $df$ generates $H^1 (S^1_r; \mathbb Z)$, and we will write $df = f'(u)\,du$ with respect to the natural parameter $u$ on the circle $S^1_r$. Consider the family of elliptic operators $\p/\p u - \ln z\cdot f' (u)$ with $z \in \mathbb C^*$. If $K^r_z (t;u,v)$ is the kernel of the operator $\exp\,(t\cdot (\p/\p u - \ln z\cdot f')^2)$ then $(\p/\p u - \ln z\cdot f')\,K_z^r (t;u,v)$ is the kernel of the operator
\[
(\p/\p u - \ln z\cdot f')\, \exp\,(t\cdot (\p/\p u - \ln z\cdot f')^2).
\]

\begin{lem}\label{L:taubes}
There are positive constants $\gamma$ and $C > 0$ independent of $z$ and $r$ such that the following estimates hold for all unitary $z$, $r \ge 1$, and $t > 0$:
\begin{gather}
\left|K^r_z (t;u,v)\right|\;\le\;C\cdot t^{-1/2}\cdot e^{-\gamma (u-v)^2/t}\quad\text{and} \notag \\
\left|(\p/\p u - \ln z\cdot f')\,K^r_z (t;u,v)\right|\;\le\;C\cdot t^{-1}\cdot e^{-\gamma (u-v)^2/t}. \notag
\end{gather}
\end{lem}

\begin{proof}
We begin by observing that, for the purpose of making kernel estimates on the circle $|z| = 1$, one may assume that $f'(u)$ is constant and is therefore equal to $1/r$. This can be seen from the formula
\[
z^h\cdot (\p/\p u - \ln z\cdot f'(u)) \cdot z^{-h}\; =\; \p/\p u - \ln z\cdot (f'(u) + h'(u)),
\]
which holds for any function $h$ defined on $S^1_r$, and the fact that adding $dh$ does not change the cohomology class of $df$. Choosing $h(u) = u/r - f(u)$ will then do the job. Conjugating the operator by the unitary complex number $z^{h(x)}$ multiplies the kernel by $z^{h(x) - h(y)}$ hence preserves its norm.

Let $z = e^{is}$ then the kernel $K^r_z (t;u,v)$ of $\exp(t\cdot (\p/\p u - is/r))$ is given by the formula
\[
K^r_z (t;u,v)\; =\; \frac 1 r\,\cdot\; \sum_{k \in \mathbb Z}\;  e^{-t\,(2\pi k-s)^2/r^2} \cdot e^{2\pi iku/r} \cdot e^{-2\pi ikv/r}
\]
with respect to the orthonormal basis $r^{-1/2}\,e^{2\pi iku/r}$ on the circle $S^1_r$. One can easily verify that
\[
K^r_z (t;u,v)\; =\; \frac 1 r\,\cdot\, K^1_z (t/r^2; u/r,v/r),
\]
where $K^1_z (t;u,v)$ stands for the heat kernel of $\exp\,(t\cdot (\p/\p u - is)^2)$ on the circle of length one. It is well known that there exist positive constants $\gamma$ and $C$ independent of $z$ such that
\[
|K^1_z (t;u,v)|\; \le\; C\cdot t^{-1/2}\cdot e^{-\gamma (u-v)^2/t}\quad\text{for all\; $t > 0$}.
\]
But then
\begin{multline}\notag
|K^r_z (t;u,v)|\; \le\; C\cdot \frac 1 r\cdot \left(\frac t {r^2}\right)^{-1/2}\cdot\; e^{-\gamma (u/r - v/r)^2/(t/r^2)} \\ \le\; C\cdot t^{-1/2}\cdot e^{-\gamma (u - v)^2/t}\quad\text{for all\; $t > 0$}
\end{multline}
with the same constants $\gamma$ and $C$ independent of $z$ and $r$. A similar calculation proves the estimate on $(\p/\p u - is/r)\,K^r_z (t;u,v)$ as well; compare with \cite[Example 2.5]{DW}.
\end{proof}

Given a closed spin Riemannian manifold $Y$ of dimension $n-1$, consider the chiral spin Dirac operators $\D^{\pm}$ on $S^1_r\times Y$ and their twisted versions $\D^{\pm}_z = \D^{\pm} - \ln z\cdot df$. We wish to derive estimates on the kernels of the operators $\exp(-t\D^-_z\D^+_z)$ and $\D^+_z \exp(-t\D^-_z\D^+_z)$ which are uniform in $z$ on the unit circle $|z| = 1$ and in $r$. Denote by $K_z (t;x,y)$ the kernel of the operator $\exp(-t\D^-_z\D^+_z)$ then $\D^+_z K_z (t;x,y)$ is the kernel of $\D^+_z \exp (-t\D^-_z\D^+_z)$.

\begin{lem}\label{L:est1}
Suppose that the Dirac operator $\D$ on $Y$ has zero kernel. Then there exist positive constants $\gamma$ and $C$ independent of $z$ and $r$ such that the following estimates hold for all unitary $z$, $r \ge 1$, and $t > 0$:
\begin{gather}
|K_z (t;x,y)|\;\le\;C\cdot t^{-n/2}\cdot e^{-\gamma\,d^2(x,y)/t}\quad\text{and}\; \label{E:K} \\
|\D_z K_z (t;x,y)|\;\le\;C\cdot t^{-(n+1)/2}\cdot e^{-\gamma\,d^2(x,y)/t}. \label{E:DK}
\end{gather}
\end{lem}

\begin{proof}
On the circle $|z| = 1$, we have $\D_z^+ = du\, (\p/\p u - \D - \ln z\cdot f')$ and $\D_z^- = (\p/\p u + \D - \ln z\cdot f')\, du$ hence $\D_z^-\D_z^+ = - (\p/\p u - \ln z \cdot f')^2 + \D^2$ and the kernel of $\exp(-t \D_z^-\D_z^+)$ is the product of the kernels of $\exp\,(t\cdot (\p/\p u - \ln z\cdot f')^2)$ and $\exp(-t \D^2)$. To obtain estimate \eqref{E:K}, simply combine the estimate $C\cdot t^{-1/2}\cdot e^{-(u-v)^2/4t}$ on the former kernel with the estimate $C\cdot t^{-(n-1)/2}\cdot e^{-\gamma\,d^2(x',y')/t}$ for $x', y' \in Y$ on the latter; see Lemma \ref{L:taubes} and \cite[Proposition 1.1]{DW}, respectively. To obtain estimate \eqref{E:DK}, write the operator $\D^+_z\exp(-t\D^-_z \D^+_z)$ in the form
\begin{equation}\label{E:DK1}
\begin{split}
du\, \left[(\p/\p u - \ln z\cdot f')\,\exp\,(t\cdot (\p/\p u - \ln z\cdot f')^2)\right]\cdot \exp(-t \D^2) \hspace{0.4in} \\ -
du\cdot\exp\,(t\cdot (\p/\p u - \ln z\cdot f')^2)\cdot \left[\D \exp(-t\D^2)\right]
\end{split}
\end{equation}
and apply the estimates of Lemma \ref{L:taubes} and \cite[Proposition 1.1]{DW} twice.
\end{proof}

\begin{lem}\label{L:est0}
Suppose that the Dirac operator $\D$ on $Y$ has zero kernel. Then there are positive constants $\mu$ and $C$ independent of $z$ and $r$ such that the following estimates hold for all unitary $z$, $r \ge 1$, and $t \ge 8$:
\[
|K_z (t;x,y)|\;\le\;C\cdot e^{-\mu t}\;\;\text{and}\quad |\D_z K_z (t;x,y)|\;\le\;C\cdot e^{-\mu t}.
\]
\end{lem}

\begin{proof}
We use again the fact that the kernel of $\exp(-t \D_z^-\D_z^+)$ is the product of the kernels of $\exp\,(t\cdot (\p/\p u - \ln z\cdot f')^2))$ and $\exp(-t \D^2)$. The former is uniformly bounded for all $t \ge 8$ by Lemma \ref{L:taubes}. As for the latter, note that the smallest eigenvalue of $\D^2$ is positive. Denote this eigenvalue by $\lambda^2$ then the kernel of $\exp(-t\,\D^2)$ can be estimated from above by $C\cdot e^{-\lambda^2 t/2}$ by \cite[Proposition 1.1]{DW}. The argument for the kernel of $\D_z K_z (t;x,y)$ is similar using equation \eqref{E:DK1} together with \cite[Proposition 1.1]{DW}.
\end{proof}

%%%%%%%%%%%%%%%%%%%%%%%%%%%%%%%%%%%%%%%%%%%%%

\subsection{Heat kernel estimates on $X_R$}
In this section, we will prove certain estimates on the heat kernels on manifolds $X_R$. To begin with, we will give a description of $X_R$ which differs notationally from that in \eqref{E:XRn}.

Let $Y$ be a connected submanifold of $X$ which is Poincar\'e dual to $df$, and let $W$ be the cobordism obtained by cutting $X$ open along $Y$. Assume that the Riemannian metric on $X$ is a product metric in a normal neighborhood $[-1,1] \times Y$. The induced metric on $W$ will have product regions $[-1,0] \times Y$ and $[0,1] \times Y$ near its boundary components. We will use these product regions to define, for every real number $R \ge 1$, the manifold
\[
X_R = W\; \cup\; ([-R-1,R+1] \times Y)
\]
by gluing the product region $[-1,0] \times Y$ of the first summand to $[-R-1,-R] \times Y$ of the second, and the product region $[0,1] \times Y$ of the first summand to $[R,R+1] \times Y$ of the second. The gluing functions we use are linear on the first factor and are the identity on the second. We will view the manifold $[-R-1,R+1] \times Y$ with the identified boundary components as the product $S^1_R \times Y$ with the circle $S^1_R$ of circumference $2R + 2$, cut open along a copy of $Y$.

Throughout this section we assume that the Dirac operator $\D$ on $Y$ has zero kernel,
%$R \ge 2$ and that
$df$ is supported in the product region $[-1,1] \times Y$, and $f$ is a function of the normal coordinate $u$ in that region.

%%%%%%%%%%%%%%%%%%%%%%%%%%%%%%%%%%%%%%%%%%%%%

\subsubsection{Gaussian estimates}
Denote by $K_z (t;x,y)$ the kernel of the operator $\exp(-t\,\D^-_z \D^+_z)$ on $X_R$ and by $K^1_z (t;x,y)$ and $K^2_z (t;x,y)$ the kernels of the operators $\exp(-t\,\D^-_z \D^+_z)$ on, respectively, $X$ and $S^1_R \times Y$. We wish to compare the functions $K_z (t;x,x)$ and $K_z^2 (t;x,x)$ over the product region $[-1,1] \times Y$ shared by the manifolds $X_R$ and $S^1_R \times Y$. To this end, define an approximate kernel of the operator $\exp(-t \D^-_z \D^+_z)$ on $X_R$ by the formula
\begin{equation}\label{E:Ka}
K^a_z (t;x,y) = \phi_1(x)\cdot K^1(t;x,y)\cdot \psi_1(y) + \phi_2(x)\cdot K^2_z(t;x,y)\cdot \psi_2(y).
\end{equation}
The functions $\psi_1$ and $\psi_2$ here form a smooth partition of unity on $X_R$ such that $\supp\,(\psi_2) = [-R-4/7,R+4/7] \times Y$ and $\psi_2 = 1$ on $[-R-3/7,R+3/7] \times Y$. The function $\phi_1$ equals zero on $[-R-1/7,R+1/7] \times Y$ and one outside of $[-R-2/7,R+2/7]\times Y$, and the function $\phi_2$ equals one on $[-R-5/7,R+5/7] \times Y$ and zero outside of $[-R-6/7,R+6/7]\times Y$. Note that $\phi_j = 1$ on $\supp\,(\psi_j)$ and that the distance between $\supp\,(\p\phi_j/\p u)$ and $\supp\,(\psi_j)$ is greater than or equal to $1/7$, $j = 1, 2$.

\begin{rmk}
It is important to note that in \eqref{E:Ka} we did not twist the Dirac operators on $X$ because $df$ is supported away from $W$ in the manifold $X_R$.
\end{rmk}

The advantage of having the approximate smoothing kernel $K^a_z (t;x,x)$ is that it is defined on the same manifold $X_R$ as $K_z (t;x,x)$ while
\[
K_z (t;x,x) - K^2_z (t;x,x) = K_z (t;x,x) - K^a_z (t;x,x)
\]
in the region $[-1,1] \times Y$ of our interest. To calculate the latter difference, consider the error term
\[
-E_z(t;x,y) = \left(\frac {\p}{\p t} + \D^-_z \D^+_z\right) K^a_z (t;x,y),
\]
where the operator $\D^-_z \D^+_z$ acts on the variable $x$ for any fixed $t$ and $y$. Since
\[
(\D^-_z\D^+_z) (\phi_j K^j_z) = (\Delta \phi_j) K^j_z - 2\,\nabla_{\nabla\phi_j} K^j_z + \phi_j\, (\D^-_z\D^+_z) K^j_z
\]
and both $K^j_z$ solve the heat equation, we obtain
\begin{equation}\label{E:err}
\begin{split}
-E_z (t;x,y)\, \hspace{3.5in} \\ = \Delta \phi_1(x)\cdot K^1 (t;x,y)\cdot \psi_1(y) - 2\,\nabla_{\nabla \phi_1(x)} K^1 (t;x,y)\cdot \psi_1(y) \; \\ +\; \Delta \phi_2(x)\cdot K^2_z (t;x,y)\cdot \psi_2(y) - 2\,\nabla_{\nabla \phi_2(x)} K^2_z (t;x,y)\cdot \psi_2(y).
\end{split}
\end{equation}

\smallskip\noindent
In particular, $E_z (t;x,y) = 0$ whenever $d(x,y) < 1/7$. Following the standard argument, see for instance \cite[Section 10.4]{MRS3}, we obtain
\[
K_z (t;x,y) - K^a_z (t;x,y) = \int_0^t \int_{X_R} K_z (s;x,w)\cdot E_z (t-s;w,y)\,dw\,ds.
\]

\smallskip\noindent
The $w$--integration in this formula extends only to $\supp_w E_z(t-s;w,y)$, which is contained in $N = ([-R-6/7,-R-1/7] \times Y)\; \cup\; ([R + 1/7,R+6/7] \times Y)$. Therefore, the above integral can be written in the form
\smallskip
\begin{equation}\label{E:conv1}
K_z (t;x,y) - K^a_z (t;x,y) = \int_0^t \int_N K_z (s;x,w)\cdot E_z (t-s;w,y)\,dw\,ds.
\end{equation}

To obtain an equation on the kernel of $\D^+_z e^{-t \D^-_z \D^+_z}$ on $X_R$ similar to \eqref{E:conv1}, apply $\D^+_z$ to both sides of that equation\,:
\smallskip
\begin{equation}\label{E:conv2}
\begin{split}
\D^+_z K_z(t;x,y) - \D^+_z K^a_z (t;x,y) \hspace{2.2in} \\ = \int_0^t \int_N \D^+_z K_z (s;x,w)\cdot E_z (t-s;w,y)\,dw\,ds.
\end{split}
\end{equation}

\medskip\noindent
We will use this formula to obtain our first on-diagonal estimate on the difference between the kernels $\D^+_z K_z$ and $\D^+_z K^a_z$. The second such estimate will be coming up in Proposition \ref{P:est3}.

\begin{pro}\label{P:est1}
There are positive constants $\alpha$, $\gamma$, and $C$ independent of unitary $z$ and $R$ such that the following estimate holds for $x \in [-1,1] \times Y \subset X_R$ and $t > 0$:
\[
|\D^+_z K_z(t;x,x) - \D^+_z K^a_z (t;x,x)|\;\le\;C\cdot e^{\alpha t}\cdot e^{-\gamma R^2/t}.
\]
\end{pro}

\noindent
The proof of this proposition will use the following two lemmas which provide us with estimates on $\D^+_z K_z$ and $E_z$ which are uniform in $z$ and $R$.

\begin{lem}\label{L:est2}
There are positive constants $\alpha$, $\gamma$, and $C$ independent of unitary $z$ and $R$ such that the following estimates hold for $x, y \in X_R$ and $t > 0$:
\begin{gather}
|K_z(t;x,y)|\;\,\le\;\,C\cdot e^{\alpha t}\cdot t^{-n/2}\cdot e^{-\gamma\,d^2(x,y)/t}\quad\text{and\;} \label{E:K-est} \\ |\D_z^+ K_z(t;x,y)|\;\le\;C\cdot e^{\alpha t}\cdot t^{-(n+1)/2}\cdot e^{-\gamma\,d^2(x,y)/t}. \label{E:DK-est}
\end{gather}
\end{lem}

\begin{proof}
According to Lemma \ref{L:est1}, such estimates hold on $S^1_R \times Y$ with $\alpha = 0$ and the constants $\gamma$ and $C$ independent of $z$ and $R$. According to \cite[Proposition 1.1]{DW}, the same estimates with $\alpha = 0$ hold on the manifold $X$ but only for non-twisted Dirac operators, which are exactly the operators that $X$ contributes into the definition \eqref{E:Ka} of the approximate kernel. Now, the kernels of $\exp(-t\D^-_z \D^+_z)$ and $\D^+_z\exp(-t\D^-_z \D^+_z)$ on $X_R$ can be constructed from this approximate kernel by an iterative procedure using the Duhamel principle as in the proof of \cite[Theorem 2.4]{DW}. In the process, one obtains the estimates \eqref{E:K-est} and \eqref{E:DK-est} on $X_R$ from the respective estimates on $S^1_R \times Y$ and $X$. The constants in these estimates will be independent of $z$ and $R$ because they were already independent of $z$ and $R$ on $S^1_R\times Y$ and $X$. One also acquires in the process a possibly non-zero constant $\alpha$ which has to do with the volume of $Y$ and is therefore independent of $z$ and $R$.
\end{proof}

\begin{lem}
There exist positive constants $\gamma$ and $C$ independent of unitary $z$ and $R$ such that the following estimate holds for $x, y \in X_R$ and $t > 0$\,:
\begin{equation}\label{E:E-est}
|E_z (t;x,y)|\;\le\;C\cdot e^{-\gamma\,d^2(x,y)/t}.
\end{equation}
\end{lem}

\begin{proof}
This follows from formula \eqref{E:err} for the error term and the usual estimates on the kernels $K^1 (t;x,y)$ and $K^2_z (t;x,y)$ and their space derivatives. That the estimates for $K^2_z (t;x,y)$ are independent of $z$ and $R$ follows as in the proof of Lemma \ref{L:taubes} from an explicit formula for the heat kernel on $S^1_R \times Y$. The negative powers of $t$ that show up in the estimates are absorbed into the factor $e^{-\gamma\,d^2(x,y)/t}$ using the observation that $E_z (t;x,y) = 0$ whenever $d(x,y) < 1/7$.
\end{proof}

\begin{proof}[Proof of Proposition \ref{P:est1}]
We can now proceed with estimating the difference $\D^+_z K_z(t;x,x)\allowbreak - \D^+_z K^a_z (t;x,x)$ for $x \in [-1,1] \times Y$ using formula \eqref{E:conv2}. For all $z \in N$ we have $d(x,z)\, \ge\, (R - 1) + 1/7\, \ge\, R/2 + 1/7$ hence $d^2(x,z)\; \ge\; R^2/4 + 1/49$ and \eqref{E:DK-est} gives
\begin{equation}\notag
\begin{split}
|\D^+_z K_z(s;x,z)|\;\le\;C_1\cdot e^{\alpha_1 s}\cdot s^{-(n+1)/2}\cdot e^{-\gamma_1\,d^2(x,z)/s} \hspace{1.2in} \\
\;\le\;C_1\cdot e^{\alpha_1 s}\cdot s^{-(n+1)/2}\cdot e^{-1/49s}\cdot e^{-\gamma_1 R^2/4s}
\;\le\;C_2\cdot e^{\alpha_1 s}\cdot e^{-\gamma_1 R^2/4s}
\end{split}
\end{equation}
Similarly, using \eqref{E:E-est}, we obtain $|E_z (t-s;z,x)|\;\le\;C_3\cdot e^{-\gamma_2 R^2/4(t-s)}$ hence
\begin{multline}\notag
|\D^+_z K_z(t;x,x) - \D^+_z K^a_z(t;x,x)| \\ \;\le\;\int_0^t \int_N\;C_4\cdot e^{\alpha_1 s}\cdot e^{-\gamma R^2/s}\cdot e^{-\gamma R^2/(t-s)}\,dz\,ds \\
\;\le\;\int_0^t \int_N\;C_4\cdot e^{\alpha_1 t}\cdot e^{-\gamma R^2/t}\,dz\,ds\;\le\;C\cdot e^{\alpha t}\cdot e^{-\gamma R^2/t},
\end{multline}

\medskip\noindent
where we used the obvious fact that $1/t\,\le\,1/s + 1/(t-s)$ for $s \in (0,t)$.
\end{proof}

%%%%%%%%%%%%%%%%%%%%%%%%%%%%%%%%%%%%%%%%%%%%%%

\subsubsection{Large time estimates}
The second estimate that goes into the proof of our theorem has to do with the smallest eigenvalue of the operator $\D$ on $Y$. Such estimates are well-known for the $L^2$--norms of heat kernels; the following proposition claims pointwise estimates.

\begin{pro}\label{P:est3}
There exist positive constants $\mu$ and $C$ independent of $R$ such that
\[
|\D^+_z K_z (t;x,y)| \le C\cdot e^{-\mu t}\quad\text{and}\quad
|\D^+_z K^a_z (t;x,y)| \le C\cdot e^{-\mu t}
\]
for all $x, y \in [-1,1] \times Y \subset X_R$ and $t \ge 8$.
\end{pro}

\begin{proof}
The estimate for $\D^+_z K^a_z (t;x,y)$ is precisely the second estimate of Lemma \ref{L:est0}. The following proof is modeled after the proof of \cite[Proposition 1.1]{DW}.

We begin with an observation about the Sobolev spaces on $X_R$ with a fixed $R$. For a unitary $z$ and a non-negative integer $s$, define the Sobolev norm
\[
\|\phi\|_{L^2_s(z)}\;=\;\|\phi\|_{L^2} + \|\D^s_z \phi\|_{L^2}.
\]
This norm is equivalent to the Sobolev norm
\[
\|\phi\|_{L^2_s}\;=\;\|\phi\|_{L^2} + \|\D^s \phi\|_{L^2}
\]
meaning that the identity operators $\Id: L^2_s (z) \to L^2_s$ and $\Id: L^2_s \to L^2_s (z)$ are bounded. The norms of these operators are continuous functions of $z$ which achieve their absolute minimum and maximum on the circle $|z| = 1$. Therefore, there exist positive constants  $m$ and $M$ independent of $z \in S^1$ such that for all $\phi$ we have
\begin{equation}\label{E:norms}
m\cdot \|\phi\|_{L^2_s}\;\le\;\|\phi\|_{L^2_s (z)}\;\le\;M\cdot \|\phi\|_{L^2_s}.
\end{equation}

With this understood, let $\D_z\,\phi_k = \lambda_k\, \phi_k$ be the spectral decomposition of the full Dirac operator $\D_z$ with $\|\phi_k\|_{L^2} = 1$ (both $\phi_k$ and $\lambda_k$ depend on $z$ but we omit this dependence from our notation). We first estimate
\[
K_z (t;x,y) = \sum_k\; e^{-t\, \lambda_k^2}\cdot \phi_k (x) \cdot \overline{\phi}_k (y).
\]
Use the Sobolev embedding theorem with $s = 2n$, see for instance \cite[Lemma 1.1.4]{G}, and inequality \eqref{E:norms} to obtain the following pointwise estimates
\[
|\phi_k (x)|\;\le\;a\cdot \|\phi_k\|_{L^2_{2n}}\;\le\;(a/m)\cdot\|\phi_k\|_{L^2_{2n} (z)} = (a/m) \cdot (1 + \lambda_k^{2n})
\]
with the constants $a$ and $m$ independent of $z$ and $k$. For any $t \ge 8$, we have
\begin{align*}
|\D_z K_z(t;x,y)| & \le\; (a/m)\cdot \sum\;|\lambda_k|\,e^{-t \lambda_k^2}\cdot (1+ \lambda_k^{2n})^2 \\
& \le\; (a/m)\cdot ((2n)!)^2\cdot \; \sum\;\; e^{-(t-3)\,\lambda_k^2} \\
& \le\; (a/m)\cdot ((2n)!)^2\cdot \; \sum\;\; e^{-(t-4)\,\lambda^2}\cdot e^{-\lambda_k^2}\;\le\;C_1\cdot e^{-\lambda^2 t/2}
\end{align*}
with some positive constants $\mu$ and $C_1$ independent of $z$. In the last line, we used the condition that $\ker \D_z = 0$ on the circle $|z| = 1$ to guarantee that the smallest eigenvalue of the family $\D^2_z$ over $z \in S^1$ is positive; we call this eigenvalue $\lambda^2$. We also used the fact that
\begin{equation}\label{E:trace}
\sum\; |\lambda_k|\,e^{-\lambda_k^2}
\end{equation}
is a continuous function of $z \in S^1$ to estimate it from above by a constant independent of $z$.
The constants in the estimates depend on $R$ in two different ways. One is via the smallest eigenvalue $\lambda$ which by Proposition \ref{eigenvalue estimate_z} is bounded away from zero by a positive constant for all sufficiently large $R$. The other is via the function \eqref{E:trace} which can be shown to be bounded for $t \ge 8$ using \cite[Lemma 10.13]{MRS3} and Proposition \ref{eigenvalue estimate_z}.
\end{proof}

%%%%%%%%%%%%%%%%%%%%%%%%%%%%%%%%%%%%%%%%%%%%%%

\subsection{Proof of Theorem \ref{T:eta}}
For our choice of $df$, the difference between $\eta (X_R)$ and $\eta (S^1_R \times Y)$ is given by integrating the quantity
\[
\int_{[-1,1]\times Y}\; \Tr\, (df\cdot (\D^+_z K_z (t;x,x) - \D^+_z K^a_z(t;x,x)))\,dx,
\]
with respect to $t$ and $z$ as in formula \eqref{E:eta} (in the above formula, $\Tr$ stands for the matrix trace). The quantity $\D^+_z K_z (t;x,x) - \D^+_z K^a_z(t;x,x))$ has been estimated twice, first by $C\cdot e^{\alpha t}\cdot e^{-\gamma R^2/t}$ for $t > 0$ in Proposition \ref{P:est1} and then by $C\cdot e^{-\mu t}$ for $t \ge 8$ in Proposition \ref{P:est3}. The graphs of $e^{\,\alpha t-\gamma R^2/t}$ and $e^{-\mu t}$ intersect at the point $t = \beta R$, where
\smallskip
\[
\beta\, =\, \sqrt{\,\frac {\gamma}{\alpha + \mu}}.
\]

\smallskip\noindent
Therefore, for sufficiently large $R$ and after adjusting the constants $C$, we can use the first estimate on the interval $0 < t \le \beta R$ and the second on the interval $\beta R \le t < \infty$. Since both estimates are independent of $x$ and $z$, we can integrate them with respect to these two variables. The integration results in estimating the distance between $\eta (X_R)$ and $\eta (S^1_R \times Y)$ by a uniform constant times the integral

\[
\int_0^{\beta R} e^{\,\alpha t - \gamma R^2/t}\,dt\; +\; \int_{\beta R}^{\infty}\;e^{-\mu t}\,dt.
\]

\medskip\noindent
The latter integral can easily be estimated from above by
\[
\beta\cdot R\cdot e^{(\,\alpha\beta - \gamma/\beta) R}\; +\; (1/\mu)\cdot e^{-\mu\beta R}.
\]
Since $\alpha\beta - \gamma/\beta$ is negative, the difference between $\eta (X_R)$ and $\eta (S^1_R \times Y)$ must approach zero as $R \to \infty$. This completes the proof of Theorem \ref{T:eta}.

%%%%%%%%%%%%%%%%%%%%%%%%%%%%%%%%%%%%%%%%%%%%

\end{document}